%% file: preprint-main.tex
\documentclass{article}

\usepackage[twoside,
paperwidth=210mm,
paperheight=297mm,
textheight=622pt,
textwidth=468pt,
centering,
headheight=50pt,
headsep=12pt,
footskip=18pt,
footnotesep=24pt plus 2pt minus 12pt,
columnsep=2pc]{geometry}

\usepackage[auth-sc]{authblk}

\input{latex-commands.tex}
\usepackage{amsthm}

\usepackage[capitalise]{cleveref}

\newtheorem{theorem}{Theorem}[section]

\newtheorem{proposition}[theorem]{Proposition}

\theoremstyle{definition}
\newtheorem{definition}{Definition}

\crefname{assumption}{assumption}{assumptions}
\Crefname{assumption}{Assumption}{Assumptions}

\crefname{problem}{problem}{problems}
\Crefname{problem}{Problem}{Problems}
\crefname{equation}{}{}
\Crefname{equation}{}{}

\theoremstyle{remark}

\usepackage{lineno}

\title{MatExPre: A matrix exponential preconditioner for the high-frequency Helmholtz equation}
\author[1]{Shubin Fu}
\author[1]{Qing Huo Liu}
\author[2]{Qiwei Zhan}
\author[3]{Eric T. Chung}
\author[3]{Changqing Ye\thanks{\href{mailto:changqingye@cuhk.edu.hk}{changqingye@cuhk.edu.hk}}}
\affil[1]{Eastern Institute of Technology, Ningbo, Zhejiang, 315200, China.}
\affil[2]{College of Information Science and Electronic Engineering, Zhejiang University, Hangzhou, Zhejiang, China.}
\affil[3]{Department of Mathematics, The Chinese University of Hong Kong, Shatin, Hong~Kong~SAR, China.}

\begin{document}
\maketitle
\begin{abstract}
  \input{abstract.tex}

\end{abstract}

\input{main-text.tex}


\bibliographystyle{siamplain}
\input{refs-bib-path.tex}
\end{document}

%% file: latex-commands.tex
\usepackage{mathtools}

\allowdisplaybreaks

\usepackage{amssymb}
\usepackage[mathscr]{eucal}

\usepackage{mismath}

\usepackage{hyperref}

\usepackage{bm} 

\usepackage{graphicx}
\graphicspath{{figs/}}

\usepackage{tikz}
\usetikzlibrary{math}
\usetikzlibrary{calc}

\usepackage{subcaption}

\usepackage{multirow}

\usepackage{makecell}
\setcellgapes{2pt}

\usepackage{booktabs}

\usepackage{siunitx}
\usepackage{algorithm}
\usepackage{algpseudocode}

\usepackage{relsize}

\newcommand{\Complex}{\mathbb{C}}

\DeclareMathOperator{\Cond}{cond}

\DeclarePairedDelimiter{\RoundBrackets}{(}{)}
\DeclarePairedDelimiter{\CurlyBrackets}{\{}{\}}
\DeclarePairedDelimiter{\SquareBrackets}{[}{]}

\DeclarePairedDelimiter{\ceil}{\lceil}{\rceil}

\newcommand{\mathdefault}[1][]{}

\usepackage[skins]{tcolorbox}
\newtcolorbox{justabox}[2][]{%
  enhanced,
  attach boxed title to top center={yshift=-3mm,yshifttext=-1mm},
  colframe=blue!75!black,
  colbacktitle=red!80!black,
  fonttitle=\bfseries,
  title=#2,#1
}

\usepackage{wasysym}

%% file: abstract.tex
In this article, we present a new preconditioner, MatExPre, for the high-frequency Helmholtz equation by leveraging the properties of matrix exponentials.
Our approach begins by reformulating the Helmholtz equation into a Schr\"{o}dinger-like equation and constructing a time-domain solver based on a fixed-point iteration.
We then establish a rigorous connection between the time-domain solver and matrix exponential integrators, which enables us to derive algebraic preconditioners that rely solely on sparse matrix-vector products.
Spectral analysis and a detailed numerical implementation strategy, including performance improvements achieved through complex shifting, are discussed.
Finally, numerical experiments on 2D and large-scale 3D homogeneous and inhomogeneous models, including benchmark seismic examples, substantiate the effectiveness and scalability of the proposed methods.

%% file: main-text.tex
\section{Introduction}
\label{sec:introduction}
Waves serve as ubiquitous carriers of energy and information, thereby enabling the probing internal or distant structures that would otherwise remain inaccessible, such as the human body, the Earth's mantle, or even the universe.
In many applications, the temporal dimension can be neglected, resulting in a time-harmonic regime.
In this article, we consider the Helmholtz equation---which represents the time-harmonic formulation of the scalar wave equation---as follows:
\begin{equation}
  \label{eq:helmholtz}
  \frac{\omega^2}{c^2}u + \Delta u = -f,
\end{equation}
where $\omega$ denotes the constant angular frequency,  $c$ represents the spatially varying wave speed, $\Delta$ is the Laplace operator, and $f$ is the harmonic source term.

To obtain a numerical solution to the Helmholtz equation that merits both accuracy and efficiency, we should devise a thorough treatment on formulations, discretizations, and solvers.
In practice, numerical schemes are based on formulating the problem on a truncated bounded domain, where the radiation condition---for unbounded domain---is mimicked by applying artificial boundary conditions (e.g., absorbing boundary conditions \cite{Engquist1979}), artificial media (e.g., perfectly matched layers (PML) \cite{Berenger1994,Chew1996}), or artificial potentials (e.g., complex absorbing potentials \cite{Manolopoulos2002}).
General numerical discretizations of the Helmholtz equation can be plagued by the so-called ``pollution effect'' \cite{Babuska1997,Ihlenburg1997,Melenk2011}, which refers to the phenomenon where the accuracy of the numerical solution deviates from the expected approximation rate as the frequency increases.
Significant endeavors have been devoted to understanding this phenomenon by establishing wavenumber explicit numerical error estimates, see, on boundary element methods \cite{Galkowski2023} and on discontinuous Galerkin methods \cite{Feng2009}.

In order to resolve the wave structure accurately, the number of degrees of freedom (DoFs) must scale with at least $\bigO(\omega^d)$, where $d$ denotes the spatial dimension.
At high frequencies, this scaling yields an enormous linear system for which direct solvers are impractical, leaving iterative solvers equipped with preconditioners as the only viable option.
Studying preconditioners for Helmholtz problems is an active area of research, with a plethora of methods proposed in the literature.
We here introduce a few representative examples.
Conventional Schwarz (also referred to as domain decomposition) methods with Dirichlet transmission conditions cannot be directly applied to the Helmholtz equation.
Significant improvements have been achieved by employing Robin transmission conditions \cite{Despres1991}, and further enhancements are obtained by utilizing optimized transmission conditions derived from the spectral analysis of the Helmholtz operator \cite{Gander2002,Gander2022}.
Sweeping preconditioners, which are based on the idea of sequentially sweeping through the domain layer by layer and compressing the intermediate dense matrices, have been proposed in \cite{Engquist2011} and exhibit robustness for high-frequency problems.
Despite further developments \cite{Poulson2013}, the intrinsic sequential nature and heavy setup cost of sweeping preconditioners limit their scalability and portability in parallel computing.
A recent review \cite{Gander2019} provides a unified perspective on optimized Schwarz methods, sweeping preconditioners, and other related methods \cite{Chen2013}.
Another influential approach is the complex shifted Laplace preconditioner \cite{Erlangga2004,Erlangga2007}.
The fundamental idea is to introduce a lower-order complex shift term to the Helmholtz operator to form the preconditioner system, which 
can enhance coercivity.
The complex shifted operator is much easier to handle, as multigrid methods---that perform poorly on the original Helmholtz operator---can be effectively applied to the shifted operator \cite{Erlangga2006}.
The judicious choice of the complex shift is crucial to the performance of the preconditioner \cite{Gander2015}, and modifications to the classic multigrid components have been demonstrated to be effective \cite{Calandra2012}.

Our approach draws inspiration from several recently developed methods.
WaveHoltz \cite{Appeloe2020} was proposed by Appel\"{o}, Garcia, and Runborg to solve the Helmholtz equation by considering its physical essence as a wave propagation problem.
The key idea is to revert to the time-domain wave equation and utilize the limiting amplitude principle to derive a contraction map that involves a time integral over the wave dynamics, where the fixed point of this map yields the solution to the Helmholtz equation.
The primary advantage of employing wave equations lies in their exceptional scalability, particularly when combined with explicit adaptive time-stepping schemes and mass-lumping techniques.
Another notable time-domain solver for the Helmholtz equation is the controllability method \cite{Auchmuty1987,Bristeau1998,Grote2019}; in contrast to WaveHoltz, it identifies the minimizer of an optimal control problem involving wave dynamics as the solution to the Helmholtz equation.
Focusing more on the algebraic aspects, the time-domain preconditioner was developed by Stolk in \cite{Stolk2021}.
Recently, Luo and Liu devised a novel solver by reformulating the Helmholtz equation into a fixed-point problem related to an exponential operator \cite{Luo2022}.
Their approach employs complex absorbing potentials to simulate the Sommerfeld radiation condition, thereby enabling operator-splitting techniques to accelerate the evaluation of the exponential operator.

In this article, we propose the ``MatExPre'' methods, which employ the matrix exponential as a key ingredient in constructing a preconditioner for the Helmholtz equation.
Compared to the aforementioned methods, the novelty of our approach lies in the following aspects:
\begin{itemize}
  \item Instead of wave equations, we derive a time-domain solver by recasting the Helmholtz equation into a Schr\"{o}dinger-like equation.
        The Schr\"{o}dinger equation involves only first-order time derivatives, which permits greater flexibility in temporal discretization.

  \item We establish the connection between the time-domain solver and the matrix exponential formulation, thereby providing a more modular and algebraic framework for constructing the solver.
        As evidence, the PML method can now be incorporated into our framework, which was previously challenging to handle in \cite{Luo2022}.

  \item We introduce a complex shifted term into the preconditioner system to enhance the convergence of the iterative solver.
        This technique is feasible only within the preconditioner framework, as the preconditioner system approximates, rather than exactly replicates, the original Helmholtz system.
\end{itemize}
Ideally, the proposed method requires only the efficient implementation of sparse matrix-vector products (SpMV), while no matrix factorizations or heavy sequential routines are needed, yielding significant potential for parallel computing on new architectures.

The remainder of this article is organized as follows.
In \cref{sec:derivation}, we derive the time-domain solver for the Helmholtz equation from the Schr\"{o}dinger equation and prove the convergence of the fixed-point iteration to the solution of the Helmholtz equation.
In \cref{sec:main}, we reveal the connection between the time-domain solver and the matrix exponential formulation, perform a spectral analysis of the preconditioner, and detail its numerical implementation strategy.
Numerical experiments are presented in \cref{sec:numerical-experiments}, including a series of tests to guide the selection of parameters and performance evaluation on 3D geological models.
Finally, we conclude the paper in \cref{sec:conclusion}.

\section{From the Schr\"{o}dinger equation}
\label{sec:derivation}
The Helmholtz equation is the time-harmonic form of the acoustic wave equation.
Specifically, given a time-harmonic source term $g\exp(-\i \omega t)$, the solution $\tilde{u}(x,t)$ to the wave equation\footnote{For simplicity, the initial and boundary conditions are omitted.}
\begin{equation}
  \label{eq:wave}
  \frac{\partial^2\tilde{u}}{\partial t^2} - c^2\Delta \tilde{u} = g\exp(-\i \omega t)
\end{equation}
should formally satisfy
\[
  \tilde{u}(x, t) \exp(\i \omega t) \rightarrow u(x) \text{ as } t\rightarrow \infty,
\]
where $g(x)\coloneqq c^2(x)f(x)$ and $u$ solves the standard Helmholtz equation \cref{eq:helmholtz}.
This statement, known as the \emph{limiting amplitude principle}, has been rigorously justified in recent work by \cite{Arnold2024}, which establishes convergence rates and provides a historical overview of the topic.

\subsection{A time-domain Helmholtz solver}
To rigorously derive our method while preserving key insights without unnecessary complexity, we consider the following Helmholtz equation with a complex absorbing potential as the model problem:
\begin{equation}
  \label{eq:helmholtz-absorbing}
  \left\{
  \begin{aligned}
     & \omega^2u + \mathcal{H} u = -g &  & \text{ in } \Omega^{\#} \text{ with }\mathcal{H}u \coloneqq c^2(x)\Delta u + \i W(x) u, \\
     & u = 0                          &  & \text{ on } \partial\Omega^{\#}.
  \end{aligned}
  \right.
\end{equation}
Here, $\Omega^\#$ denotes an enlarged domain that contains the domain $\Omega$ and $W(x)$ is a non-negative smooth function vanishing in $\Omega$.
The effect of the complex potential $\i W(x)$ is absorbing the outgoing waves, which is a technique well-established in quantum scattering calculations \cite{Manolopoulos2002}.
We can observe that the frequency-domain formulation \cref{eq:helmholtz} derives from the time-domain wave equation \cref{eq:wave} through the standard substitution $\partial_t \rightarrow -\i \omega$.
Interestingly, instead of the correspondence $\omega^2 u \sim -\partial_{tt}\tilde{u}$, we can also retain one $\omega$-factor with the time derivative, leading to $\omega^2 u \sim \i\omega\partial_{t}\tilde{u}$.
This yields an alternative time-domain form of the Helmholtz equation as follows:
\begin{equation}
  \label{eq:Sch-form}
  \left\{
  \begin{aligned}
     & \i\omega \partial_t \tilde{u} + \mathcal{H} \tilde{u} = -g\exp(-\i \omega t) &  & \forall (x, t) \in \Omega^\# \times (0, \infty),           \\
     & \tilde{u}(x, t) = 0                                                          &  & \forall (x, t) \in \partial\Omega^{\#} \times (0, \infty), \\
     & \tilde{u}(x, 0)=u_0(x)                                                       &  & \forall x \in \Omega^{\#}.                                 \\
  \end{aligned}
  \right.
\end{equation}
We can recognize that the above equation is a Schr\"{o}dinger-like equation with a non-Hermitian operator $v\mapsto \i W v$ and inhomogeneous source term $g\exp(-\i \omega t)$.

With a slight abuse of notation, we also denote $\tilde{u}(t)$ as $\tilde{u}(\cdot, t)$.
Since the operator $\mathcal{H}$ is time-independent, we can simplify \cref{eq:Sch-form} through the temporal scaling $t\mapsto \omega t$, yielding $\partial_t \tilde{u} - \i \mathcal{H}\tilde{u} = \i g \exp(-\i \omega^2 t)$.
The variation-of-constants formula \cite{Hochbruck2010} then gives the solution to the time-scaled version of \cref{eq:Sch-form} as:
\begin{equation}
  \label{eq:Sch-form-sol}
  \tilde{u}(t) = \exp(\i t \mathcal{H}) u_0 + \i \exp(-\i \omega^2 t) \int_0^t \exp(\i \omega^2 \tau) \exp(\i \tau \mathcal{H}) g \di \tau.
\end{equation}
It is natural to raise the question on whether the limiting amplitude principle holds for this time-domain formulation.

One elementary approach is to examine the spectral decomposition of the operator $\mathcal{H}$, which can provide explicit expressions to $u$ in \cref{eq:helmholtz-absorbing} and $\tilde{u}$ in \cref{eq:Sch-form-sol}.
However, a key challenge arises from the infinite-dimensional setting and the fact that $\mathcal{H}$ is not self-adjoint, making its spectral decomposition mathematically non-rigorous.
While a more sophisticated framework, such as the one presented in \cite{Arnold2024}, could address these issues, such an in-depth treatment falls outside the scope of this paper, as our focus is primarily on computational aspects.
Instead, we conduct a formal spectral analysis to derive useful insights, while deferring rigorous mathematical justification.
Let $\Phi_k$ be an eigenfunction of $\mathcal{H}$ corresponding to the eigenvalue $\lambda_k$, such that $\mathcal{H}\Phi_k = \lambda_k \Phi_k$, where $\lambda_k$ is complex-valued.
Assuming $g$ and $u_0$ admit spectral decompositions as $g = \sum_k g_k \Phi_k$ and $u_0 = \sum_k u_{0,k} \Phi_k$.
Then, the solution $u$ of \cref{eq:helmholtz-absorbing} can be expressed as
\[
  u = \sum_k u_k \Phi_k \text{ with } u_k = -\frac{g_k}{\lambda_k + \omega^2}.
\]
Similarly, from \cref{eq:Sch-form-sol}, the solution $\tilde{u}$ expands as
\begin{equation}
  \label{eq:Sch-form-sol-exp}
  \begin{aligned}
     & \tilde{u}(t) = \sum_k \tilde{u}_k(t) \Phi_k                                                                                                                                                      \\
     & \qquad \text{ with } \tilde{u}_k(t) = \exp(-\i \omega^2 t) \SquareBrackets*{\exp(\i t (\lambda_k + \omega^2)) u_{0,k} -  \frac{1 - \exp(\i t (\lambda_k + \omega^2))}{\lambda_k + \omega^2}g_k}.
  \end{aligned}
\end{equation}
We can immediately observe that as $t \rightarrow \infty$, $\tilde{u}(t)\exp(\i \omega^2 t) \rightarrow  u$ can hold once $\exp(\i t (\lambda_k+\omega^2)) \rightarrow 0$.
This condition is satisfied if the imaginary part of $\lambda_k$, denoted as $\Im(\lambda_k)$, is strictly positive.

Proving that eigenvalues satisfy specific range conditions is generally highly non-trivial.
However, due to the introduction of the complex potential, we can prove that the imaginary part of the eigenvalue $\lambda$ is non-negative, stated as the following proposition.
\begin{proposition}
  Suppose that $c \in C^\infty(\Omega^\#) \cap L^\infty(\Omega^\#)$ and is bounded below from zero.
  Let $W \in L^\infty(\Omega^\#)$ with $W \geq 0$ a.e. in $\Omega^\#$.
  If $\Phi \in H^1_0(\Omega^\#) \cap H^2(\Omega^\#)$ satisfies $\mathcal{H}\Phi=\lambda \Phi$ in the $L^2(\Omega^\#)$ sense, then $\Im(\lambda) \geq 0$.
\end{proposition}
\begin{proof}
  Taking $\Phi / c^2$ as a test function to $\mathcal{H}\Phi=\lambda \Phi$, we have
  \[
    \lambda \int_{\Omega^\#} \abs{\Phi}^2/ c^2 \di x = \int_{\Omega^\#} \overline{\Phi}\Delta \Phi + \i W/c^2 \abs{\Phi}^2 \di x= \int_{\Omega^\#} -\abs{\nabla \Phi}^2 + \i W/c^2 \abs{\Phi}^2 \di x. \\
  \]
  Therefore, it holds that
  \[
    \Im(\lambda) = \frac{\int_{\Omega^\#} W/c^2 \abs{\Phi}^2 \di x}{\int_{\Omega^\#} \abs{\Phi}^2/ c^2 \di x} \geq \frac{\int_{\Omega^\#} W \abs{\Phi}^2 \di x}{\norm{c}_{L^\infty(\Omega^\#)}^2\norm{1/c}_{L^\infty(\Omega^\#)}^2\norm{\Phi}_{L^2(\Omega^\#)}^2} \geq 0.
  \]
\end{proof}

This proposition indicates that the eigenvalues of $\mathcal{H}$ lie in the upper half-plane, which is weaker than the condition $\Im(\lambda) > 0$.
If assuming $W \geq W_{\mathup{min}}>0$---as in the complex shifted technique \cite{Erlangga2006,Erlangga2007}, we can derive an estimate on the imaginary part of the eigenvalue $\lambda$ as
\[
  \Im(\lambda) \geq \frac{W_{\mathup{min}}}{\norm{c}_{L^\infty(\Omega^\#)}^2\norm{1/c}_{L^\infty(\Omega^\#)}^2} > 0.
\]
However, such an assumption modifies the original problem.
There is another fact that can support the strict positivity of $\Im(\lambda)$.
Recall that the Sommerfeld radiation condition defines a non-local operator $\mathcal{T}$ on $\partial \Omega$, leading to a variational form of the Helmholtz equation as follows \cite{Nedelec2001}: find $u \in H^1(\Omega)$ such that
\[
  \int_\Omega \nabla \overline{v} \cdot \nabla u - \omega^2/c^2 \overline{v}u \di x - \int_{\partial \Omega} \overline{v} \mathcal{T}u \di s = \int_\Omega \overline{v} f \di x, \quad \forall v \in H^1(\Omega).
\]
For a spherical domain, Theorem 2.6.4 in \cite{Nedelec2001} states that $\Im \int_{\partial \Omega} \overline{v} \mathcal{T}v \di s \geq 0$.
Furthermore, from the spherical harmonics expansion, $\Im \int_{\partial \Omega} \overline{v} \mathcal{T}v \di s=0$ if and only $v$ vanishes on the exterior domain to $\Omega$.
By the analytical continuation, this implies that $v=0$, which essentially that $\Im(\lambda)$ can only be positive (see Theorem 2.6.5 in \cite{Nedelec2001}).
From the complex absorbing potential and the non-local operator $\mathcal{T}$, we can conclude that non-negative imaginary parts of eigenvalues seem to be a universal property for characterizing wave propagation in bounded domains.
We introduce notation that $\lambda^*=\min_k \Im(\lambda_k)$ and \emph{conjecture} that the constant $\lambda^*>0$.

We can follow WaveHoltz approaches \cite{Appeloe2020} to derive a time-domain solver for the Helmholtz equation.
The fundamental idea is identifying \emph{fixed-point} solution to the Schr\"{o}dinger propagation is essentially the solution to the Helmholtz equation.
More precisely, for any $v$, we can define the operator $\mathcal{S}_{T,g}v\coloneqq \exp(\i \omega^2 T) \tilde{v}(T)$, where $\tilde{v}(t)$ solves the abstract ordinary differential equation:
\begin{equation}
  \label{eq:Sch-prop}
  \partial_t \tilde{v} - \i \mathcal{H}\tilde{v} = \i g \exp(-\i \omega^2 t) \quad \forall t \in (0, T) \text{ with } \tilde{v}(0)=v.
\end{equation}
We can summarize the above discussion as the following theorem.
\begin{theorem} \label{thm:fixed-point-ver-1}
  Suppose that the source term $g$ admits a finite spectral decomposition as $g=\sum_{k}g_k \Phi_k$ and $\lambda^* > 0$.
  Then, starting from zero, the fixed-point iteration $v\mapsto \mathcal{S}_{T,g}v$ convergences, and the limit solves the Helmholtz equation \cref{eq:helmholtz-absorbing}.
\end{theorem}
\begin{proof}
  Under the assumption that $g=\sum_{k}g_k \Phi_k$, where the summation is finite, we can reduce the problem to a finite-dimensional setting.
  According to the formula \cref{eq:Sch-form-sol-exp}, for each $k$, we have
  \[
    \abs{(\mathcal{S}_{T,g}v')_k - (\mathcal{S}_{T,g}v'')_k} \leq \abs{\exp\RoundBrackets*{\i T (\lambda_k + \omega^2)}} \abs{v'_k - v''_k} \leq \exp(-\lambda^*T) \abs{v'_k - v''_k}.
  \]
  Thus, on each eigenmode, the fixed-point iteration converges exponentially.
  Moreover, we can explicitly calculate that $v_k=-g_k/(\lambda_k+\omega^2)$ if $\mathcal{S}_{T,g}v=v$.
\end{proof}
From the proof, we can see that the value of $\lambda^* T$ determines the convergence rate of the fixed-point iteration.

\subsection{An alternative time-domain solver}
Inspired by the WaveHoltz approach, we propose an alternative time-domain solver that employs a time integral over the Schr\"{o}dinger propagation to filter out undesired frequency components.
Specifically, we consider an operator
\[
  \mathcal{G}_{T,g}v\coloneqq 1/T\int_{0}^{T}\exp(\i \omega^2 t)\tilde{v}(t) \di t,
\]
where $\tilde{v}$ evolves according to \cref{eq:Sch-prop}.
We can establish the following theorem, demonstrating that the fixed point of the operator $\mathcal{G}_{T,g}$ is again the solution to the Helmholtz equation \cref{eq:helmholtz-absorbing}.
\begin{theorem} \label{thm:fixed-point-ver-2}
  Suppose that the source term $g$ admits a finite spectral decomposition as $g=\sum_{k}g_k \Phi_k$, $\lambda^*\geq 0$ and $\max_k\abs{\lambda_k+\omega^2} T > 2$.
  Then, starting from zero, the fixed-point iteration $v\mapsto \mathcal{G}_{T,g}v$ convergences, and the limit solves the Helmholtz equation \cref{eq:helmholtz-absorbing}.
\end{theorem}
\begin{proof}
  Again, we work on each eigenmode.
  Utilizing the formula \cref{eq:Sch-form-sol-exp}, we can derive that
  \begin{align*}
    \abs{(\mathcal{G}_{T,g}v')_k - (\mathcal{G}_{T,g}v'')_k} & \leq \abs{\frac{1}{T}\int_{0}^T \exp(\i t(\lambda_k+\omega^2)) \di t} \abs{v'_k-v''_k}                  \\
                                                             & \leq \frac{\abs{\exp(\i T (\lambda_k+\omega^2))}+1}{\max_k\abs{\lambda_k+\omega^2} T} \abs{v'_k-v''_k}.
  \end{align*}
  Clearly, the fixed-point iteration converges if $\max_k\abs{\lambda_k+\omega^2} T > 2$, and we can further examine that the fixed point is the solution to the Helmholtz equation.
\end{proof}
For this time-domain solver, the requirement on $\lambda^*$ is weaker than \cref{thm:fixed-point-ver-1}, while now the convergence rate now depends on $2/\RoundBrackets*{\max_k\abs{\lambda_k+\omega^2}T}$.
It is worth noting that the quantity $\max_k\abs{\lambda_k+\omega^2}$ relates to the well-posedness of the Helmholtz equation.
To see this, consider the eigenvalue problem $(\omega^2\mathcal{I}+\mathcal{H})\Phi=(\omega^2+\lambda)\Phi$, which yields
\[
  \norm{\Phi}_{L^2(\Omega^\#)} \leq \norm{\Phi}_{L^2(\Omega^\#)} \abs{\omega^2+\lambda} \norm{(\omega^2\mathcal{I}+\mathcal{H})^{-1}}_{L^2\rightarrow L^2}
\]
where $\mathcal{I}$ is the identity operator and $\norm{\cdot}_{L^2\rightarrow L^2}$ denotes the operator norm from $L^2(\Omega^\#)$ to itself.
Drawing on wavenumber-explicit regularity estimates \cite{Cummings2006,Lafontaine2022}, we can postulate that $\norm{(\omega^2\mathcal{I}+\mathcal{H})^{-1}}_{L^2\rightarrow L^2} \approx \bigO(\omega^{-1})$.
Consequently, the time horizon $T$ can be chosen as $\bigO(\omega^{-1})$ to  guarantee fixed-point iteration convergence.

\section{A matrix exponential preconditioner}
\label{sec:main}
As discussed earlier, implementing Helmholtz solvers based on the Schr\"{o}dinger equation requires numerically solving \cref{eq:Sch-prop}, typically using time discretization schemes like Euler, Crank--Nicolson, or Runge--Kutta methods.
While feasible, these approaches also introduce deterioration from temporal discretization errors.
Furthermore, for Helmholtz problems, perfectly matched layer (PML) methods usually outperform significantly the aforementioned complex absorbing potential method \cite{Berenger1994,Chew1996}.
It leaves unclear how to construct the Schr\"{o}dinger propagation operator in the PML case, as PMLs modify the Laplace operator into a non-self-adjoint form.


\subsection{Representations of matrix inverses via exponential integrators}
If we examine \cref{eq:Sch-form-sol} closely and group the integrand---
\[
  \exp(\i\omega^2\tau)\exp(\i\tau\mathcal{H}) \Rightarrow \exp(\i \tau(\omega^2\mathcal{I}+\mathcal{H})),
\]
where $\omega^2\mathcal{I}+\mathcal{H}$ is exactly  the Helmholtz operator we need to invert.
This observation motivates us to directly solve the linear system arising from the discretized Helmholtz equation.
Let $A$ be a matrix and $I$ be the identity matrix, the operator exponential form in \cref{eq:Sch-form-sol} further suggests considering the integral of the matrix exponential,
\[
  \int_{0}^t\exp(\i \tau A)\di \tau = \i A^{-1}\RoundBrackets*{I-\exp(\i t A)}
\]
which leads to an identity for $A^{-1}$ as
\begin{equation}
  \label{eq:magic-inverse}
  A^{-1} = -\i \RoundBrackets*{I-\exp(\i t A)}^{-1} \int_{0}^t\exp(\i \tau A)\di \tau,
\end{equation}
Then, if we assume that $A$ is diagonalizable, with all eigenvalues having positive imaginary parts, the inversion $\RoundBrackets*{I-\exp(\i t A)}^{-1}$ can be approximated by fixed point iterations as $\RoundBrackets*{I-\exp(\i t A)}^{-1} = \sum_{n=0}^{\infty} \exp(\i nt A)$.
Here, the matrix $-\i A$ is also called positive semistable in literature \cite{Benzi2005}.
We can utilize \emph{quadrature} rules to calculate the integral $\int_0^t \exp(\i \tau A) \di \tau$ as in \cite{Luo2022}, and the identity \cref{eq:magic-inverse} implies that $A^{-1}$ can be well approximated if the matrix exponential action $v\mapsto \exp(\i s A)v$ is efficiently computable \cite{Eiermann2006,AlMohy2011}.
A more elegant approach to handling the time integral involves using exponential integrators \cite{Hochbruck2010,AlMohy2011}.
We first introduce the following family of \emph{entire functions}:
\begin{definition}\label{def:psi}
  Let $l$ be a non-negative integer, the entire function $\psi_l:\Complex \rightarrow \Complex$ is defined by the recurrence relation:
  \[
    \psi_{l+1}(z) = \frac{\psi_l(z) - 1/l!}{z} \text{ with } \psi_0(z) = \exp(z), \quad \forall z\in\Complex.
  \]
\end{definition}
A direct calculation shows that the time integral evaluates to $\int_0^t \exp(\i \tau A) \di \tau=t\psi_1(\i t A)$.
This allows us to express the matrix inverse identity \cref{eq:magic-inverse} compactly in terms of $\psi$-functions:
\begin{equation}
  \label{eq:inverse-1}
  A^{-1} = -\i t \RoundBrackets*{\sum_{n=0}^{\infty} \psi_0(\i t A)^n}\psi_1(\i t A).
\end{equation}
Building on our time-domain solver with operator $\mathcal{G}_{T,g}$, we derive its matrix counterpart through double integration of $\exp(\i \tau A)$:
\begin{align*}
  \frac{1}{h}\int_{0}^{h} \int_{0}^{t}\exp(\i \tau A) \di \tau \di t & = \i A^{-1}\RoundBrackets*{I-\frac{1}{h}\int_{0}^{h}\exp(\i t A) \di t} \\
                                                                     & = \i A^{-1}\RoundBrackets*{I-\psi_1(\i h A)}.
\end{align*}
Equivalently, the left-hand side simplifies to
\[
  \frac{1}{h}\int_{0}^{h} \int_{0}^{t}\exp(\i \tau A) \di \tau \di t=h \psi_2(\i h A),
\]
giving another matrix inverse formula:
\[
  A^{-1} = -\i h \RoundBrackets*{I- \psi_1(\i h A)}^{-1}\psi_2(\i h A).
\]
However, unlike \cref{eq:inverse-1}, it remains unclear whether the matrix $\RoundBrackets*{I- \psi_1(\i h A)}^{-1}$ can be approximated by fixed-point iterations under the sole condition $\Im(\lambda^*)>0$.
This distinction is mathematically substantiated in \cref{thm:fixed-point-ver-2}, which requires an additional assumption $\max_k\abs{\lambda_k+\omega^2} T > 2$.
Consequently, we primarily employ the first inverse formula \cref{eq:magic-inverse} in subsequent analysis.

\subsection{The design of preconditioners}
For any matrix $M$, both the matrix inverse $M^{-1}$ and the matrix exponential $\exp(M)$ can be interpreted as matrix functions corresponding to the scalar functions $1/z$ and $\exp(z)$, respectively.
The key distinction, as highlighted in \cref{def:psi}, lies in their analytic properties: while the $\psi_l$ functions are entire and well-defined for all $z \in \mathbb{C}$, the function $1/z$ has a singularity at $z=0$.
As a result, the computational performance, particularly for large matrices, exhibits significant differences.
For example, while our understanding of restarted Krylov methods for solving linear systems remains quite limited, these methods have been shown to converge superlinearly when applied to entire matrix functions (see Corollary 4.3 in \cite{Eiermann2006}).
Furthermore, based on the power series expansion $\exp(M) = \sum_{n=0}^\infty M^n / n!$, the convergence rate for computing entire matrix functions can be effectively predicted by the norm of $M$, rather than the condition number of $M$ as in iterative linear solvers (cf.\ \cite{AlMohy2011,Carson2024}).
Consequently, employing $\psi_0(\i t A)$ and $\psi_1(\i t A)$ as foundational components for constructing preconditioners can be practical and efficient.

If we truncate the summation in \cref{eq:inverse-1} to $N+1$ terms, denoted by
\[
  P_N \coloneqq -\i t \RoundBrackets*{\sum_{n=0}^{N} \psi_0(\i t A)^n}\psi_1(\i t A),
\]
we obtain an approximation of $A^{-1}$.
Since $P_N$ is a matrix function that depends on $A$, it shares the same eigenvectors as $A$ and commutes with it.
For the eigenvalue $\lambda_k$ of $A$, the corresponding eigenvalue of $P_N A$ is given by
\[
  \RoundBrackets*{\sum_{n=0}^N \exp(\i n t \lambda_k)}\RoundBrackets*{1-\exp(\i t \lambda_k)} = 1 - \exp(\i (N+1) t \lambda_k).
\]
The modulus of $\exp(\i (N+1) t \lambda_k)$ thus determines the quality of the approximation provided by $P_N$, which can be formalized in the following theorem.
\begin{theorem}\label{thm:radius}
  Let \(\lambda \) be any eigenvalue of the matrix \(A\) with $\Im(\lambda) \geq \lambda^* > 0$.
  Then, the eigenvalues of $P_N A$ lie within a disk centered at $1$ with radius $\exp(-(N+1)t\lambda^*)$.
\end{theorem}

\subsection{Numerical investigations on eigenvalue patterns}
\label{subsec:numerical-eigenvalue-patterns}
Beyond homogeneous Dirichlet boundary conditions, which are rare in practice, the mathematical tools for analyzing the eigenvalue distribution of Helmholtz operators remain limited.
To gain insights into the eigenvalue distribution of the Helmholtz operator, we turn to numerical evidence. Specifically, consider the following 1D equation:
\[
  \omega^2 u + \frac{1}{s}\frac{\di}{\di x}\RoundBrackets*{\frac{1}{s}\frac{\di u}{\di x}} = -f \quad \text{in } (-\delta, 1+\delta), \quad u(-\delta) = u(1+\delta) = 0,
\]
where
\begin{equation}
  \label{eq:pml-1d-s}
  s(x) = \begin{cases}
    1 + \i C_{\text{pml}} / (\delta \omega) (x / \delta)^2       & \forall x \in (-\delta, 0),  \\
    1                                                            & \forall x \in (0, 1),        \\
    1 + \i C_{\text{pml}} / (\delta \omega) ((x - 1) / \delta)^2 & \forall x \in (1, 1+\delta).
  \end{cases}
\end{equation}
Here, $\delta$ denotes the PML thickness, and $C_{\text{pml}}$ is a constant that can be adjusted to control the absorbing strength.
We employ the standard three-point finite difference (FD) discretization to obtain the matrix $A$, where $h$ is the uniform mesh size.
We also denote $\mathtt{ppw}$ as the number of points per wavelength, defined as $\mathtt{ppw} = 2\pi/\omega h$.

\begin{figure}[!ht]
  \centering
  \resizebox{\textwidth}{!}{\input{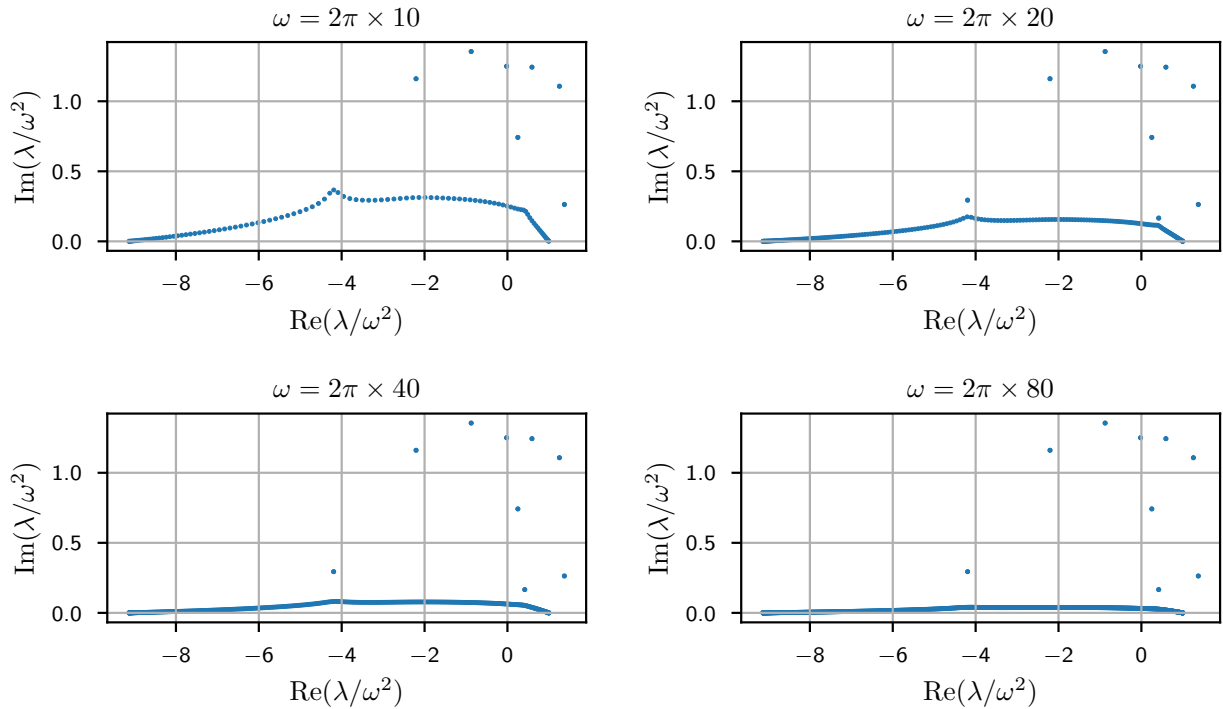}}
  \caption{Eigenvalues of the 1D Helmholtz operator with the PML method.
    The standard three-point FD method is used, with parameters set as $\mathtt{ppw} = 10$, $\delta = \mathtt{ppw} \times h$, and $C_{\text{pml}} = 20$.}\label{fig:pml-1d-eigval-dis}
\end{figure}

We begin by fixing $\mathtt{ppw} = 10$, $\delta = \mathtt{ppw} \times h$, and $C_{\text{pml}} = 20$, while varying $\omega \in \{10, 20, 40, 80\}$ to analyze the eigenvalue distribution of the matrix $A$.
These parameter values for $\mathtt{ppw}$, $\delta$, and $C_{\text{pml}}$ follow practical thumb rules commonly used in applications.
The 1D setting enables us to compute all eigenvalues in the complex plane, with the results presented in \cref{fig:pml-1d-eigval-dis}.
For improved visualization, the eigenvalues are scaled by $\omega^2$.
Two distinct patterns can be observed in the eigenvalue distribution: one accumulates and forms a curve above the real axis, while the other consists of isolated eigenvalues located farther from the real axis.
The first pattern is associated with the Helmholtz operator $\omega^2 u + u''$ defined on the inner domain, whereas the second pattern corresponds to the PML construction.
Additionally, it is evident that $\min_k \Im(\lambda_k) \geq 0$, and its quantitative relationship with respect to $\omega$ warrants further investigation through additional experiments.

\begin{figure}[!ht]
  \centering
  \resizebox{0.8\textwidth}{!}{\input{figs/pml-1d-lambda-C-width.pgf}}
  \caption{The minimal imaginary part of eigenvalues $\min_k \Im(\lambda_k)$ changes w.r.t.\ frequency $\mathtt{freq}$: (Left) fixed PML width $\delta$ and (Right) fixed PML parameter $C_{\text{pml}}$.}\label{fig:pml-1d-lambda-C-width}
\end{figure}
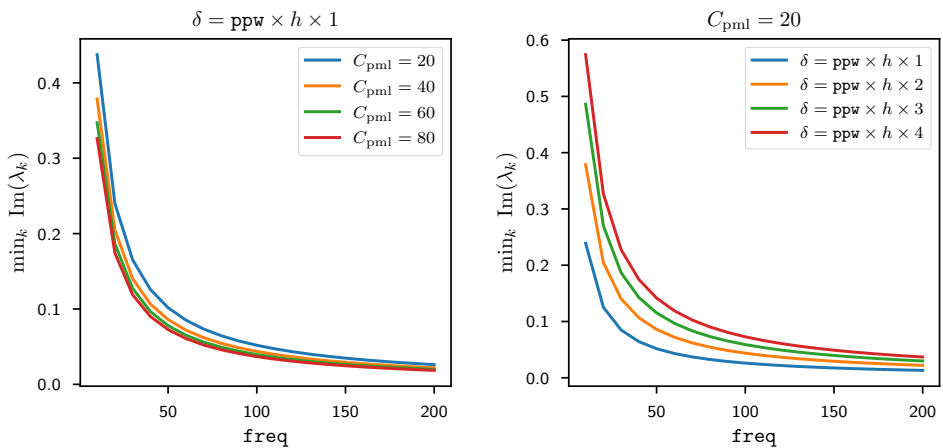

We now set $\omega = 2\pi \times \mathtt{freq}$ with $\mathtt{freq} \in \{10, 20, \dots, 200\}$ to examine the behavior of the value of $\min_k \Im(\lambda_k)$.
Two groups of experiments are conducted: the first fixes $\delta = \mathtt{ppw} \times h$ while varying $C_{\text{pml}}$, and the second fixes $C_{\text{pml}} = 20$ while varying $\delta$, with $\mathtt{ppw}$ set to $10$.
The results are presented in \cref{fig:pml-1d-lambda-C-width}.
From both subplots, we observe that $\min_k \Im(\lambda_k)$ decreases as $\omega$ increases.
However, the rate of decrease is not uniform, and the pattern strongly suggests the existence of a limiting value as $\omega \to \infty$. This aligns with the conjecture that $\lambda^* > 0$.
In the left subplot of \cref{fig:pml-1d-lambda-C-width}, we see that smaller values of $C_{\text{pml}}$ result in larger $\min_k \Im(\lambda_k)$.
Similarly, the right subplot shows that a wider absorbing layer also leads to larger values of $\min_k \Im(\lambda_k)$.

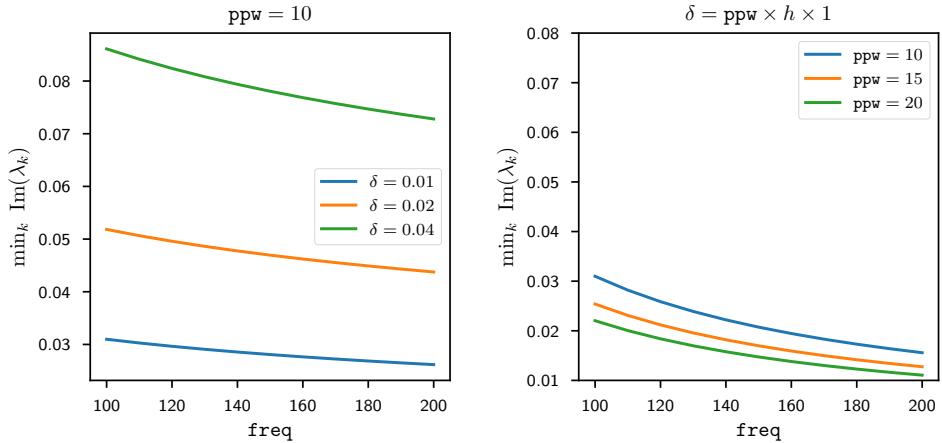
\begin{figure}[!ht]
  \centering
  \resizebox{0.8\textwidth}{!}{\input{figs/pml-1d-delta-ppw.pgf}}
  \caption{The minimal imaginary part of eigenvalues $\min_k \Im(\lambda_k)$ changes w.r.t.\ frequency $\mathtt{freq}$: (Left) varying $\delta$ and (Right) varying $\mathtt{ppw}$.}\label{fig:pml-1d-delta-ppw}
\end{figure}

In \cite{Luo2022}, the layer used to impose the complex absorbing potential has a fixed width ($\geq 0.25$) that is independent of the mesh size $h$.
We also consider this scenario by setting $\delta \in \{0.01, 0.02, 0.04\}$, while keeping $C_{\text{pml}} = 20$ and $\mathtt{ppw} = 10$.
The results are presented in the left subplot of \cref{fig:pml-1d-delta-ppw}.
It is evident that the width of the absorbing layer significantly affects $\min_k \Im(\lambda_k)$.
For larger values of $\delta$, such as $\delta \geq 0.25$, the value of $\min_k \Im(\lambda_k)$ is expected to increase substantially.
However, in this case, a considerable number of DoFs become redundant, as PML methods are highly effective at absorbing outgoing waves even with smaller widths.
We also examine the effect of $\mathtt{ppw}$ on the minimal imaginary part of the eigenvalues, as higher-frequency problems typically require more points per wavelength for accuracy \cite{Babuska1997}.
We set $\mathtt{ppw} \in \{10, 15, 20\}$ and fix $\delta$ to one wavelength.
The results, shown in the right subplot of \cref{fig:pml-1d-delta-ppw}, reveal that increasing $\mathtt{ppw}$ leads to only a slight decrease in $\min_k \Im(\lambda_k)$.

\subsection{Numerical strategies}
Several numerical strategies can be employed to enhance the performance of the preconditioner.
The fixed-point iterations for solving
\begin{equation}
  \label{eq:fixed-point}
  \RoundBrackets*{I-\exp(\i t A)} v = b
\end{equation}
in \cref{eq:magic-inverse} can be accelerated using the GMRES method \cite{Saad1986}, which, as noted in \cite{Walker2011}, is essentially equivalent to Anderson acceleration for linear fixed-point problems.
It is worth emphasizing that the convergence behavior of GMRES for non-normal matrices is often unpredictable and cannot generally be inferred from the eigenvalue distribution of the matrix \cite{Greenbaum1996}.
Elementary analysis shows that the $n$-th step residual of GMRES applied to \cref{eq:fixed-point} satisfies
\begin{equation}
  \label{eq:gmres-residual}
  \frac{\norm{r_n}}{\norm{r_0}} \leq \Cond(X) \min_{p \in \mathcal{P}_n} \max_k \abs{p(1-\exp(\i t \lambda_k))},
\end{equation}
where $\mathcal{P}_n$ is the set of polynomials of degree at most $n$, and $X$ is the matrix formed by the eigenvectors.
The convergence can be hindered by the condition number of $X$.
Nevertheless, since all eigenvalues of $I-\exp(\i t A)$ cluster around $1$, GMRES has been observed to perform well in practice \cite{Benzi2004,Erlangga2007}.

As discussed earlier, the minimal imaginary part of the eigenvalues plays a crucial role in solving \cref{eq:fixed-point}, with larger values of $\lambda^*$ leading to faster convergence.
In \cref{subsec:numerical-eigenvalue-patterns}, we numerically demonstrate that widening the absorbing layer can effectively increase $\lambda^*$, albeit at the cost of introducing more DoFs.
An alternative approach to achieve this goal is to introduce a complex shift.
Instead of using the original Helmholtz operator $A$ as the input in the preconditioner, we consider a shifted operator:
\begin{equation}
  A_s \coloneqq A + \i \omega^2 s D,
\end{equation}
where $s$ is a positive constant, and $D$ is a positive-definite (commonly diagonal) matrix corresponding to the discretization of $1/c^2$.
The shifted operator $A_s$ is then substituted into \cref{eq:magic-inverse} to construct the preconditioner for $A$.
The use of such a shifted operator as an enhancement to the original preconditioner construction is not new; see \cite{Tsuji2014} for sweeping preconditioners and \cite{Calandra2012} for multigrid preconditioners.
As highlighted in \cite{Ernst2012,Gander2015}, complex shift preconditioners face a dilemma when solving high-frequency problems, balancing the approximability of $A_s$ and the effectiveness of multigrid methods on the shifted problem.
We demonstrate through numerical experiments that the proposed method is far less restrictive with respect to the choice of $s$.


As the preconditioning process involves iterative solvers for the fixed-point problem, the outer iteration for solving the Helmholtz equation should consider the flexible GMRES (FGMRES) variant \cite{Saad1993}, which allows for the preconditioner to change at each iteration.
Meanwhile, for iterative solvers such as FGMRES, the constant factor \( -\i t \) in \cref{eq:inverse-1} does not affect the preconditioned system and can therefore be ignored.
We summarize the above discussion in the following algorithm.
\begin{algorithm}[!ht]
  \caption{The preconditioning step in one FGMRES iteration for solving the Helmholtz equation $Au = b$.}
  \label{alg:preconditioner}
  \begin{algorithmic}[1]
    \Require The matrix $A_s$, the time length $t$, the relative tolerance $\mathtt{fprtol}$, the residual $r$, the output $w$
    \State Update $r$ by $r \leftarrow \psi_{1}(\i t A_s)r$, store the residual norm  as $z$
    \While{Inner GMRES does not achieve the relative tolerance $\mathtt{fprtol}$}
    \State Perform a GMRES iteration on $w$ and update the residual for the linear system $(I-\exp(\i t A_s))w=r$
    \EndWhile
    \State \Return $w$ for the outer iteration of FGMRES
  \end{algorithmic}
\end{algorithm}

Since the evolution of the solutions via the proposed method involves essentially only SpMV operations, the algorithm exhibits an intrinsic computational complexity lower bound of $\bigO(\omega)$.
To illustrate this, suppose the right-hand side is a delta function.
Due to the sparsity pattern of the discretized matrix, the support (i.e., the diameter) of the solution vector can increase by at most $\bigO(h)$ with each update involving one SpMV operation.
However, by examining the Green function of the Helmholtz equation, it is evident that the true solution decays slowly away from the delta source.
This observation implies that, in order to achieve reasonable accuracy, the support of the solution vector must remain bounded below by $\bigO(1)$.
Consequently, the number of SpMV operations required is at least $\bigO(\omega)$, by considering the fixed $\mathtt{ppw}$ regime.
We will show numerically that the proposed method achieve this optimal complexity.

\section{Numerical experiments}
\label{sec:numerical-experiments}
In this section, we present numerical experiments that both validate our proposed preconditioner and demonstrate its effectiveness for solving the Helmholtz equation.
Our implementation is built on the PETSc library, which provides infrastructures for parallel computing.
The PETSc was configured to use double precision for all complex arithmetic, and we notice that several reported results were obtained in single precision \cite{Calandra2012}.
Exploring multi-precision arithmetic is a promising direction for future research, as it aligns with the hardware architecture of modern computing accelerators \cite{Higham2022}.
To compute the matrix functions $\psi_l(M)$, we employ the SLEPc library \cite{Hernandez2005}, which implements the restarted Krylov method proposed in \cite{Eiermann2006}.
Both PETSc and SLEPc have been compiled with the ``-O3'' optimization flag and linked with Intel MKL and MPI libraries.
All experiments were conducted on a cluster where each node is equipped with dual Intel(R) Xeon(R) Gold 6258R CPUs (each with $28$ cores) and $192$ GB of RAM.
The computing times are reported using PETSc's profiling functionality (via the ``-log\_view'' option).
The source code is available on GitHub (\href{https://github.com/Laphet/matexpre.git}{https://github.com/Laphet/matexpre.git}).

Most of the numerical settings are based on the PML method with the configuration specified in \cref{eq:pml-1d-s} and $C_{\text{pml}}=20$.
The domain of interest is taken a rectangle (in 2D) or a rectangular cuboid (in 3D), with one wavelength of absorbing layer extending beyond the domain boundaries, where zero Dirichlet boundary conditions are imposed on the outer boundary.
For discretization, we employ the standard five-point FD method in 2D and the seven-point method in 3D.
In addition, we adhere to the convention that $\mathtt{ppw}\geq 10$, and hence for heterogeneous problems the relation $c_{\mathup{min}}{2\pi}/{\omega} \geq 10 \times h$
is always satisfied.
We recall that $\mathtt{freq}$ is denoted as the frequency with the relation $\omega = 2\pi \times \mathtt{freq}$.
The iterative solvers FGMRES and GMRES used in \cref{alg:preconditioner} are provided by PETSc with default settings (e.g., a restart parameter of $30$).
The outer iteration (FGMRES) stops when the relative tolerance falls below $10^{-5}$.
For the matrix functions $\exp(\cdot)$ and $\psi_1(\cdot)$, we again employ the default settings in SLEPc.

To gain an intuitive understanding of the parameter effects---especially for geological models---we apply \emph{spatial} and \emph{temporal} scaling for most numerical experiments, such that the maximum dimension of the inner domain (excluding PMLs) is $1$ and the minimal wave speed is standardized to $1$.

\subsection{2D homogeneous models}\label{subsec:2d-homogeneous-models}
We first consider the homogeneous case, i.e., $c \equiv 1$.
The inner domain is the unit square $(0,1)^2$, and we fix $\mathtt{ppw}=10$, such that the frequency $\mathtt{freq}$ and the mesh size $h$ satisfy $\mathtt{freq}\times h = 1/10$. The source term is defined as

\[
  f(x,y) = -\operatorname{normal}(x,y;\tfrac{1}{2}-6h, \tfrac{1}{2}, 3h)+\operatorname{normal}(x,y;\tfrac{1}{2}+6h, \tfrac{1}{2}, 3h),
\]
where $\operatorname{normal}(x,y;x^*,y^*,r)$ denotes the normalized Gaussian function centered at $(x^*,y^*)$ with radius $r$.

\subsubsection{Effect of the time length in exponential integrators}
In \cref{thm:radius} it is demonstrated that the time length \(t\) enters directly into the radius of the disk containing the eigenvalues of \(P_NA\).
One therefore expects that increasing \(t\) will improve the approximability of the matrix exponential.
However, a larger \(t\) also incurs a higher iteration count when computing the action of \(\exp(\i At)\).
According to \cite{AlMohy2011}, the matrix 1-norm \(\norm{M}_1\) provides a practical estimate of the number of Krylov iterations required to approximate \(\exp(M)\).
Since \(A\) comprises the discretized Laplace operator \(\Delta\) and the time-harmonic term \(\omega^2/c^2\), one obtains $\norm{A}_{1} \lesssim h^{-2}+\omega^{2}$.
In the fixed $\mathtt{ppw}$ regime, this relation simply reduces to $\norm{A}_{1} \lesssim \omega^2$.
Therefore, we will examine the iteration counts for computing $\exp(\i A t) v$ and $\psi_1(\i A t)v$ according to the rule $t\propto 1/ \mathtt{freq}^2$.
The vector $v$ is from the first step of the FGMRES iteration in \cref{alg:preconditioner}, and no shift is involved in this case.
The convergence criteria are determined by SLEPc in its default configuration by measuring the effectiveness of the Krylov subspace.
However, directly determining the accuracy of the matrix exponential is not straightforward, as a corresponding residual concept for the matrix exponential is not available.

The iteration count results are summarized in \cref{tab:exp-iter-counts}, where each iteration only requires one SpMV operation.
Notably, the counts required to compute $\exp(\i A t)v$ and $\psi_1(\i A t)v$ are nearly identical, with the latter being marginally more efficient.
Furthermore, by setting \(t\) proportional to \(1/\mathtt{freq}^2\), the iteration counts remain stable across varying frequencies.
These findings corroborate the theoretical prediction that the computation of $\exp(M)z$ and $\psi(M)z$ necessitates a number of SpMV operations proportional to $\norm{M}_1$, and the nearly linear growth of the counts with respect to $t$ further reinforces this conclusion.

\begin{table}[!ht]
  \centering
  \caption{The iteration counts for computing the exponential integrators $\exp(\i A t)v$ and $\psi_1(\i A t)v$ for various values of \(t\) proportional to \(1/\mathtt{freq}^2\), where $A$ is from discretizing the 2D homogeneous Helmholtz equation.
    The columns labeled $\psi_0$ and $\psi_1$ correspond to $\exp(\i A t)v$ and $\psi_1(\i A t)v$, respectively.}
  \label{tab:exp-iter-counts}
  \begin{footnotesize}
    \makegapedcells
    \begin{tabular}{c c c c c c c c c}
      \toprule
      \multirow{2}{*}{$\mathtt{freq}$} & \multicolumn{2}{c}{$t=\frac{0.4}{\mathtt{freq}\times \mathtt{freq}}$} & \multicolumn{2}{c}{$t=\frac{0.8}{\mathtt{freq}\times \mathtt{freq}}$} & \multicolumn{2}{c}{$t=\frac{1.2}{\mathtt{freq}\times \mathtt{freq}}$} & \multicolumn{2}{c}{$t=\frac{1.6}{\mathtt{freq}\times \mathtt{freq}}$}                                             \\
      \cline{2-9}
                                       & $\psi_0$                                                              & $\psi_1$                                                              & $\psi_0$                                                              & $\psi_1$                                                              & $\psi_0$ & $\psi_1$ & $\psi_0$ & $\psi_1$ \\
      \midrule
      $20$                             & $7$                                                                   & $7$                                                                   & $12$                                                                  & $12$                                                                  & $17$     & $17$     & $22$     & $21$     \\
      $40$                             & $7$                                                                   & $6$                                                                   & $12$                                                                  & $12$                                                                  & $18$     & $18$     & $23$     & $23$     \\
      $80$                             & $7$                                                                   & $6$                                                                   & $12$                                                                  & $12$                                                                  & $18$     & $18$     & $24$     & $23$     \\
      $160$                            & $7$                                                                   & $7$                                                                   & $12$                                                                  & $12$                                                                  & $18$     & $18$     & $24$     & $23$     \\
      \bottomrule
    \end{tabular}
  \end{footnotesize}
\end{table}

Because evaluating matrix functions is a critical component of the preconditioner, which is frequently invoked in solving the fixed-point system \cref{eq:fixed-point}, we adopt the setting \(t=\bigO(\omega^{-2})\) such that each matrix function evaluation requires an approximately constant number of SpMV operations.
In subsequent experiments, we choose \(t\) such that the average number of SpMV operations lies between $5$ and $10$.
This strategy is designed to outperform traditional Runge-Kutta methods (e.g., RK4, which requires $4$ SpMV operations per update) for evaluating the matrix exponential.

\subsubsection{The effect of the shift parameter}
\label{subsubsec:effect-of-shift}
We present visualizations of the computed solutions for frequencies $\mathtt{freq}\in \CurlyBrackets{40,80,160,320}$.
In these experiments, the shift parameter is set to $s=0$, the relative tolerance for the fixed-point system \cref{eq:fixed-point} is fixed at $\mathtt{fprtol}=0.01$, and the time length is chosen as $t=0.4/\mathtt{freq}^2$.

In all experiments, the outer FGMRES solver converges within $4$ iterations, whereas the inner GMRES solver applied to the fixed-point system \cref{eq:fixed-point} requires an iteration count that grows with frequency. \Cref{tab:gmres-iter-counts} summarizes, for each outer FGMRES iteration, the number of inner GMRES iterations and the corresponding relative residual.
Within each outer step (especially the 1st, 2nd, and 4th iteration), the inner iteration count increases almost \emph{linearly} in $\mathtt{freq}$, suggesting that \cref{alg:preconditioner} solves high-frequency Helmholtz problems with only $\bigO(\omega)$ SpMV operations, even without any complex shift.
The growth is consistent with the spectral-radius estimate of \cref{thm:radius}, where the parameters $\lambda^*$ and $t$ govern the clustering of eigenvalues of $\exp(\i t A)$ and deteriorate as $\mathtt{freq}$ increases.
Moreover, if we assume a stable $\lambda^*$ (cf.\ \cref{fig:pml-1d-lambda-C-width}), a plain fixed-point iteration would demand $\bigO(\omega^2)$ steps to achieve equivalent accuracy; the acceleration afforded by GMRES thus reduces the work of the inner solver to $\bigO(\omega)$.

\begin{table}[!ht]
  \caption{The number of inner GMRES iterations (``iter.'' column) and the relative residual (``res.'' column) in each outer FGMRES iteration for the 2D homogeneous Helmholtz equation, where $s=0$ and $t=0.4/\mathtt{freq}^2$.}
  \label{tab:gmres-iter-counts}
  \centering
  \begin{footnotesize}
    \makegapedcells
    \begin{tabular}{c c c c c c c c c}
      \toprule
      \multirow{2}{*}{$\mathtt{freq}$} & \multicolumn{2}{c}{1st FGMRES} & \multicolumn{2}{c}{2nd FGMRES} & \multicolumn{2}{c}{3rd FGMRES} & \multicolumn{2}{c}{4th FGMRES}                                                       \\
      \cline{2-9}
                                       & iter.                          & res.                           & iter.                          & res.                           & iter. & res.             & iter. & res.             \\
      \midrule
      $40$                             & $15$                           & $\num{2.391e-2}$               & $14$                           & $\num{5.584e-4}$               & $18$  & $\num{1.570e-6}$ & ---   & ---              \\
      $80$                             & $25$                           & $\num{1.545e-1}$               & $19$                           & $\num{6.932e-4}$               & $39$  & $\num{1.333e-5}$ & $30$  & $\num{1.091e-7}$ \\
      $160$                            & $43$                           & $\num{1.910e-1}$               & $30$                           & $\num{8.101e-4}$               & $97$  & $\num{1.808e-5}$ & $58$  & $\num{1.355e-7}$ \\
      $320$                            & $108$                          & $\num{8.609e-2}$               & $76$                           & $\num{1.663e-3}$               & $95$  & $\num{7.765e-5}$ & $110$ & $\num{3.641e-7}$ \\
      \bottomrule
    \end{tabular}
  \end{footnotesize}
\end{table}

Although the proposed preconditioner exhibits linear complexity in terms of SpMV operations, the inner GMRES process remains computationally expensive due to its slow convergence.
To accelerate the inner GMRES iterations, we investigate the introduction of a complex shift.
Naturally, the outer FGMRES iteration counts will increase since the preconditioner system now essentially involves solving a shifted problem, i.e., \(A_s z = r\).
Taking \(s=\bigO(\omega^{-\alpha})\) into account, the results in \cite{Gander2015} suggest that if \(\alpha>1\), a frequency-independent outer solver may be achieved; conversely, if \(\alpha<1\), the outer solver performance deteriorates significantly.
Therefore, we focus particularly on the borderline case \(\alpha=1\) and set \(s=1/\mathtt{freq}\), \(2/\mathtt{freq}\), \(3/\mathtt{freq}\) or \(4/\mathtt{freq}\) for numerical experiments.
Through this setting on $s$ and $t$, we can see that the spectral radius of $\exp(\i t A_s)$ scales as $1-\bigO(\omega^{-1})$.
We report the number of outer FGMRES iterations and the average number of inner GMRES iterations (mean with standard deviation) in \cref{tab:gmres-iter-counts-shift}.
Our results indicate that the proposed preconditioner is robust with respect to frequency when \(s\) is chosen proportional to \(1/\omega\), as the outer FGMRES iterations show only slight variations across the different cases.
As expected, a larger shift \(s\) leads to an increased number of outer iterations, although the trend appears sublinear, which may be attributed to the acceleration effect of the outer FGMRES solver.
Compared to \cref{tab:gmres-iter-counts}, it is noteworthy that the inner GMRES iterations are significantly reduced, and the pronounced variance seen previously is mitigated, as evidenced by the consistent standard deviations across all experiments.
Furthermore, larger complex shifts result in fewer inner iterations, with the reduction exhibiting a sublinear pattern.
These observations suggest that the optimal choice of the complex shift should balance the performance of both the outer and inner iterations.

\begin{table}[!ht]
  \caption{The number of outer FGMRES iterations (column labeled ``outer'') and inner GMRES iterations (column labeled ``inner'', presented as ``mean $\pm$ standard deviation'') for the 2D homogeneous Helmholtz equation under various shift parameters.
  }
  \label{tab:gmres-iter-counts-shift}
  \centering
  \begin{footnotesize}
    \makegapedcells
    \begin{tabular}{c c c c c c c c c}
      \toprule
      \multirow{2}{*}{$\mathtt{freq}$} & \multicolumn{2}{c}{$s=1/\mathtt{freq}$} & \multicolumn{2}{c}{$s=2/\mathtt{freq}$} & \multicolumn{2}{c}{$s=3/\mathtt{freq}$} & \multicolumn{2}{c}{$s=4/\mathtt{freq}$}                                                 \\
      \cline{2-9}
                                       & outer                                   & inner                                   & outer                                   & inner                                   & outer & inner         & outer & inner         \\
      \midrule
      $40$                             & $7$                                     & $5.5\pm 1.0$                            & $10$                                    & $4.3\pm 0.6$                            & $12$  & $3.2\pm 0.4$  & $14$  & $3.0\pm 0.0$  \\
      $80$                             & $7$                                     & $10.2\pm 1.6$                           & $10$                                    & $7.4\pm 1.2$                            & $12$  & $6.1\pm 1.0$  & $14$  & $4.9\pm 0.7$  \\
      $160$                            & $7$                                     & $19.8\pm 2.5$                           & $9$                                     & $14.8\pm 2.6$                           & $11$  & $11.7\pm 2.1$ & $13$  & $9.5\pm 1.5$  \\
      $320$                            & $7$                                     & $43.5\pm 5.6$                           & $9$                                     & $29.4\pm 4.6$                           & $11$  & $22.8\pm 4.1$ & $13$  & $18.4\pm 3.0$ \\
      \bottomrule
    \end{tabular}
  \end{footnotesize}
\end{table}

\subsubsection{Remark on solution evolution}

We investigate the evolution of the solution vector during the outer FGMRES iterations.
The parameters are set as $\mathtt{freq}=40$, $t=0.4/\mathtt{freq}^2$, and $s=4/\mathtt{freq}$.
\Cref{fig:solution-2d-iteration} displays the first four iterations; for improved visualization, the magnitude is capped at $0.01$.
It is observed that the solution vector is initially concentrated around the source term and that its support gradually expands as the iterations progress.
This behavior is consistent with our theoretical understanding for the algorithm.

\begin{figure}[!ht]
  \centering
  \begin{subfigure}[b]{0.240\textwidth}
    \includegraphics[width=\textwidth]{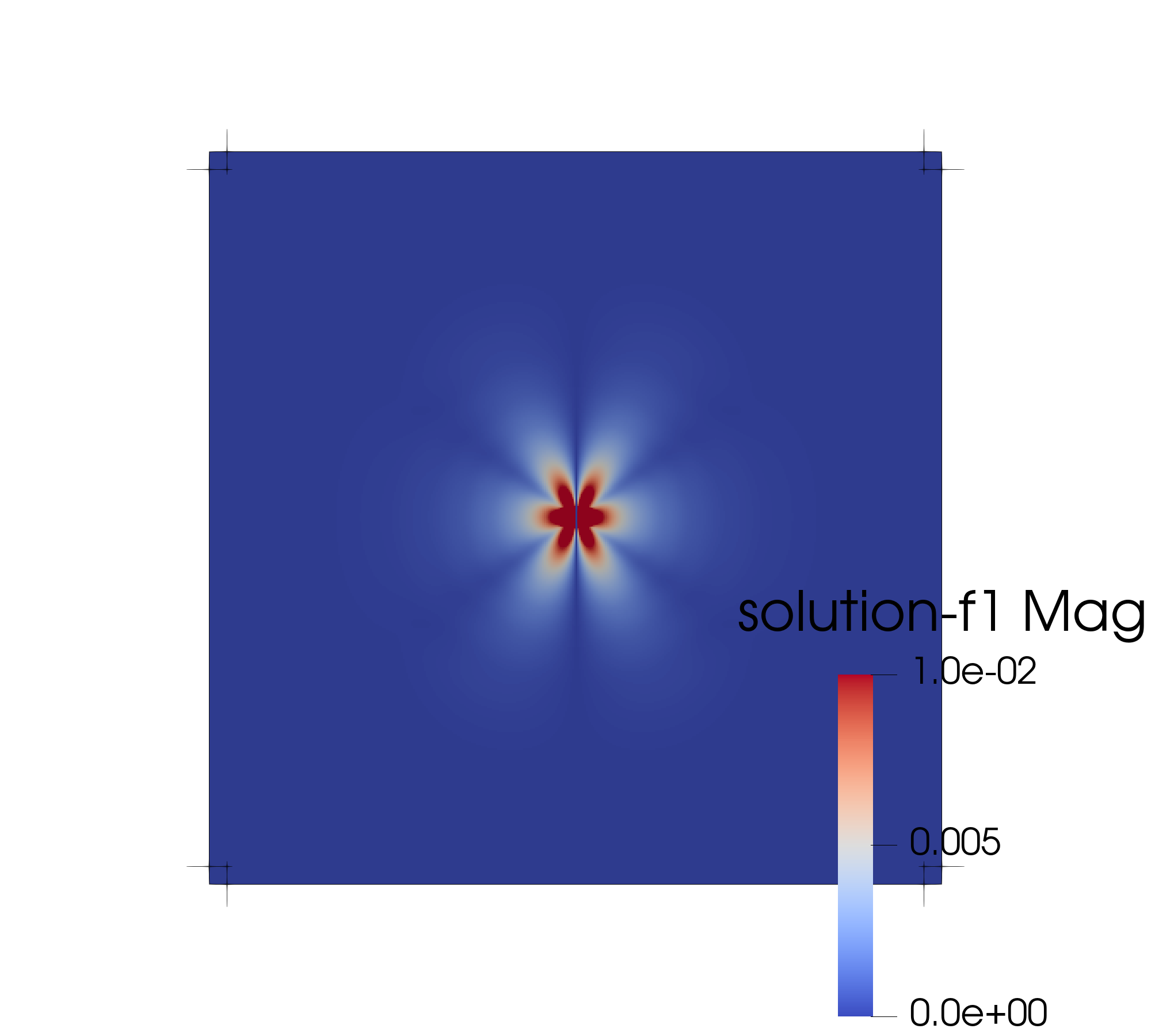}
    \caption{}
  \end{subfigure}
  \begin{subfigure}[b]{0.240\textwidth}
    \includegraphics[width=\textwidth]{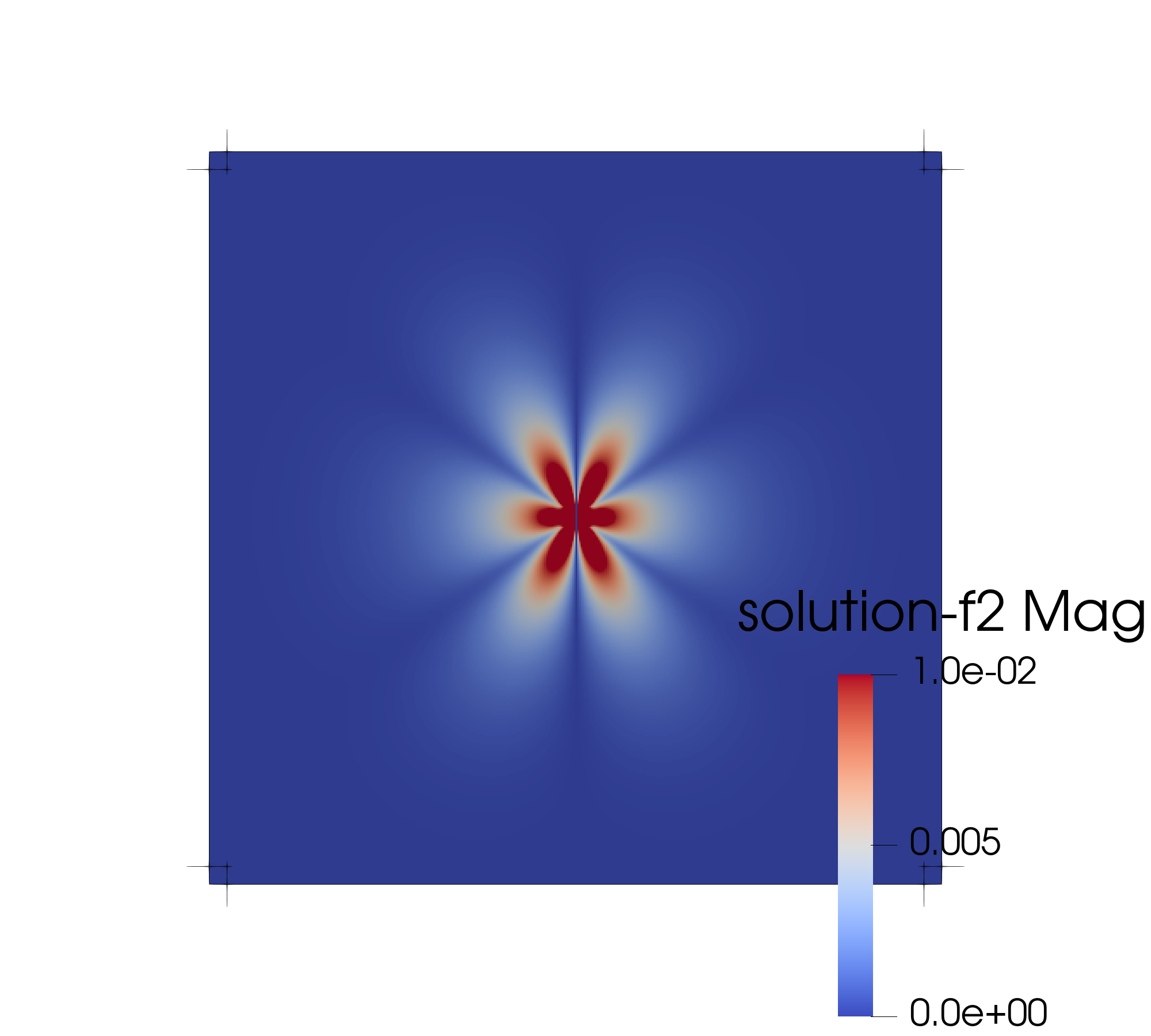}
    \caption{}
  \end{subfigure}
  \begin{subfigure}[b]{0.240\textwidth}
    \includegraphics[width=\textwidth]{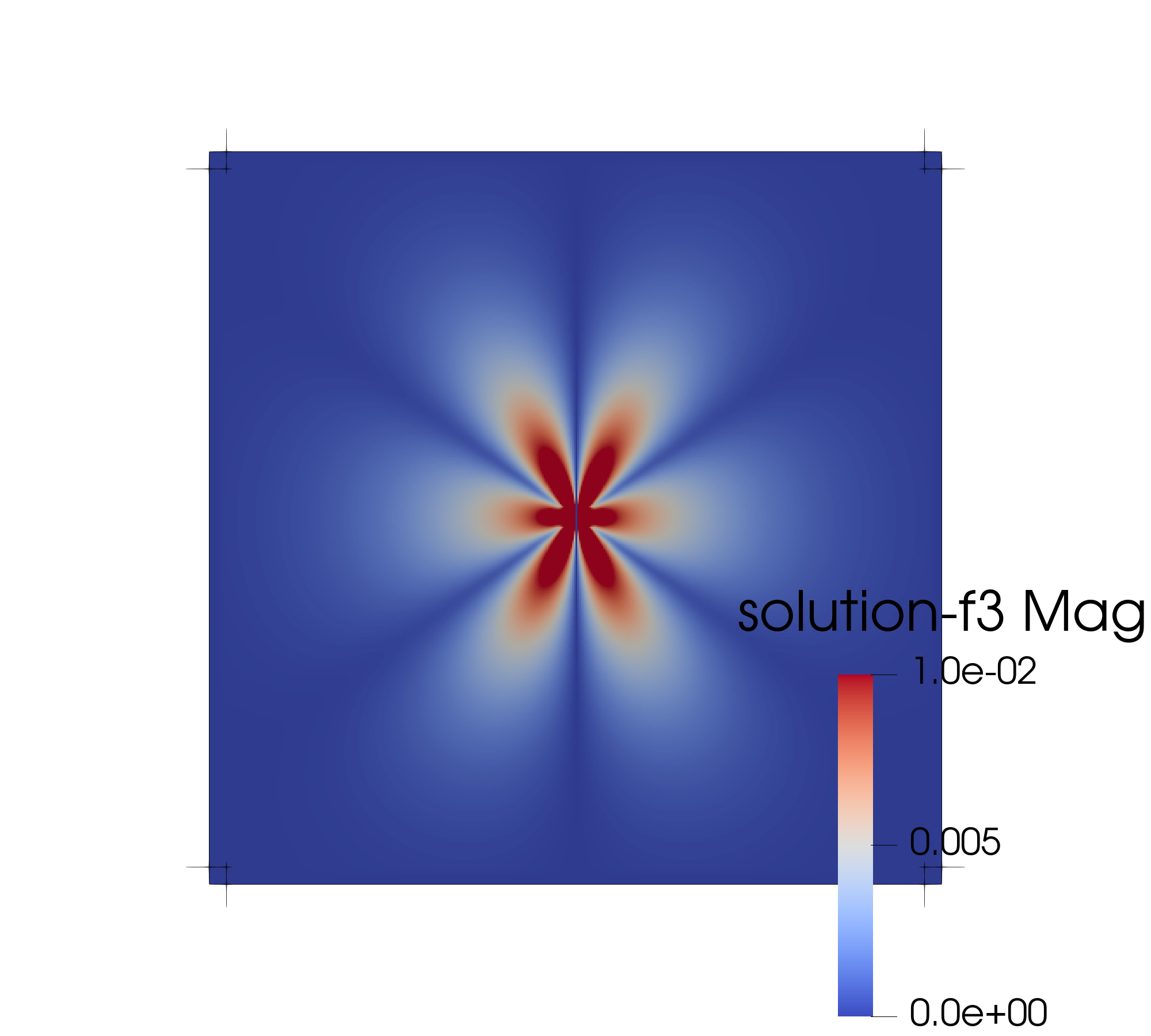}
    \caption{}
  \end{subfigure}
  \begin{subfigure}[b]{0.240\textwidth}
    \includegraphics[width=\textwidth]{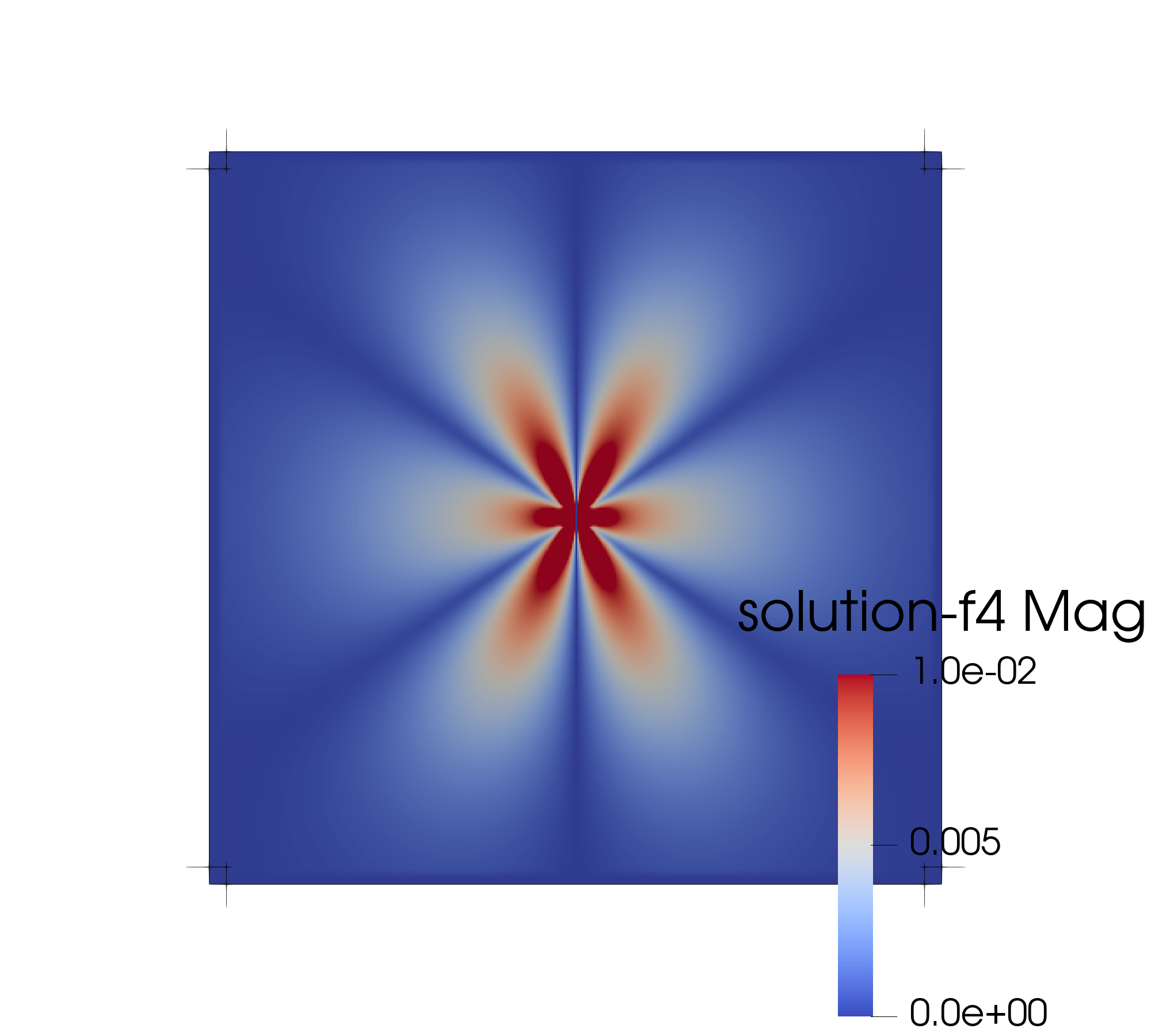}
    \caption{}
  \end{subfigure}
  \caption{The solution magnitude from the first four outer FGMRES iterations for the 2D homogeneous Helmholtz equation ($\mathtt{freq}=40$, $t=0.4/\mathtt{freq}^2$, and $s=4/\mathtt{freq}$).
    Subfigures (a)--(d) correspond to consecutive iterations.
    For better visualization, the magnitude is capped at $0.01$.}\label{fig:solution-2d-iteration}
\end{figure}

\subsection{3D homogeneous models}
We then extend our numerical experiments to 3D homogeneous models.
The inner domain now is a unit cube $(0,1)^3$, and we again fix $\mathtt{ppw}=10$.
The source term is defined as
\begin{align*}
  f(x,y,z) & = -\operatorname{normal}(x,y,z;\tfrac{1}{2}-6h, \tfrac{1}{2}, \tfrac{1}{2}, 3h)+\operatorname{normal}(x,y,z;\tfrac{1}{2}+6h, \tfrac{1}{2}, \tfrac{1}{2}, 3h)      \\
           & \quad -\operatorname{normal}(x,y,z;\tfrac{1}{2}, \tfrac{1}{2}-6h, \tfrac{1}{2}, 3h) +\operatorname{normal}(x,y,z;\tfrac{1}{2}, \tfrac{1}{2}+6h, \tfrac{1}{2}, 3h) \\
           & \quad -\operatorname{normal}(x,y,z;\tfrac{1}{2}, \tfrac{1}{2}, \tfrac{1}{2}-6h, 3h)+\operatorname{normal}(x,y,z;\tfrac{1}{2}, \tfrac{1}{2}, \tfrac{1}{2}+6h, 3h),
\end{align*}
where now $\operatorname{normal}$ denotes the 3D normalized Gaussian function (cf.\ the 2D case in \cref{subsec:2d-homogeneous-models}).

We determine an appropriate time length \(t\) for the 3D homogeneous Helmholtz equation by repeating the experiments described in \cref{tab:exp-iter-counts}.
The results, summarized in \cref{tab:exp-iter-counts-3d}, indicate that the iteration counts for computing \(\exp(\i A t)v\) and \(\psi_1(\i A t)v\) remain nearly identical.
However, compared to the 2D case, a slight increase in iteration counts is observed, which is expected because \(A\) is now obtained from a seven-point scheme that is less sparse than the five-point scheme in 2D.
By setting \(t=0.4/\mathtt{freq}^2\), the iteration counts remain stable across varying frequencies and stay below \(10\).
Therefore, we adopt this value of \(t\) in the subsequent experiments.

\begin{table}[!ht]
  \centering
  \caption{The iteration counts for computing the exponential integrators $\exp(\i A t)v$ and $\psi_1(\i A t)v$ for various values of \(t\) proportional to \(1/\mathtt{freq}^2\), where $A$ is from discretizing the 3D homogeneous Helmholtz equation.
    The columns labeled $\psi_0$ and $\psi_1$ correspond to $\exp(\i A t)v$ and $\psi_1(\i A t)v$, respectively.}
  \label{tab:exp-iter-counts-3d}
  \begin{footnotesize}
    \makegapedcells
    \begin{tabular}{c c c c c c c c c}
      \toprule
      \multirow{2}{*}{$\mathtt{freq}$} & \multicolumn{2}{c}{$t=\frac{0.4}{\mathtt{freq} \times \mathtt{freq}}$} & \multicolumn{2}{c}{$t=\frac{0.8}{\mathtt{freq} \times \mathtt{freq}}$} & \multicolumn{2}{c}{$t=\frac{1.2}{\mathtt{freq} \times \mathtt{freq}}$} & \multicolumn{2}{c}{$t=\frac{1.6}{\mathtt{freq} \times \mathtt{freq}}$}                                             \\
      \cline{2-9}
                                       & $\psi_0$                                                               & $\psi_1$                                                               & $\psi_0$                                                               & $\psi_1$                                                               & $\psi_0$ & $\psi_1$ & $\psi_0$ & $\psi_1$ \\
      \midrule
      $10$                             & $9$                                                                    & $9$                                                                    & $15$                                                                   & $16$                                                                   & $19$     & $20$     & $22$     & $22$     \\
      $20$                             & $9$                                                                    & $8$                                                                    & $16$                                                                   & $16$                                                                   & $22$     & $24$     & $27$     & $30$     \\
      $40$                             & $9$                                                                    & $8$                                                                    & $17$                                                                   & $16$                                                                   & $25$     & $24$     & $33$     & $32$     \\
      $80$                             & $9$                                                                    & $8$                                                                    & $17$                                                                   & $17$                                                                   & $25$     & $25$     & $33$     & $34$     \\
      \bottomrule
    \end{tabular}
  \end{footnotesize}
\end{table}

We provide a visualization of the solutions for the 3D homogeneous Helmholtz equation computed by the proposed method with frequency $\mathtt{freq}\in \CurlyBrackets{10, 20, 40, 80}$, as illustrated in \cref{fig:solution-3d-freq10}.
The remaining parameters will be exposed later.
For clarity, the magnitude is limited to the range $[-0.1,0.1]$.
It is evident that the intricate wave structures become more pronounced at higher frequencies, and all cases exhibit distinct magnitude ``valleys'', which are likely attributable to the right-hand side being a superposition of several poles.

\begin{figure}[!ht]
  \centering
  \begin{subfigure}[b]{0.24\textwidth}
    \includegraphics[width=\textwidth]{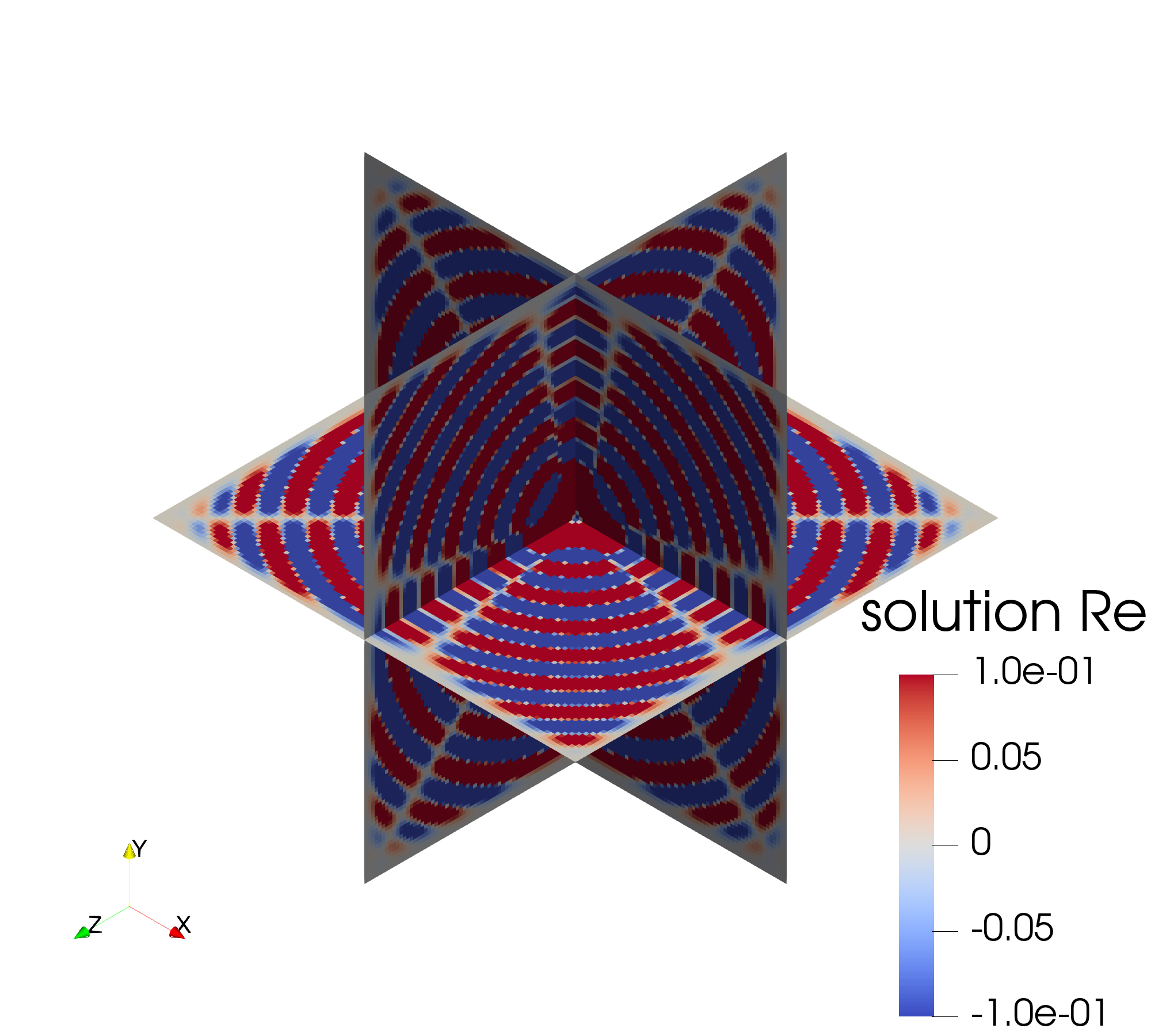}
    \caption{}
  \end{subfigure}
  \begin{subfigure}[b]{0.24\textwidth}
    \includegraphics[width=\textwidth]{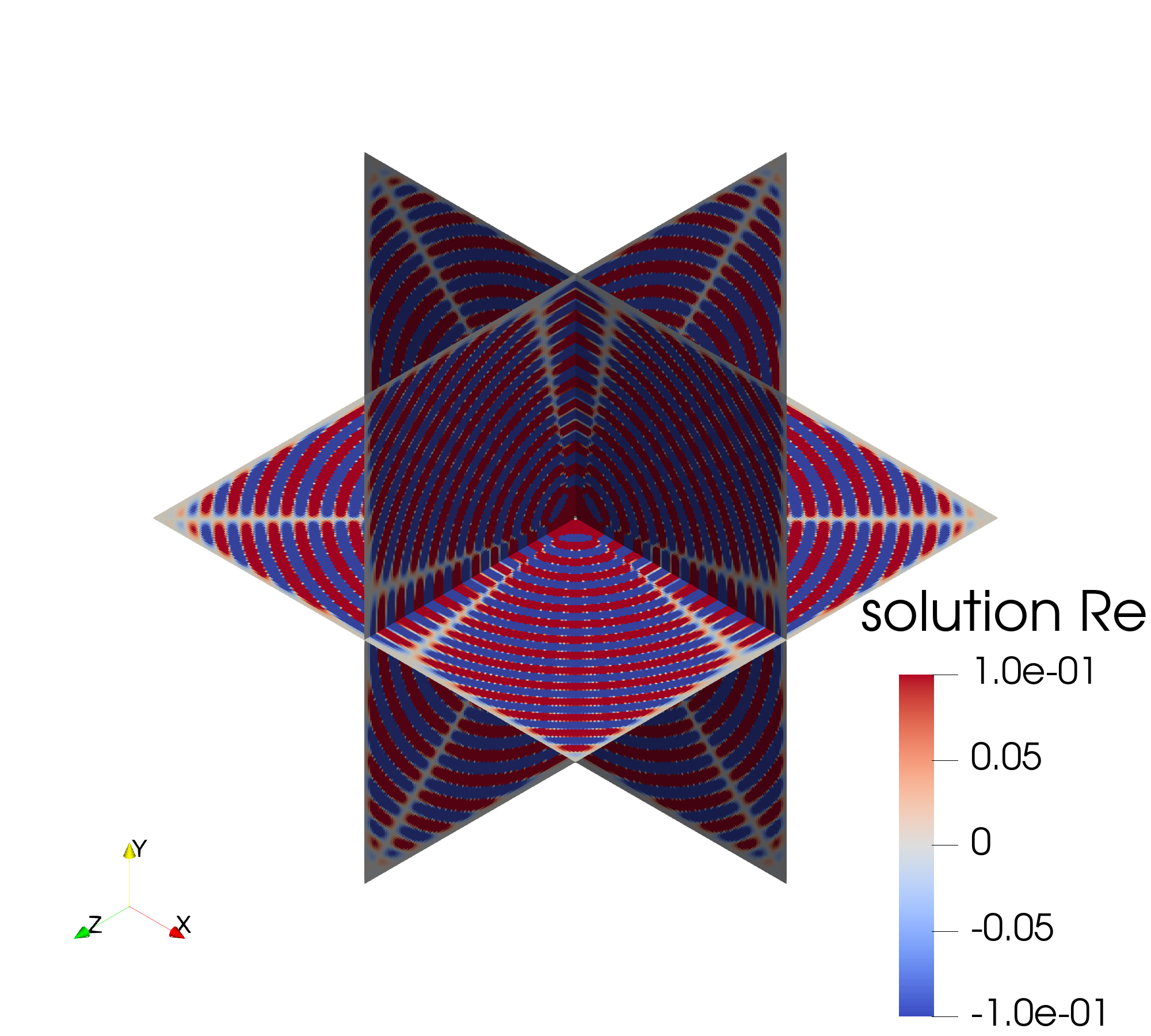}
    \caption{}
  \end{subfigure}
  \begin{subfigure}[b]{0.24\textwidth}
    \includegraphics[width=\textwidth]{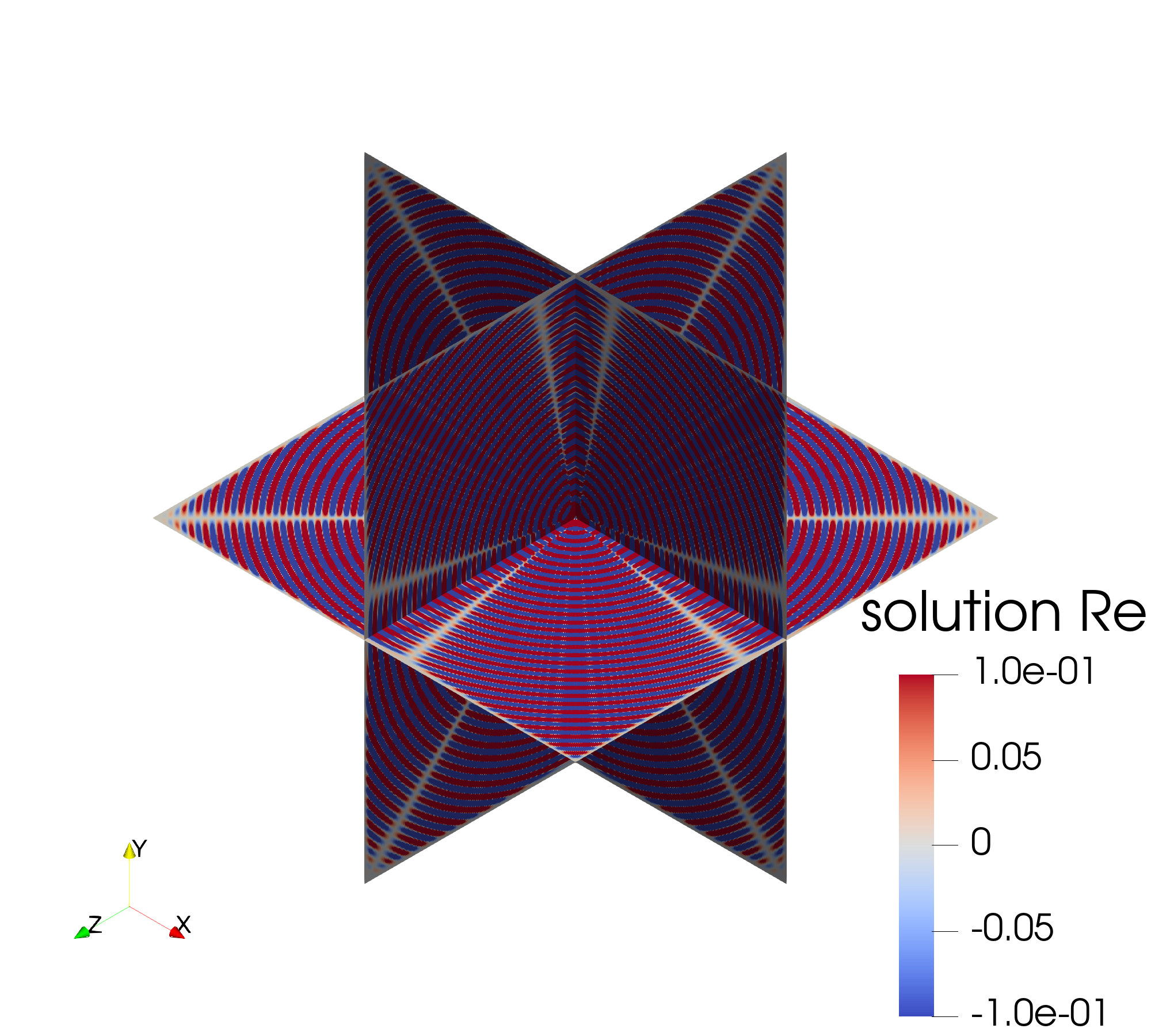}
    \caption{}
  \end{subfigure}
  \begin{subfigure}[b]{0.24\textwidth}
    \includegraphics[width=\textwidth]{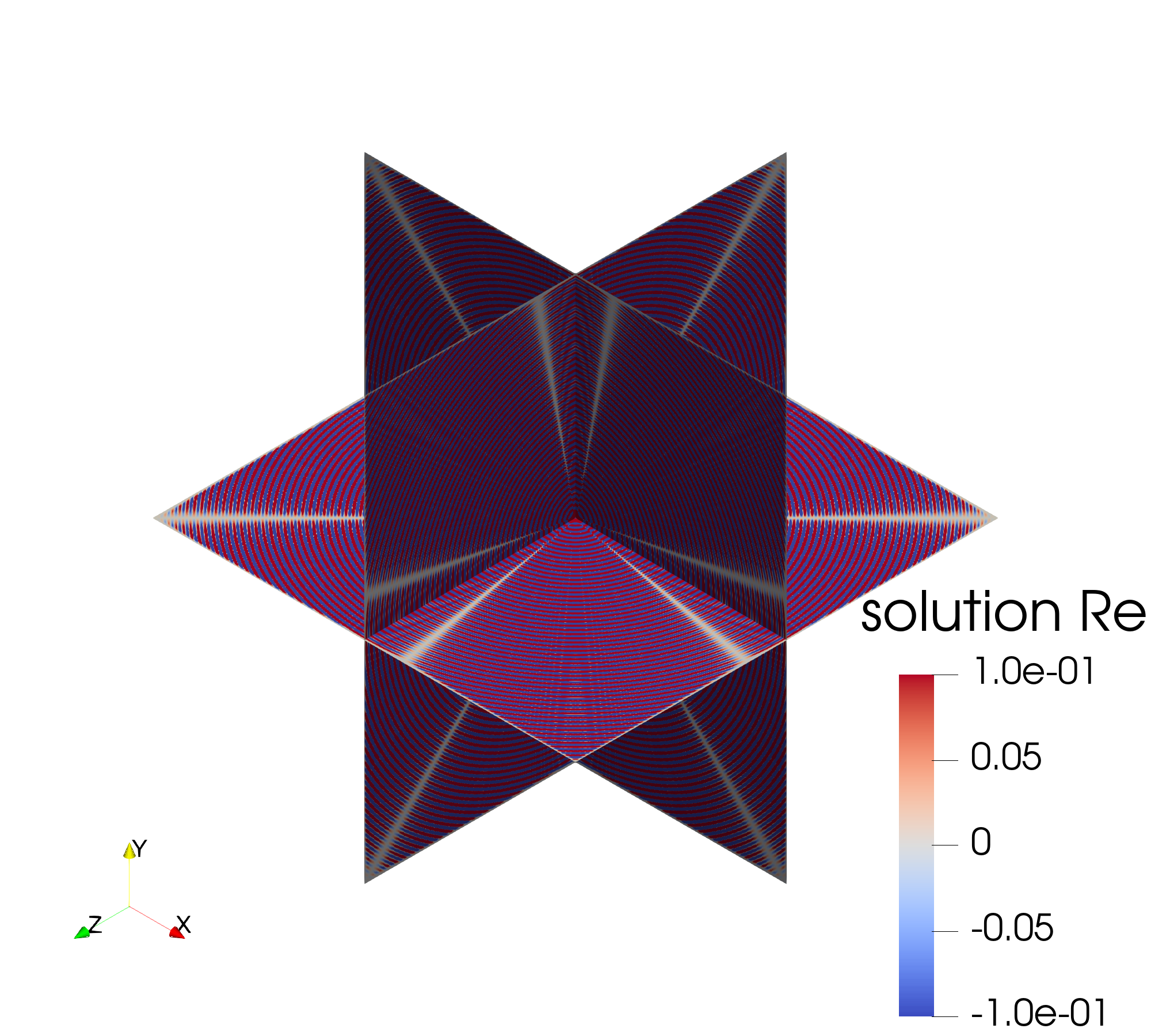}
    \caption{}
  \end{subfigure}
  \caption{Magnitude of the solutions for the 3D homogeneous Helmholtz equation at frequencies
    $\mathtt{freq}=10$, $20$, $40$, and $80$.
    Subfigures (a)--(d) correspond to each frequency, respectively.}\label{fig:solution-3d-freq10}
\end{figure}

\subsubsection{The effect of the fixed-point tolerance}
Without a shift, if the fixed-point system is solved and to high precision, the outer FGMRES is expected to converge in only a few iterations.
This is because we are essentially solving an equivalent system by \cref{eq:magic-inverse} that involves matrix exponentials, rendering the inner tolerance $\mathtt{fprtol}$ somewhat redundant compared to the accuracy of the outer solver.
However, as demonstrated in \cref{subsubsec:effect-of-shift}, since the shift is employed to both accelerate and stabilize the overall convergence, it is necessary to choose an appropriate inner tolerance that effectively balances the convergence behavior of the outer FGMRES with the overall efficiency.

We conduct a series of experiments with $s=1/\mathtt{freq}$ and $t=0.4/\mathtt{freq}^2$, while varying the inner tolerance $\mathtt{fprtol}\in \CurlyBrackets{0.01,0.02,0.04,0.08}$.
The results are tabulated in \cref{tab:gmres-iter-counts-shift-3d}.
Notably, the outer FGMRES iterations remain constant at $8$ across all cases, demonstrating that the combined inner and outer solvers are robust even when the inner tolerance is as high as nearly $0.1$.
Furthermore, increasing the inner tolerance yields significant savings in the inner GMRES iterations for higher frequencies; for example, the average number of inner iterations decreases from $11.1$ to $7.0$ for $\mathtt{freq}=80$.
We postulate that the acceleration provided by the inner GMRES solver contributes to a nearly invariant inner iteration count across different tolerances.
Moreover, \cref{tab:gmres-iter-counts-shift-3d} indicates that the chosen parameters $s$ and $t$ result in a linear growth of SpMV operations with respect to frequency, as the inner GMRES iterations are almost proportional to $\mathtt{freq}$.
Lastly, compared with the data in \cref{tab:gmres-iter-counts}, our method does not exhibit significant performance degradation in the 3D case, unlike several existing methods that rely heavily on domain transmission techniques \cite{Engquist2011,Chen2013,Gander2019}.

\begin{table}[!ht]
  \caption{The number of outer FGMRES iterations (column labeled ``outer'') and inner GMRES iterations (column labeled ``inner'', presented as ``mean $\pm$ standard deviation'') for the 3D homogeneous Helmholtz equation under various fixed-point tolerances.}
  \label{tab:gmres-iter-counts-shift-3d}
  \centering
  \begin{footnotesize}
    \makegapedcells
    \begin{tabular}{c c c c c c c c c}
      \toprule
      \multirow{2}{*}{$\mathtt{freq}$} & \multicolumn{2}{c}{$\mathtt{fprtol}=0.01$} & \multicolumn{2}{c}{$\mathtt{fprtol}=0.02$} & \multicolumn{2}{c}{$\mathtt{fprtol}=0.04$} & \multicolumn{2}{c}{$\mathtt{fprtol}=0.08$}                                               \\
      \cline{2-9}
                                       & outer                                      & inner                                      & outer                                      & inner                                      & outer & inner        & outer & inner        \\
      \midrule
      $10$                             & $8$                                        & $2.1\pm 0.3$                               & $8$                                        & $2.0\pm 0.0$                               & $8$   & $2.0\pm 0.0$ & $8$   & $2.0\pm 0.5$ \\
      $20$                             & $8$                                        & $3.6\pm 0.6$                               & $8$                                        & $3.1\pm 0.3$                               & $8$   & $2.5\pm 0.7$ & $8$   & $2.2\pm 0.4$ \\
      $40$                             & $8$                                        & $5.6\pm 0.9$                               & $8$                                        & $5.0\pm 1.0$                               & $8$   & $4.3\pm 0.8$ & $8$   & $3.7\pm 0.8$ \\
      $80$                             & $8$                                        & $11.1\pm 1.7$                              & $8$                                        & $9.8\pm 1.8$                               & $8$   & $8.6\pm 1.9$ & $8$   & $7.0\pm 1.5$ \\
      \bottomrule
    \end{tabular}
  \end{footnotesize}
\end{table}

\subsubsection{Scalability tests}
Our parallel implementation of the proposed preconditioner is based on the MPI standard, which distributes the matrix \(A\) and associated vectors among the available processes.
Message passing is employed to construct Krylov subspaces---encompassing SpMV operations and inner products---for computing matrix functions, as well as inner GMRES and outer FGMRES iterations.
We utilize the DMDA module in PETSc to manage data distribution and communication on structured grids, whose design philosophy is to minimize communication overhead and maximize data locality.
Since the proposed preconditioner requires no matrix factorization or explicit sequential routines, we anticipate its parallel performance to exhibit satisfactory scalability.
For the forthcoming numerical experiments, we adopt the following parameters: \(s=1/\mathtt{freq}\), \(t=0.4/\mathtt{freq}^2\), and \(\mathtt{fprtol}=0.08\).

To validate the strong scalability of the proposed method, we test two cases, \(\mathtt{freq}=40\) and \(80\), using varying numbers of computing nodes.
In the first case, the total number of DoFs is approximately \(420^3\), while in the second case it is around \(820^3\); both linear systems are large and require substantial computational resources.
For \(\mathtt{freq}=40\), we employ \(10\), \(15\), \(20\), \(25\), and \(30\) nodes, whereas for \(\mathtt{freq}=80\) we use \(20\), \(40\), \(60\), and \(80\) nodes.
The wall-clock time for each case is recorded and presented in the left part of \cref{fig:strong-scaling}; a speed-up is observed with an increase in the number of nodes (for example, the computing time for \(\mathtt{freq}=80\) decreases from \(\num{4.887e+3}\) seconds to \(\num{1.371e+3}\) seconds as the number of nodes increases from \(20\) to \(80\)).
To further quantify strong scalability, we introduce the parallel efficiency defined as $\RoundBrackets{\mathtt{bench\_time} \times \mathtt{bench\_nodes}}/\RoundBrackets{\mathtt{time} \times \mathtt{nodes}}$,
where \(\mathtt{bench\_time}\) and \(\mathtt{bench\_nodes}\) represent the wall-clock time and number of nodes for the benchmark case (i.e., \(\mathtt{freq}=40\) with \(10\) nodes and \(\mathtt{freq}=80\) with \(20\) nodes, respectively).
An efficiency close to \(1\) indicates that the proposed method exhibits nearly linear strong scalability.
As shown in the right part of \cref{fig:strong-scaling}, the parallel efficiency for \(\mathtt{freq}=40\) even exceeds \(1\), which may be attributed to factors such as cache effects in the benchmark case, while for \(\mathtt{freq}=80\) the efficiency is around \(80\%\), primarily due to increased communication overhead with more nodes.
We anticipate further performance improvements if the per-node computing power is enhanced, for instance through CPU-GPU hybrid computing, which can mitigate the communication overhead observed with many nodes.

\begin{figure}[!ht]
  \centering
  \resizebox{\textwidth}{!}{\input{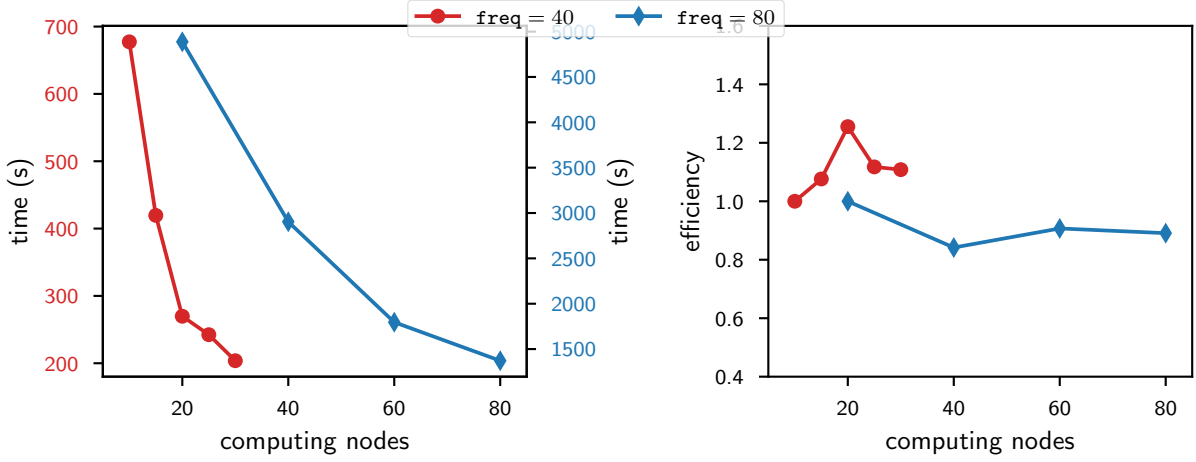}}
  \caption{Strong scalability results for the 3D homogeneous Helmholtz equation with \(\mathtt{freq}=40\) or \(80\). The left plot shows the wall-clock time for different numbers of nodes, while the right plot illustrates the corresponding parallel efficiency.}\label{fig:strong-scaling}
\end{figure}

Since the convergence of the proposed method depends linearly on the frequency, analyzing weak scalability--i.e., measuring the solution time while keeping the number of DoFs per node constant---is subtle.
Instead of recording the total solution time, we measure the time required to compute the matrix function \(\psi_1(\i A t)v\) followed by one step of fixed-point iteration, which constitutes the most computationally expensive part of the algorithm.
We vary the frequency over \(\CurlyBrackets{20, 30,\dots, 80}\) and set the number of nodes as $\mathtt{nodes} = \ceil{\RoundBrackets*{{\mathtt{freq}}/{20}}^3}$,
with the results illustrated in \cref{fig:weak-scaling}.
It can be observed that the computing time remains in the range of \(40\) to \(50\) seconds, corresponding to a weak scalability efficiency of approximately \(80\%\).

\begin{figure}[!ht]
  \centering
  \resizebox{0.8\textwidth}{!}{\input{figs/plot-weak.pgf}}
  \caption{Weak scalability results for the 3D homogeneous Helmholtz equation.
    The plot shows the time required for computing $\psi_1(\i A t)v$ and performing one fixed-point iteration, with the number of nodes scaled cubically to the frequency.}\label{fig:weak-scaling}
\end{figure}
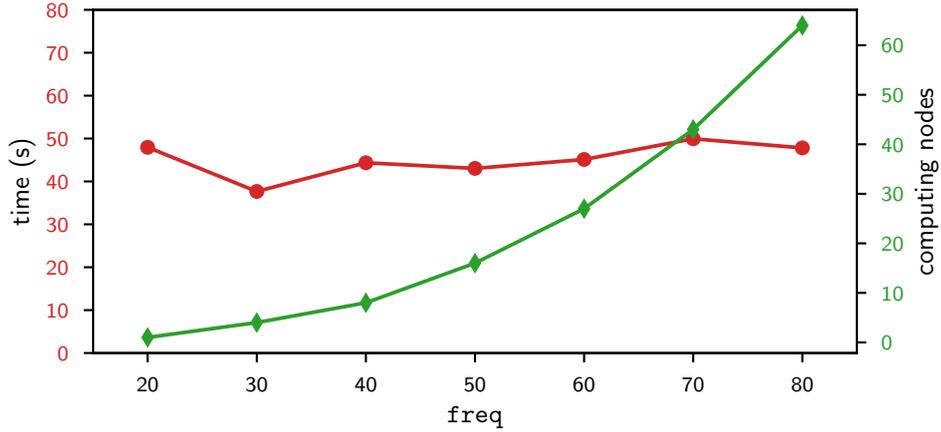

\subsection{Inhomogeneous models}
\label{subsec:inhomo-models}
In this section, we apply the proposed preconditioner to inhomogeneous models, including smooth models and geological models.
The experiments on homogeneous models provide an insight into the selection of the parameters, and hence we mainly follow the previous settings.

\subsubsection{Converging lens models}
For the 2D converging lens model, we consider a smooth velocity profile defined as
\[
  c(x, y) = 1 - \frac{1}{2}\exp\RoundBrackets*{-\frac{(x-x_0)^2+(y-y_0)^2}{2\sigma^2}},
\]
with \((x_0, y_0)=(0.5,\, 0.1)\) and \(\sigma=1/32\).
Similarly, for the 3D converging lens model the velocity profile is given by
\[
  c(x, y, z)= 1 - \frac{1}{2}\exp\RoundBrackets*{-\frac{(x-x_0)^2+(y-y_0)^2+(z-z_0)^2}{2\sigma^2}},
\]
where \((x_0, y_0, z_0)=(0.5,\, 0.25,\, 0.5)\) and \(\sigma=1/32\).
These models exhibit a region of reduced velocity that acts as a converging lens.
Notice that \(c_{\mathup{min}}=1/2\), and in order to satisfy \(\mathtt{ppw}=10\) it is necessary that the mesh size \(h\) and the frequency \(\mathtt{freq}\) fulfill the relation \(h={1}/(20 \times \mathtt{freq})\).
The source term is taken as a delta function located at the center of the inner domain.

We set $\mathtt{freq}\in \CurlyBrackets{40,80,160}$ for the 2D case and $\mathtt{freq}\in \CurlyBrackets{10,20,40}$ for the 3D case.
Since the iteration counts for computing the matrix functions are largely influenced by the mesh size, and the $h$-to-$\omega$ relation has been modified (e.g., from $h=1/(10\times\mathtt{freq})$ to $h=1/(20\times\mathtt{freq})$), we now set $t=0.1/\mathtt{freq}^2$.
The shift parameter is set as before, i.e., $s=1/\mathtt{freq}$, and the inner tolerance is fixed at $\mathtt{fprtol}=0.08$.
The inner and outer iteration counts are reported in \cref{tab:gmres-iter-counts-lens}, where the PML layer now contains $20$ points, corresponding to one wavelength near the boundary region.
Notably, the outer FGMRES iteration counts remain stable across different frequencies, while the inner GMRES iterations exhibit a linear growth.
These patterns are consistent with the homogeneous case; however, the inner GMRES iterations are increased by almost a factor of $2$ compared to the homogeneous cases (cf.\ \cref{tab:gmres-iter-counts-shift-3d}).
This is likely because the spectral radius alone cannot predict the convergence behavior of GMRES methods, i.e., the value $\Cond(X)$ in \cref{eq:gmres-residual} may increase for inhomogeneous models.

\begin{table}[!ht]
  \caption{The number of outer FGMRES iterations (column labeled ``outer'') and inner GMRES iterations (column labeled ``inner'', presented as ``mean $\pm$ standard deviation'') for the 2D and 3D converging lens model.}
  \label{tab:gmres-iter-counts-lens}
  \centering
  \begin{footnotesize}
    \makegapedcells
    \begin{tabular}{c c c c c c c c}
      \toprule
      \multicolumn{4}{c}{2D model} & \multicolumn{4}{c}{3D model}                                                                              \\
      \cline{1-4} \cline{5-8}
      $\mathtt{freq}$              & DoFs                         & outer & inner          & $\mathtt{freq}$ & DoFs    & outer & inner         \\
      \midrule
      $40$                         & $840^2$                      & $11$  & $12.1\pm 2.5$  & $10$            & $240^3$ & $8$   & $3.8\pm 0.7$  \\
      $80$                         & $1640^2$                     & $10$  & $22.5\pm 5.4$  & $20$            & $440^3$ & $7$   & $7.5\pm 1.3$  \\
      $160$                        & $3240^2$                     & $10$  & $37.1\pm 13.6$ & $40$            & $840^3$ & $9$   & $12.7\pm 2.9$ \\
      \bottomrule
    \end{tabular}
  \end{footnotesize}
\end{table}

\subsubsection{SEG/EAGE salt model}
SEG/EAGE is a well-known benchmark seismic model featuring sophisticated structures \cite{Aminzadeh1997}.
The original model spans a physical domain of \(\qty{13.5}{km} \times \qty{13.5}{km}\times \qty{4.2}{km}\) and its velocity profile is discretized on a grid of \(676\times 676 \times 210\) points, with velocities ranging from \(\qty{1.500}{km.s^{-1}}\) to \(\qty{4.482}{km.s^{-1}}\).
As noted earlier, we scale both the spatial and temporal domains such that \(h=1/675\) and \(c_{\mathup{min}}=1\).
The scaled model is illustrated in \cref{fig:seg-eage}, where the PML layer is depicted with a thickness of \(10\) points.

\begin{figure}[!ht]
  \centering
  \begin{subfigure}[b]{0.32\textwidth}
    \centering
    \includegraphics[width=\textwidth]{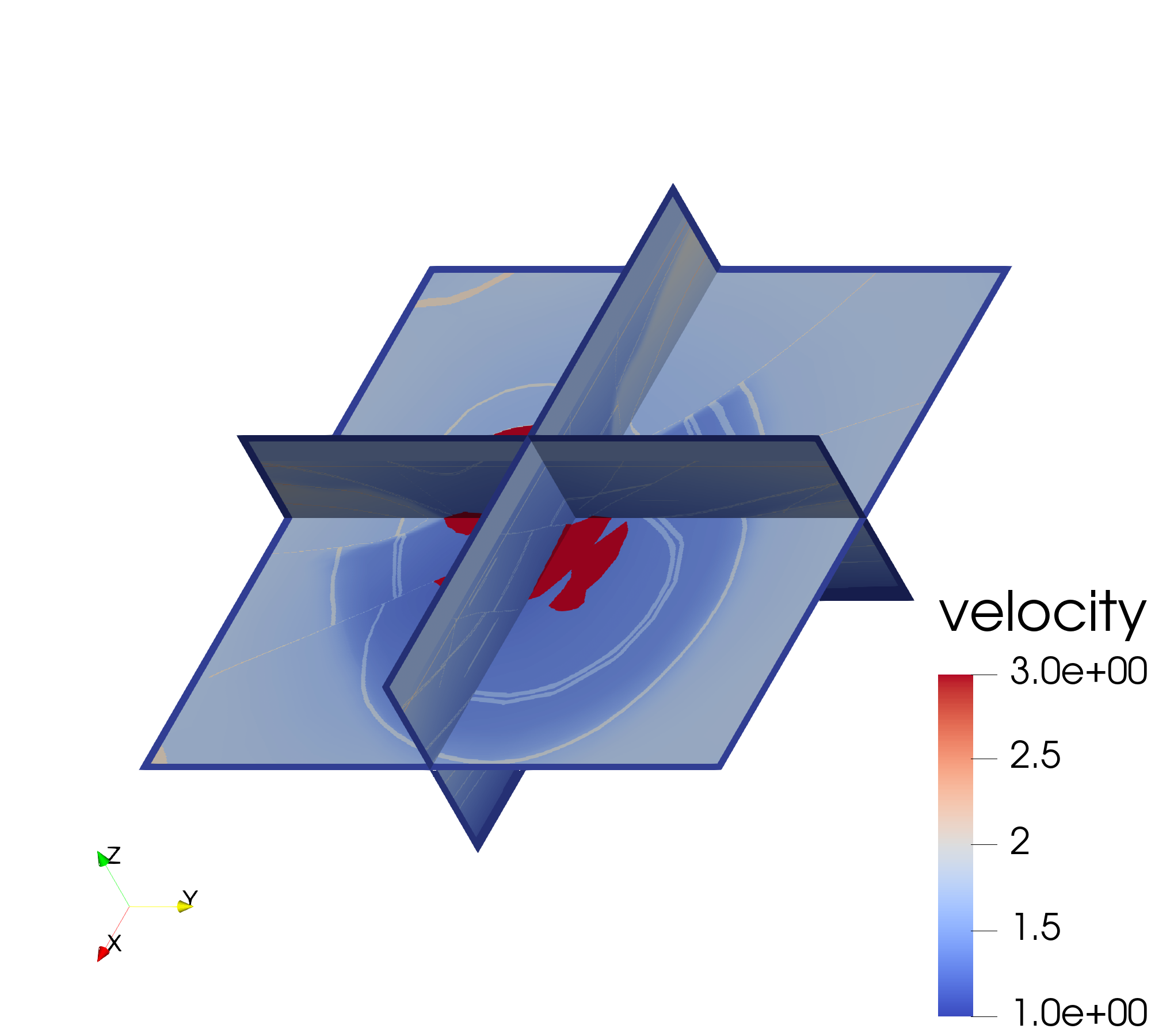}
    \caption{}
  \end{subfigure}
  \begin{subfigure}[b]{0.32\textwidth}
    \includegraphics[width=\textwidth]{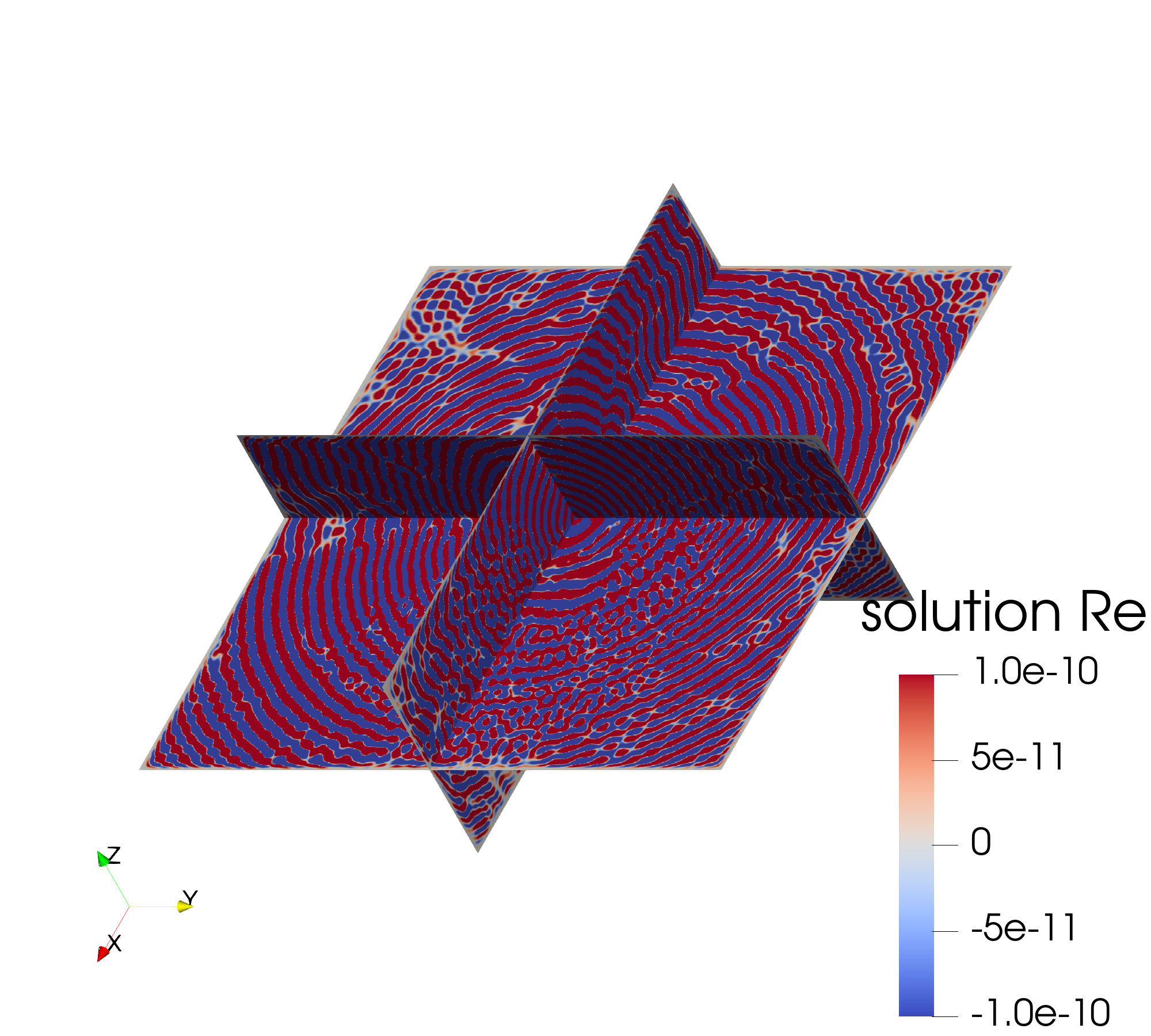}
    \caption{}
  \end{subfigure}
  \begin{subfigure}[b]{0.32\textwidth}
    \includegraphics[width=\textwidth]{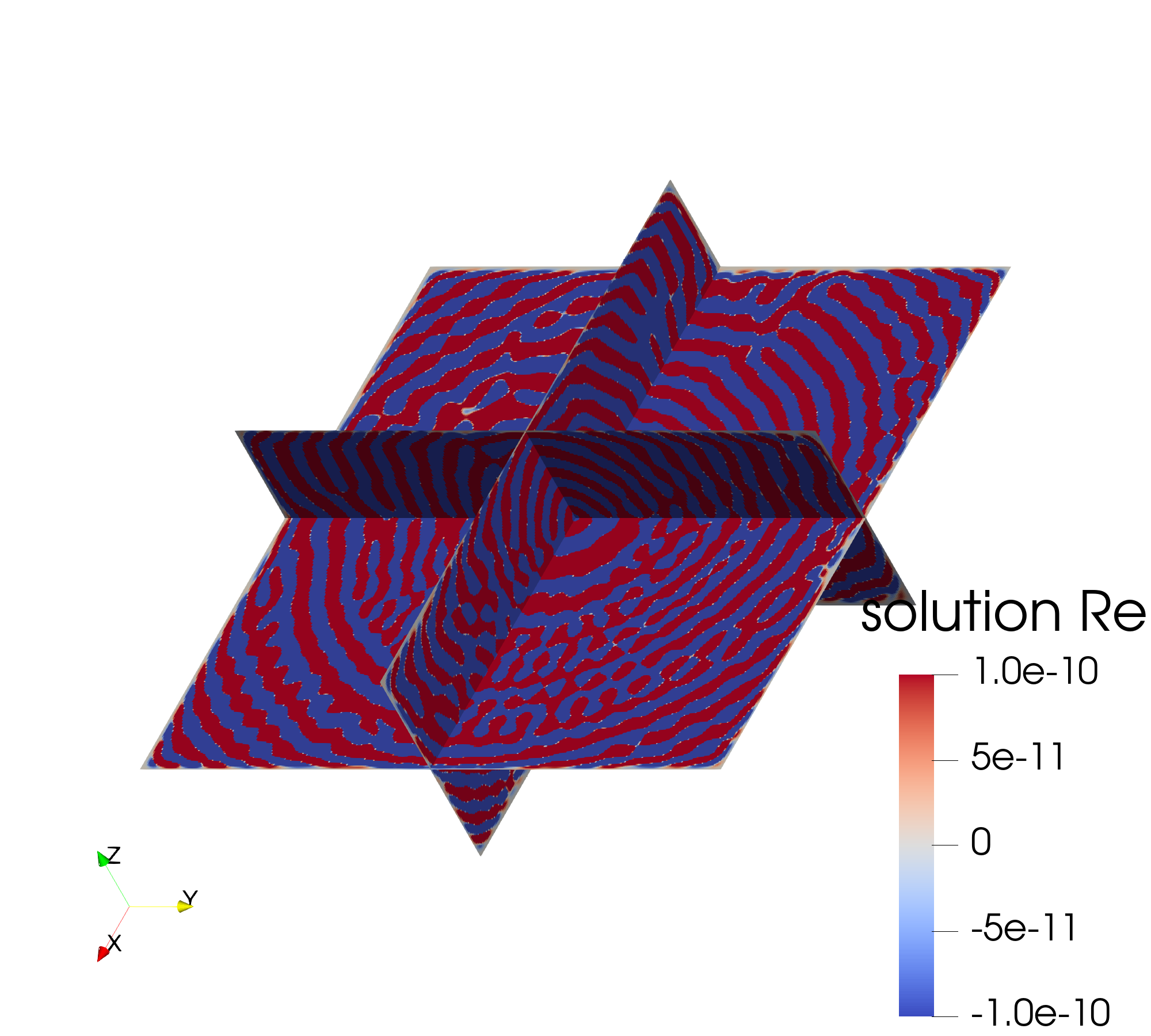}
    \caption{}
  \end{subfigure}
  \caption{(a) The scaled velocity field of the SEG/EAGE salt model with the PML layer enclosed.
    The real part of the solution for this model at frequencies \(\mathtt{freq}=67.5\) (b) and \(\mathtt{freq}=34\) (c).}\label{fig:seg-eage}
\end{figure}

To satisfy the requirement that $\mathtt{ppw}\geq 10$, the maximal frequency is set to \(\mathtt{freq}=67.5\).
For comparison, we also consider a lower frequency of \(\mathtt{freq}=34\) by coarsening the mesh while maintaining \(\mathtt{ppw}=10\).
The source term is defined as a delta function located at the center of the inner domain, and the real parts of the solutions for these two frequencies are presented in \cref{fig:seg-eage} with the magnitude suitably capped for improved visualization.
The computational parameters and iteration counts are reported in \cref{tab:gmres-iter-counts-salt}.
In this study, the two key parameters \(t\) and \(s\) are set following the previous rule, i.e., \(t=0.4/\mathtt{freq}^2\) and \(s=1/\mathtt{freq}\), with the values rounded to examine the robustness near the previous settings.
As demonstrated in \cref{tab:gmres-iter-counts-salt}, the inner GMRES iterations scale linearly with frequency, while the outer FGMRES iterations, although up to three times higher compared to the homogeneous cases, remain acceptable.

\begin{table}[!ht]
  \caption{The computational parameters and results for the SEG/EAGE salt model.}
  \label{tab:gmres-iter-counts-salt}
  \centering
  \begin{footnotesize}
    \makegapedcells
    \begin{tabular}{c c c c c c c}
      \toprule
      $\mathtt{freq}$ & DoFs                       & $t$            & $s$           & $\mathtt{fprtol}$ & outer & inner         \\
      \midrule
      $67.5$          & $695\times 695\times 229$  & $\num{9.0e-5}$ & $\num{0.015}$ & $\num{0.08}$      & $19$  & $12.7\pm 0.6$ \\
      $34$            & $357\times 357 \times 124$ & $\num{4.0e-5}$ & $\num{0.030}$ & $\num{0.08}$      & $26$  & $6.0\pm 0.1$  \\
      \bottomrule
    \end{tabular}
  \end{footnotesize}
\end{table}

\subsubsection{SEG/EAGE overthrust model}
SEG/EAGE overthrust model is another well-known benchmark seismic model \cite{Aminzadeh1997}.
The grids of this model contain \(801\times801\times187\) points, with a uniform grid spacing of \(\qty{0.025}{km}\).
The minimal and maximal velocities are \(\qty{2.178}{km.s^{-1}}\) and \(\qty{6.000}{km.s^{-1}}\), respectively.
We again scale both the spatial and temporal domains such that \(h=1/800\) and \(c_{\mathup{min}}=1\), and the scaled model is shown in \cref{fig:overthrust}.

\begin{figure}[!ht]
  \centering
  \begin{subfigure}[b]{0.32\textwidth}
    \centering
    \includegraphics[width=\textwidth]{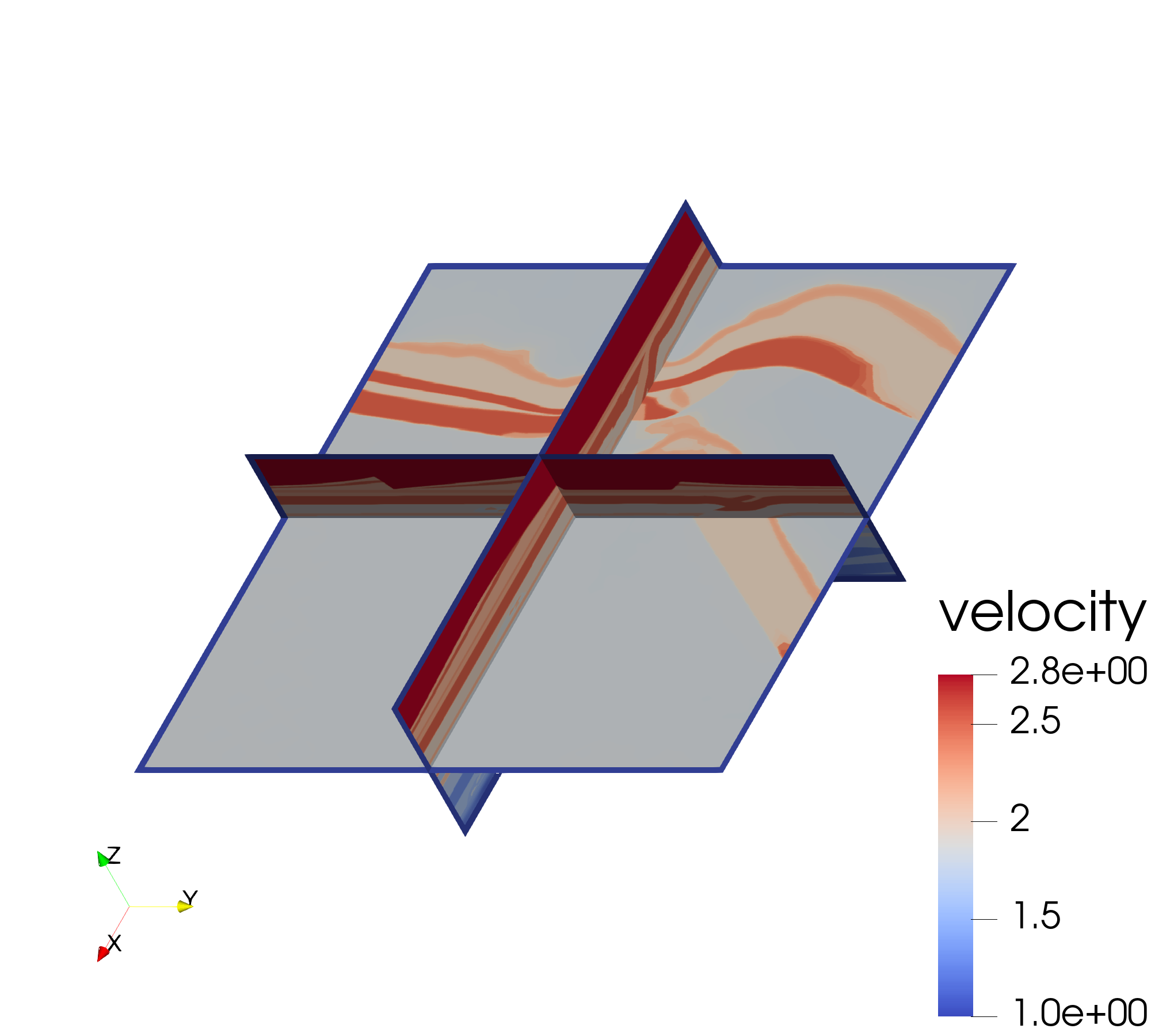}
    \caption{}
  \end{subfigure}
  \begin{subfigure}[b]{0.32\textwidth}
    \includegraphics[width=\textwidth]{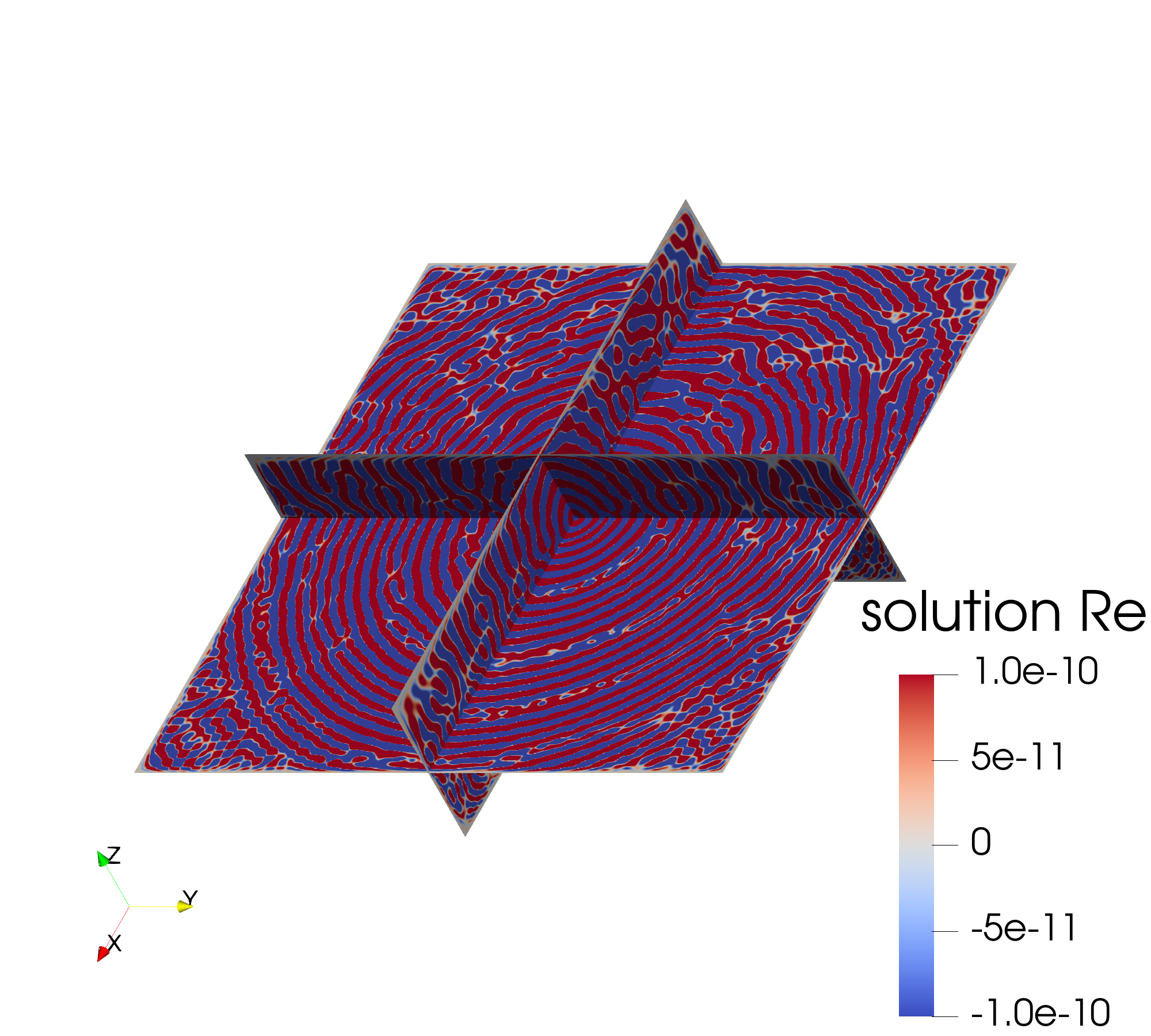}
    \caption{}
  \end{subfigure}
  \begin{subfigure}[b]{0.32\textwidth}
    \includegraphics[width=\textwidth]{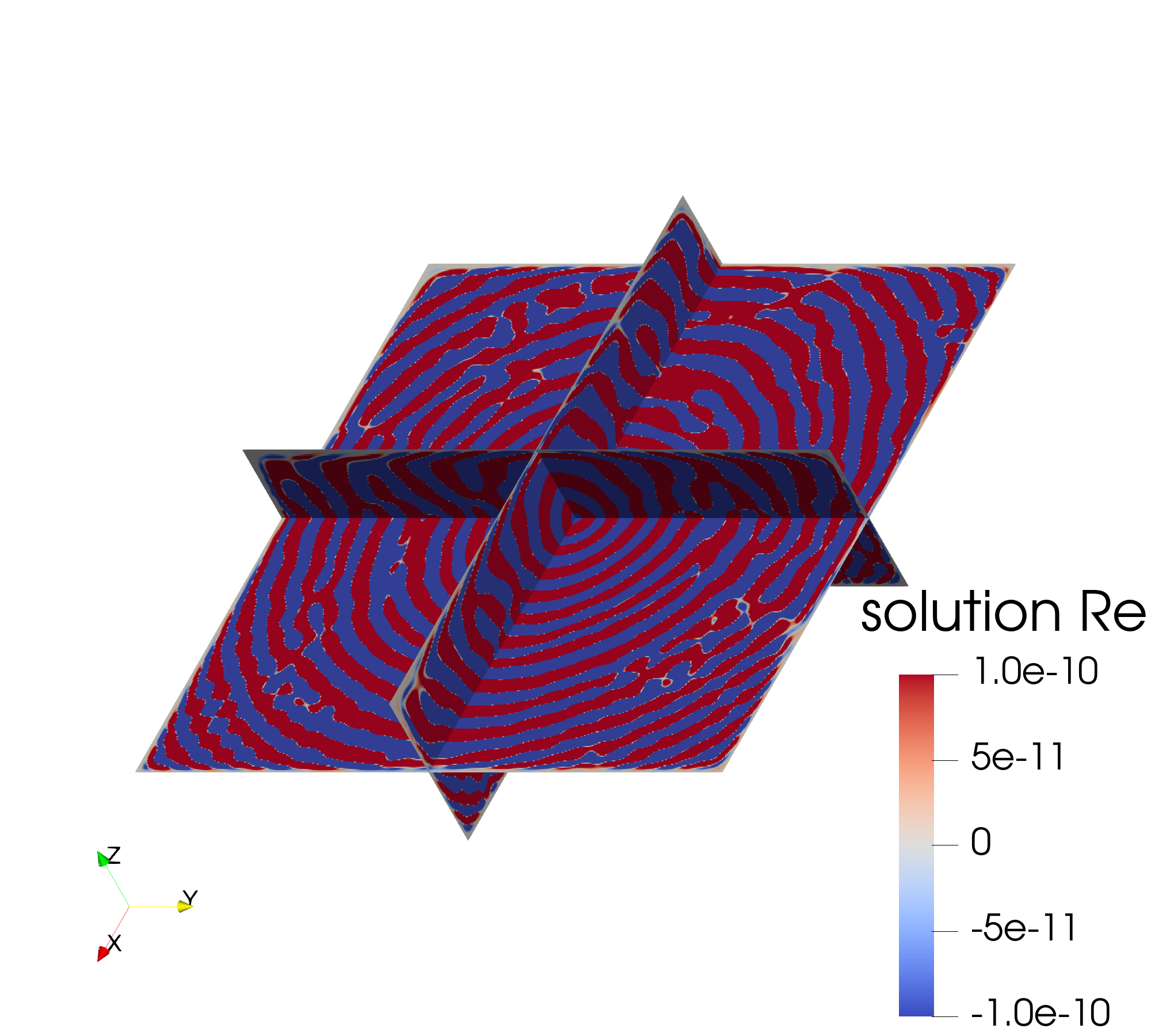}
    \caption{}
  \end{subfigure}
  \caption{(a) The scaled velocity field of the SEG/EAGE overthrust model with the PML layer enclosed.
    The real part of the solution for the SEG/EAGE overthrust model at frequencies \(\mathtt{freq}=80\) (b) and \(\mathtt{freq}=40\) (c).}\label{fig:overthrust}
\end{figure}

We again selected two frequencies, \(\mathtt{freq}=80\) and \(40\), to evaluate the performance of the proposed method.
The source term is given by a delta function located at the center of the inner domain.
The parameters \(s\) and \(t\) are chosen according to the established rule, and the real parts of the solutions at these two frequencies are shown in \cref{fig:overthrust}.
The corresponding iteration counts and other details are summarized in \cref{tab:gmres-iter-counts-overthrust}. As indicated in \cref{tab:gmres-iter-counts-overthrust}, the outer FGMRES iterations remain stable across the two frequencies, closely resembling the behavior observed in the homogeneous case (cf. \cref{tab:gmres-iter-counts-shift-3d}).
However, the inner GMRES iterations increase by approximately a factor of \(2\) for \(\mathtt{freq}=40\) and by a factor of \(3\) for \(\mathtt{freq}=80\) relative to the homogeneous cases.
These findings suggest that the performance of the proposed method is affected by the inhomogeneity of the model, which is natural given the increased complexity of wave propagation in such media and the method roots in the physical interpretation of the Helmholtz equation.

\begin{table}[!ht]
  \caption{The computational parameters and results for the SEG/EAGE overthrust model.}
  \label{tab:gmres-iter-counts-overthrust}
  \centering
  \begin{footnotesize}
    \makegapedcells
    \begin{tabular}{c c c c c c c}
      \toprule
      $\mathtt{freq}$ & DoFs                    & $t$             & $s$            & $\mathtt{fprtol}$ & outer & inner         \\
      \midrule
      $80$            & $820\times820\times206$ & $\num{6.25e-5}$ & $\num{0.0125}$ & $\num{0.08}$      & $11$  & $24.8\pm 1.7$ \\
      $40$            & $420\times420\times113$ & $\num{2.50e-4}$ & $\num{0.025}$  & $\num{0.08}$      & $12$  & $8.3\pm 0.6$  \\
      \bottomrule
    \end{tabular}
  \end{footnotesize}
\end{table}

\subsubsection{Marmousi-II model}
The Marmousi-II model is an updated version of the original Marmousi model, featuring an enlarged physical domain and increased velocity contrast \cite{Martin2006}; see subfigure (a) of \cref{fig:marmousi-ii-solution} for the P-wave velocity profile.
Notably, a water layer is incorporated atop the original model, making it more physically appropriate to enforce a sound-hard boundary condition at the water surface.
On the remaining boundaries, a first-order absorbing boundary condition is imposed \cite{Engquist1979}, with the FD discretization for these conditions implemented via the ghost point technique \cite{Trottenberg2007}.
According to \cite{Martin2006}, the source is placed near the top surface (\(z=\qty{10}{m}\)), and we adhere to this setting with its position centered in the \(x\)-direction.

\begin{figure}[!ht]
  \centering
  \begin{subfigure}[b]{\textwidth}
    \centering
    \includegraphics[width=0.49\textwidth]{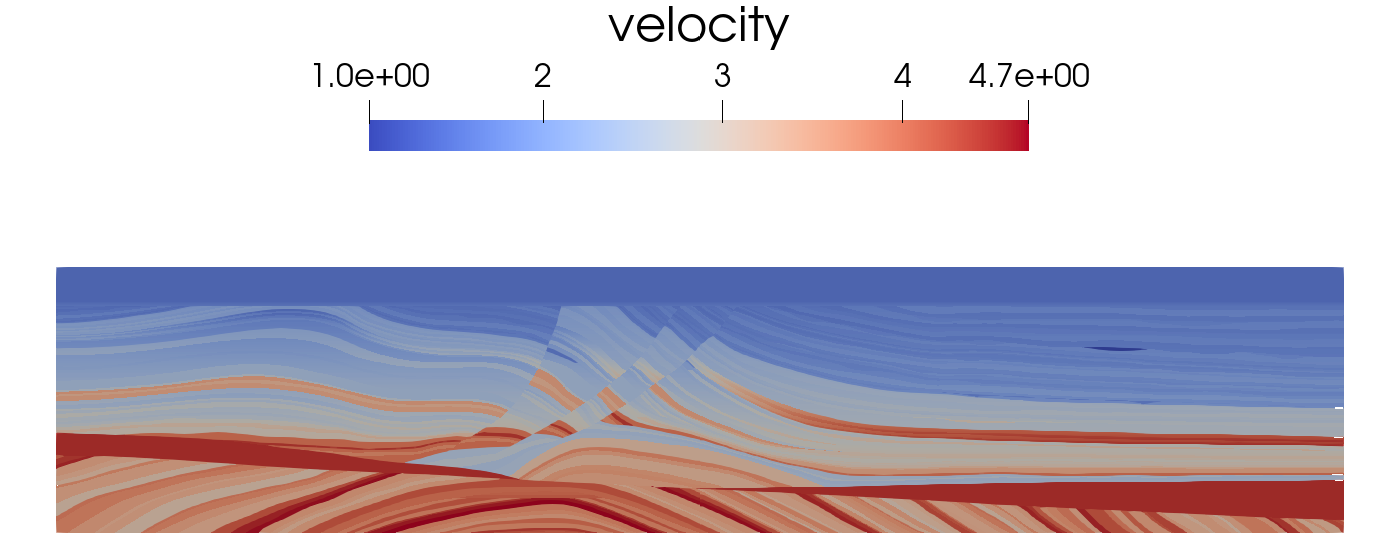}
    \caption{}
  \end{subfigure}
  \begin{subfigure}[b]{0.49\textwidth}
    \includegraphics[width=\textwidth]{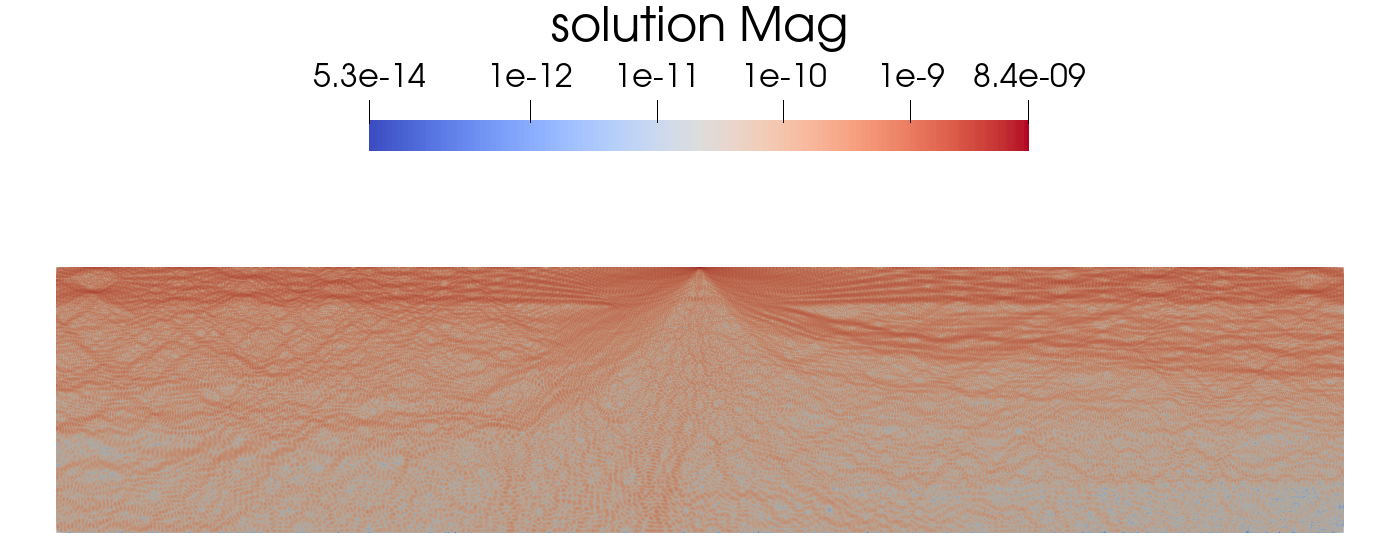}
    \caption{}
  \end{subfigure}
  \begin{subfigure}[b]{0.49\textwidth}
    \includegraphics[width=\textwidth]{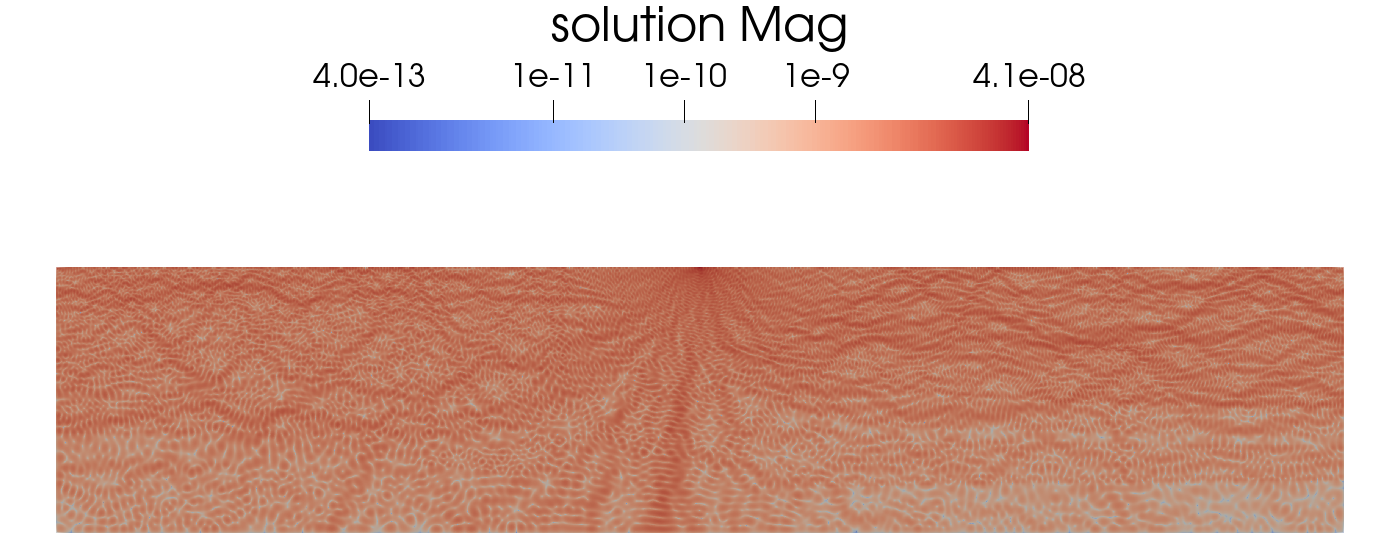}
    \caption{}
  \end{subfigure}
  \caption{The P-wave velocity field of the Marmousi-II model.
    The real part of the solution for the Marmousi-II model at frequencies \(\mathtt{freq}=680\) (b) and \(\mathtt{freq}=340\) (c).}\label{fig:marmousi-ii-solution}
\end{figure}

We also consider two frequencies \(\mathtt{freq}=680\) and \(340\), corresponding to the scaled model with \(h=1/6800\) and \(1/3400\), respectively.
In \cref{fig:marmousi-ii-solution}, the magnitude of the two solutions is visualized on a logarithmic scale, clearly revealing the complex wave structures.
The specific computational parameters and results are summarized in \cref{tab:marmousi-ii}.
Our method can be applied to this type of mixed boundary condition without any modifications.
However, compared to the 2D homogeneous models (cf.\ \cref{tab:gmres-iter-counts-shift}), we observe a significant increase in the inner GMRES iterations while the outer FGMRES iteration counts remain similar.
It can be verified that all eigenvalues of the matrix derived from this boundary condition have non-negative imaginary parts; thus, by applying a complex shift, the spectral gap to the real axis is enlarged.
As explained earlier, spectral information alone does not fully determine the convergence of the GMRES method, and a large condition number of the eigenvector matrix may deteriorate the convergence speed.
Therefore, we postulate that the eigenvectors corresponding to such mixed boundary conditions are more complex than those from PML methods, leading to a larger number of inner GMRES iterations.

\begin{table}[!ht]
  \caption{The computational parameters and results for the Marmousi-II model.}
  \label{tab:marmousi-ii}
  \centering
  \begin{footnotesize}
    \makegapedcells
    \begin{tabular}{c c c c c c c}
      \toprule
      $\mathtt{freq}$ & DoFs             & $t$            & $s$            & $\mathtt{fprtol}$ & outer & inner           \\
      \midrule
      $680$           & $6800\times1400$ & $\num{1.0e-6}$ & $\num{1.5e-3}$ & $\num{0.08}$      & $12$  & $269.9\pm 83.4$ \\
      $340$           & $3400\times700$  & $\num{4.0e-6}$ & $\num{3.0e-3}$ & $\num{0.08}$      & $12$  & $152.0\pm 29.4$ \\
      \bottomrule
    \end{tabular}
  \end{footnotesize}
\end{table}

\section{Conclusion}
\label{sec:conclusion}
As computing power rapidly evolves, harnessing the potential of modern hardware has become a pressing challenge.
For instance, GPU architectures feature numerous small computing units that can execute massive numbers of threads simultaneously.
Heavy sequential routines within an algorithm can significantly hinder overall performance on such architectures.
Additionally, efficient implementations often require a deep understanding of the underlying hardware architecture, which is not always feasible.
Therefore, it is crucial to develop algorithms that combine both parallelism and portability.

In this paper, we have introduced a new preconditioner for the Helmholtz equation, named MatExPre, which is designed to be both efficient and straightforward to implement on modern hardware architectures.
Reminiscent of several time-domain methods, such as WaveHoltz, the proposed preconditioner is derived from a time-domain Schr\"{o}dinger-like equation rather than the traditional wave equation.
We have further demonstrated that extending the time-domain approach leads to a modular algebraic solver incorporating matrix exponential integral actions.
The algorithm is composed of nested Krylov subspace solvers, in which the inner GMRES solver accelerates the fixed-point system while the outer FGMRES solver solves the preconditioned system.
We have also introduced a technique of applying a complex shift to the original Helmholtz operator within the preconditioner to further accelerate the convergence of the inner GMRES solver.
Our observations indicate that, with appropriate parameter selections, the proposed preconditioner can solve high-frequency Helmholtz problems with the number of SpMV operations growing linearly with frequency.
Several aspects of the proposed method remain open to improvement.
For example, given its flexibility with different types of discretizations, it would be meaningful to extend the method to alternative schemes that offer advantages for high-frequency simulations.

%% file: figs/pml-1d-lambda-C-width.pgf
\begingroup%
\makeatletter%
\begin{pgfpicture}%
\pgfpathrectangle{\pgfpointorigin}{\pgfqpoint{6.500000in}{3.250000in}}%
\pgfusepath{use as bounding box, clip}%
\begin{pgfscope}%
\pgfsetbuttcap%
\pgfsetmiterjoin%
\definecolor{currentfill}{rgb}{1.000000,1.000000,1.000000}%
\pgfsetfillcolor{currentfill}%
\pgfsetlinewidth{0.000000pt}%
\definecolor{currentstroke}{rgb}{1.000000,1.000000,1.000000}%
\pgfsetstrokecolor{currentstroke}%
\pgfsetdash{}{0pt}%
\pgfpathmoveto{\pgfqpoint{0.000000in}{0.000000in}}%
\pgfpathlineto{\pgfqpoint{6.500000in}{0.000000in}}%
\pgfpathlineto{\pgfqpoint{6.500000in}{3.250000in}}%
\pgfpathlineto{\pgfqpoint{0.000000in}{3.250000in}}%
\pgfpathlineto{\pgfqpoint{0.000000in}{0.000000in}}%
\pgfpathclose%
\pgfusepath{fill}%
\end{pgfscope}%
\begin{pgfscope}%
\pgfsetbuttcap%
\pgfsetmiterjoin%
\definecolor{currentfill}{rgb}{1.000000,1.000000,1.000000}%
\pgfsetfillcolor{currentfill}%
\pgfsetlinewidth{0.000000pt}%
\definecolor{currentstroke}{rgb}{0.000000,0.000000,0.000000}%
\pgfsetstrokecolor{currentstroke}%
\pgfsetstrokeopacity{0.000000}%
\pgfsetdash{}{0pt}%
\pgfpathmoveto{\pgfqpoint{0.762486in}{0.611944in}}%
\pgfpathlineto{\pgfqpoint{3.172206in}{0.611944in}}%
\pgfpathlineto{\pgfqpoint{3.172206in}{2.886667in}}%
\pgfpathlineto{\pgfqpoint{0.762486in}{2.886667in}}%
\pgfpathlineto{\pgfqpoint{0.762486in}{0.611944in}}%
\pgfpathclose%
\pgfusepath{fill}%
\end{pgfscope}%
\begin{pgfscope}%
\pgfsetbuttcap%
\pgfsetroundjoin%
\definecolor{currentfill}{rgb}{0.000000,0.000000,0.000000}%
\pgfsetfillcolor{currentfill}%
\pgfsetlinewidth{0.803000pt}%
\definecolor{currentstroke}{rgb}{0.000000,0.000000,0.000000}%
\pgfsetstrokecolor{currentstroke}%
\pgfsetdash{}{0pt}%
\pgfsys@defobject{currentmarker}{\pgfqpoint{0.000000in}{-0.048611in}}{\pgfqpoint{0.000000in}{0.000000in}}{%
\pgfpathmoveto{\pgfqpoint{0.000000in}{0.000000in}}%
\pgfpathlineto{\pgfqpoint{0.000000in}{-0.048611in}}%
\pgfusepath{stroke,fill}%
}%
\begin{pgfscope}%
\pgfsys@transformshift{1.333209in}{0.611944in}%
\pgfsys@useobject{currentmarker}{}%
\end{pgfscope}%
\end{pgfscope}%
\begin{pgfscope}%
\definecolor{textcolor}{rgb}{0.000000,0.000000,0.000000}%
\pgfsetstrokecolor{textcolor}%
\pgfsetfillcolor{textcolor}%
\pgftext[x=1.333209in,y=0.485556in,,top]{\color{textcolor}{\sffamily\fontsize{8.000000}{9.600000}\selectfont\catcode`\^=\active\def^{\ifmmode\sp\else\^{}\fi}\catcode`\%=\active\def
\end{pgfscope}%
\begin{pgfscope}%
\pgfsetbuttcap%
\pgfsetroundjoin%
\definecolor{currentfill}{rgb}{0.000000,0.000000,0.000000}%
\pgfsetfillcolor{currentfill}%
\pgfsetlinewidth{0.803000pt}%
\definecolor{currentstroke}{rgb}{0.000000,0.000000,0.000000}%
\pgfsetstrokecolor{currentstroke}%
\pgfsetdash{}{0pt}%
\pgfsys@defobject{currentmarker}{\pgfqpoint{0.000000in}{-0.048611in}}{\pgfqpoint{0.000000in}{0.000000in}}{%
\pgfpathmoveto{\pgfqpoint{0.000000in}{0.000000in}}%
\pgfpathlineto{\pgfqpoint{0.000000in}{-0.048611in}}%
\pgfusepath{stroke,fill}%
}%
\begin{pgfscope}%
\pgfsys@transformshift{1.909697in}{0.611944in}%
\pgfsys@useobject{currentmarker}{}%
\end{pgfscope}%
\end{pgfscope}%
\begin{pgfscope}%
\definecolor{textcolor}{rgb}{0.000000,0.000000,0.000000}%
\pgfsetstrokecolor{textcolor}%
\pgfsetfillcolor{textcolor}%
\pgftext[x=1.909697in,y=0.485556in,,top]{\color{textcolor}{\sffamily\fontsize{8.000000}{9.600000}\selectfont\catcode`\^=\active\def^{\ifmmode\sp\else\^{}\fi}\catcode`\%=\active\def
\end{pgfscope}%
\begin{pgfscope}%
\pgfsetbuttcap%
\pgfsetroundjoin%
\definecolor{currentfill}{rgb}{0.000000,0.000000,0.000000}%
\pgfsetfillcolor{currentfill}%
\pgfsetlinewidth{0.803000pt}%
\definecolor{currentstroke}{rgb}{0.000000,0.000000,0.000000}%
\pgfsetstrokecolor{currentstroke}%
\pgfsetdash{}{0pt}%
\pgfsys@defobject{currentmarker}{\pgfqpoint{0.000000in}{-0.048611in}}{\pgfqpoint{0.000000in}{0.000000in}}{%
\pgfpathmoveto{\pgfqpoint{0.000000in}{0.000000in}}%
\pgfpathlineto{\pgfqpoint{0.000000in}{-0.048611in}}%
\pgfusepath{stroke,fill}%
}%
\begin{pgfscope}%
\pgfsys@transformshift{2.486185in}{0.611944in}%
\pgfsys@useobject{currentmarker}{}%
\end{pgfscope}%
\end{pgfscope}%
\begin{pgfscope}%
\definecolor{textcolor}{rgb}{0.000000,0.000000,0.000000}%
\pgfsetstrokecolor{textcolor}%
\pgfsetfillcolor{textcolor}%
\pgftext[x=2.486185in,y=0.485556in,,top]{\color{textcolor}{\sffamily\fontsize{8.000000}{9.600000}\selectfont\catcode`\^=\active\def^{\ifmmode\sp\else\^{}\fi}\catcode`\%=\active\def
\end{pgfscope}%
\begin{pgfscope}%
\pgfsetbuttcap%
\pgfsetroundjoin%
\definecolor{currentfill}{rgb}{0.000000,0.000000,0.000000}%
\pgfsetfillcolor{currentfill}%
\pgfsetlinewidth{0.803000pt}%
\definecolor{currentstroke}{rgb}{0.000000,0.000000,0.000000}%
\pgfsetstrokecolor{currentstroke}%
\pgfsetdash{}{0pt}%
\pgfsys@defobject{currentmarker}{\pgfqpoint{0.000000in}{-0.048611in}}{\pgfqpoint{0.000000in}{0.000000in}}{%
\pgfpathmoveto{\pgfqpoint{0.000000in}{0.000000in}}%
\pgfpathlineto{\pgfqpoint{0.000000in}{-0.048611in}}%
\pgfusepath{stroke,fill}%
}%
\begin{pgfscope}%
\pgfsys@transformshift{3.062673in}{0.611944in}%
\pgfsys@useobject{currentmarker}{}%
\end{pgfscope}%
\end{pgfscope}%
\begin{pgfscope}%
\definecolor{textcolor}{rgb}{0.000000,0.000000,0.000000}%
\pgfsetstrokecolor{textcolor}%
\pgfsetfillcolor{textcolor}%
\pgftext[x=3.062673in,y=0.485556in,,top]{\color{textcolor}{\sffamily\fontsize{8.000000}{9.600000}\selectfont\catcode`\^=\active\def^{\ifmmode\sp\else\^{}\fi}\catcode`\%=\active\def
\end{pgfscope}%
\begin{pgfscope}%
\definecolor{textcolor}{rgb}{0.000000,0.000000,0.000000}%
\pgfsetstrokecolor{textcolor}%
\pgfsetfillcolor{textcolor}%
\pgftext[x=1.967346in,y=0.331234in,,top]{\color{textcolor}{\sffamily\fontsize{10.000000}{12.000000}\selectfont\catcode`\^=\active\def^{\ifmmode\sp\else\^{}\fi}\catcode`\%=\active\def
\end{pgfscope}%
\begin{pgfscope}%
\pgfsetbuttcap%
\pgfsetroundjoin%
\definecolor{currentfill}{rgb}{0.000000,0.000000,0.000000}%
\pgfsetfillcolor{currentfill}%
\pgfsetlinewidth{0.803000pt}%
\definecolor{currentstroke}{rgb}{0.000000,0.000000,0.000000}%
\pgfsetstrokecolor{currentstroke}%
\pgfsetdash{}{0pt}%
\pgfsys@defobject{currentmarker}{\pgfqpoint{-0.048611in}{0.000000in}}{\pgfqpoint{-0.000000in}{0.000000in}}{%
\pgfpathmoveto{\pgfqpoint{-0.000000in}{0.000000in}}%
\pgfpathlineto{\pgfqpoint{-0.048611in}{0.000000in}}%
\pgfusepath{stroke,fill}%
}%
\begin{pgfscope}%
\pgfsys@transformshift{0.762486in}{0.623858in}%
\pgfsys@useobject{currentmarker}{}%
\end{pgfscope}%
\end{pgfscope}%
\begin{pgfscope}%
\definecolor{textcolor}{rgb}{0.000000,0.000000,0.000000}%
\pgfsetstrokecolor{textcolor}%
\pgfsetfillcolor{textcolor}%
\pgftext[x=0.485246in, y=0.585277in, left, base]{\color{textcolor}{\sffamily\fontsize{8.000000}{9.600000}\selectfont\catcode`\^=\active\def^{\ifmmode\sp\else\^{}\fi}\catcode`\%=\active\def
\end{pgfscope}%
\begin{pgfscope}%
\pgfsetbuttcap%
\pgfsetroundjoin%
\definecolor{currentfill}{rgb}{0.000000,0.000000,0.000000}%
\pgfsetfillcolor{currentfill}%
\pgfsetlinewidth{0.803000pt}%
\definecolor{currentstroke}{rgb}{0.000000,0.000000,0.000000}%
\pgfsetstrokecolor{currentstroke}%
\pgfsetdash{}{0pt}%
\pgfsys@defobject{currentmarker}{\pgfqpoint{-0.048611in}{0.000000in}}{\pgfqpoint{-0.000000in}{0.000000in}}{%
\pgfpathmoveto{\pgfqpoint{-0.000000in}{0.000000in}}%
\pgfpathlineto{\pgfqpoint{-0.048611in}{0.000000in}}%
\pgfusepath{stroke,fill}%
}%
\begin{pgfscope}%
\pgfsys@transformshift{0.762486in}{1.117506in}%
\pgfsys@useobject{currentmarker}{}%
\end{pgfscope}%
\end{pgfscope}%
\begin{pgfscope}%
\definecolor{textcolor}{rgb}{0.000000,0.000000,0.000000}%
\pgfsetstrokecolor{textcolor}%
\pgfsetfillcolor{textcolor}%
\pgftext[x=0.485246in, y=1.078925in, left, base]{\color{textcolor}{\sffamily\fontsize{8.000000}{9.600000}\selectfont\catcode`\^=\active\def^{\ifmmode\sp\else\^{}\fi}\catcode`\%=\active\def
\end{pgfscope}%
\begin{pgfscope}%
\pgfsetbuttcap%
\pgfsetroundjoin%
\definecolor{currentfill}{rgb}{0.000000,0.000000,0.000000}%
\pgfsetfillcolor{currentfill}%
\pgfsetlinewidth{0.803000pt}%
\definecolor{currentstroke}{rgb}{0.000000,0.000000,0.000000}%
\pgfsetstrokecolor{currentstroke}%
\pgfsetdash{}{0pt}%
\pgfsys@defobject{currentmarker}{\pgfqpoint{-0.048611in}{0.000000in}}{\pgfqpoint{-0.000000in}{0.000000in}}{%
\pgfpathmoveto{\pgfqpoint{-0.000000in}{0.000000in}}%
\pgfpathlineto{\pgfqpoint{-0.048611in}{0.000000in}}%
\pgfusepath{stroke,fill}%
}%
\begin{pgfscope}%
\pgfsys@transformshift{0.762486in}{1.611154in}%
\pgfsys@useobject{currentmarker}{}%
\end{pgfscope}%
\end{pgfscope}%
\begin{pgfscope}%
\definecolor{textcolor}{rgb}{0.000000,0.000000,0.000000}%
\pgfsetstrokecolor{textcolor}%
\pgfsetfillcolor{textcolor}%
\pgftext[x=0.485246in, y=1.572574in, left, base]{\color{textcolor}{\sffamily\fontsize{8.000000}{9.600000}\selectfont\catcode`\^=\active\def^{\ifmmode\sp\else\^{}\fi}\catcode`\%=\active\def
\end{pgfscope}%
\begin{pgfscope}%
\pgfsetbuttcap%
\pgfsetroundjoin%
\definecolor{currentfill}{rgb}{0.000000,0.000000,0.000000}%
\pgfsetfillcolor{currentfill}%
\pgfsetlinewidth{0.803000pt}%
\definecolor{currentstroke}{rgb}{0.000000,0.000000,0.000000}%
\pgfsetstrokecolor{currentstroke}%
\pgfsetdash{}{0pt}%
\pgfsys@defobject{currentmarker}{\pgfqpoint{-0.048611in}{0.000000in}}{\pgfqpoint{-0.000000in}{0.000000in}}{%
\pgfpathmoveto{\pgfqpoint{-0.000000in}{0.000000in}}%
\pgfpathlineto{\pgfqpoint{-0.048611in}{0.000000in}}%
\pgfusepath{stroke,fill}%
}%
\begin{pgfscope}%
\pgfsys@transformshift{0.762486in}{2.104802in}%
\pgfsys@useobject{currentmarker}{}%
\end{pgfscope}%
\end{pgfscope}%
\begin{pgfscope}%
\definecolor{textcolor}{rgb}{0.000000,0.000000,0.000000}%
\pgfsetstrokecolor{textcolor}%
\pgfsetfillcolor{textcolor}%
\pgftext[x=0.485246in, y=2.066222in, left, base]{\color{textcolor}{\sffamily\fontsize{8.000000}{9.600000}\selectfont\catcode`\^=\active\def^{\ifmmode\sp\else\^{}\fi}\catcode`\%=\active\def
\end{pgfscope}%
\begin{pgfscope}%
\pgfsetbuttcap%
\pgfsetroundjoin%
\definecolor{currentfill}{rgb}{0.000000,0.000000,0.000000}%
\pgfsetfillcolor{currentfill}%
\pgfsetlinewidth{0.803000pt}%
\definecolor{currentstroke}{rgb}{0.000000,0.000000,0.000000}%
\pgfsetstrokecolor{currentstroke}%
\pgfsetdash{}{0pt}%
\pgfsys@defobject{currentmarker}{\pgfqpoint{-0.048611in}{0.000000in}}{\pgfqpoint{-0.000000in}{0.000000in}}{%
\pgfpathmoveto{\pgfqpoint{-0.000000in}{0.000000in}}%
\pgfpathlineto{\pgfqpoint{-0.048611in}{0.000000in}}%
\pgfusepath{stroke,fill}%
}%
\begin{pgfscope}%
\pgfsys@transformshift{0.762486in}{2.598450in}%
\pgfsys@useobject{currentmarker}{}%
\end{pgfscope}%
\end{pgfscope}%
\begin{pgfscope}%
\definecolor{textcolor}{rgb}{0.000000,0.000000,0.000000}%
\pgfsetstrokecolor{textcolor}%
\pgfsetfillcolor{textcolor}%
\pgftext[x=0.485246in, y=2.559870in, left, base]{\color{textcolor}{\sffamily\fontsize{8.000000}{9.600000}\selectfont\catcode`\^=\active\def^{\ifmmode\sp\else\^{}\fi}\catcode`\%=\active\def
\end{pgfscope}%
\begin{pgfscope}%
\definecolor{textcolor}{rgb}{0.000000,0.000000,0.000000}%
\pgfsetstrokecolor{textcolor}%
\pgfsetfillcolor{textcolor}%
\pgftext[x=0.429691in,y=1.749306in,,bottom,rotate=90.000000]{\color{textcolor}{\sffamily\fontsize{10.000000}{12.000000}\selectfont\catcode`\^=\active\def^{\ifmmode\sp\else\^{}\fi}\catcode`\%=\active\def
\end{pgfscope}%
\begin{pgfscope}%
\pgfpathrectangle{\pgfqpoint{0.762486in}{0.611944in}}{\pgfqpoint{2.409719in}{2.274722in}}%
\pgfusepath{clip}%
\pgfsetrectcap%
\pgfsetroundjoin%
\pgfsetlinewidth{1.405250pt}%
\definecolor{currentstroke}{rgb}{0.121569,0.466667,0.705882}%
\pgfsetstrokecolor{currentstroke}%
\pgfsetdash{}{0pt}%
\pgfpathmoveto{\pgfqpoint{0.872019in}{2.783270in}}%
\pgfpathlineto{\pgfqpoint{0.987316in}{1.809100in}}%
\pgfpathlineto{\pgfqpoint{1.102614in}{1.439429in}}%
\pgfpathlineto{\pgfqpoint{1.217912in}{1.245362in}}%
\pgfpathlineto{\pgfqpoint{1.333209in}{1.125855in}}%
\pgfpathlineto{\pgfqpoint{1.448507in}{1.044880in}}%
\pgfpathlineto{\pgfqpoint{1.563804in}{0.986393in}}%
\pgfpathlineto{\pgfqpoint{1.679102in}{0.942171in}}%
\pgfpathlineto{\pgfqpoint{1.794399in}{0.907563in}}%
\pgfpathlineto{\pgfqpoint{1.909697in}{0.879741in}}%
\pgfpathlineto{\pgfqpoint{2.024995in}{0.856888in}}%
\pgfpathlineto{\pgfqpoint{2.140292in}{0.837782in}}%
\pgfpathlineto{\pgfqpoint{2.255590in}{0.821571in}}%
\pgfpathlineto{\pgfqpoint{2.370887in}{0.807644in}}%
\pgfpathlineto{\pgfqpoint{2.486185in}{0.795550in}}%
\pgfpathlineto{\pgfqpoint{2.601483in}{0.784949in}}%
\pgfpathlineto{\pgfqpoint{2.716780in}{0.775581in}}%
\pgfpathlineto{\pgfqpoint{2.832078in}{0.767243in}}%
\pgfpathlineto{\pgfqpoint{2.947375in}{0.759773in}}%
\pgfpathlineto{\pgfqpoint{3.062673in}{0.753043in}}%
\pgfusepath{stroke}%
\end{pgfscope}%
\begin{pgfscope}%
\pgfpathrectangle{\pgfqpoint{0.762486in}{0.611944in}}{\pgfqpoint{2.409719in}{2.274722in}}%
\pgfusepath{clip}%
\pgfsetrectcap%
\pgfsetroundjoin%
\pgfsetlinewidth{1.405250pt}%
\definecolor{currentstroke}{rgb}{1.000000,0.498039,0.054902}%
\pgfsetstrokecolor{currentstroke}%
\pgfsetdash{}{0pt}%
\pgfpathmoveto{\pgfqpoint{0.872019in}{2.492674in}}%
\pgfpathlineto{\pgfqpoint{0.987316in}{1.635066in}}%
\pgfpathlineto{\pgfqpoint{1.102614in}{1.316260in}}%
\pgfpathlineto{\pgfqpoint{1.217912in}{1.150187in}}%
\pgfpathlineto{\pgfqpoint{1.333209in}{1.048341in}}%
\pgfpathlineto{\pgfqpoint{1.448507in}{0.979511in}}%
\pgfpathlineto{\pgfqpoint{1.563804in}{0.929884in}}%
\pgfpathlineto{\pgfqpoint{1.679102in}{0.892409in}}%
\pgfpathlineto{\pgfqpoint{1.794399in}{0.863109in}}%
\pgfpathlineto{\pgfqpoint{1.909697in}{0.839574in}}%
\pgfpathlineto{\pgfqpoint{2.024995in}{0.820254in}}%
\pgfpathlineto{\pgfqpoint{2.140292in}{0.804110in}}%
\pgfpathlineto{\pgfqpoint{2.255590in}{0.790418in}}%
\pgfpathlineto{\pgfqpoint{2.370887in}{0.778659in}}%
\pgfpathlineto{\pgfqpoint{2.486185in}{0.768451in}}%
\pgfpathlineto{\pgfqpoint{2.601483in}{0.759506in}}%
\pgfpathlineto{\pgfqpoint{2.716780in}{0.751603in}}%
\pgfpathlineto{\pgfqpoint{2.832078in}{0.744571in}}%
\pgfpathlineto{\pgfqpoint{2.947375in}{0.738272in}}%
\pgfpathlineto{\pgfqpoint{3.062673in}{0.732598in}}%
\pgfusepath{stroke}%
\end{pgfscope}%
\begin{pgfscope}%
\pgfpathrectangle{\pgfqpoint{0.762486in}{0.611944in}}{\pgfqpoint{2.409719in}{2.274722in}}%
\pgfusepath{clip}%
\pgfsetrectcap%
\pgfsetroundjoin%
\pgfsetlinewidth{1.405250pt}%
\definecolor{currentstroke}{rgb}{0.172549,0.627451,0.172549}%
\pgfsetstrokecolor{currentstroke}%
\pgfsetdash{}{0pt}%
\pgfpathmoveto{\pgfqpoint{0.872019in}{2.337329in}}%
\pgfpathlineto{\pgfqpoint{0.987316in}{1.544230in}}%
\pgfpathlineto{\pgfqpoint{1.102614in}{1.252476in}}%
\pgfpathlineto{\pgfqpoint{1.217912in}{1.101093in}}%
\pgfpathlineto{\pgfqpoint{1.333209in}{1.008451in}}%
\pgfpathlineto{\pgfqpoint{1.448507in}{0.945923in}}%
\pgfpathlineto{\pgfqpoint{1.563804in}{0.900880in}}%
\pgfpathlineto{\pgfqpoint{1.679102in}{0.866890in}}%
\pgfpathlineto{\pgfqpoint{1.794399in}{0.840328in}}%
\pgfpathlineto{\pgfqpoint{1.909697in}{0.819000in}}%
\pgfpathlineto{\pgfqpoint{2.024995in}{0.801497in}}%
\pgfpathlineto{\pgfqpoint{2.140292in}{0.786875in}}%
\pgfpathlineto{\pgfqpoint{2.255590in}{0.774478in}}%
\pgfpathlineto{\pgfqpoint{2.370887in}{0.763832in}}%
\pgfpathlineto{\pgfqpoint{2.486185in}{0.754592in}}%
\pgfpathlineto{\pgfqpoint{2.601483in}{0.746497in}}%
\pgfpathlineto{\pgfqpoint{2.716780in}{0.739345in}}%
\pgfpathlineto{\pgfqpoint{2.832078in}{0.732982in}}%
\pgfpathlineto{\pgfqpoint{2.947375in}{0.727283in}}%
\pgfpathlineto{\pgfqpoint{3.062673in}{0.722150in}}%
\pgfusepath{stroke}%
\end{pgfscope}%
\begin{pgfscope}%
\pgfpathrectangle{\pgfqpoint{0.762486in}{0.611944in}}{\pgfqpoint{2.409719in}{2.274722in}}%
\pgfusepath{clip}%
\pgfsetrectcap%
\pgfsetroundjoin%
\pgfsetlinewidth{1.405250pt}%
\definecolor{currentstroke}{rgb}{0.839216,0.152941,0.156863}%
\pgfsetstrokecolor{currentstroke}%
\pgfsetdash{}{0pt}%
\pgfpathmoveto{\pgfqpoint{0.872019in}{2.233516in}}%
\pgfpathlineto{\pgfqpoint{0.987316in}{1.484329in}}%
\pgfpathlineto{\pgfqpoint{1.102614in}{1.210598in}}%
\pgfpathlineto{\pgfqpoint{1.217912in}{1.068930in}}%
\pgfpathlineto{\pgfqpoint{1.333209in}{0.982352in}}%
\pgfpathlineto{\pgfqpoint{1.448507in}{0.923966in}}%
\pgfpathlineto{\pgfqpoint{1.563804in}{0.881933in}}%
\pgfpathlineto{\pgfqpoint{1.679102in}{0.850226in}}%
\pgfpathlineto{\pgfqpoint{1.794399in}{0.825457in}}%
\pgfpathlineto{\pgfqpoint{1.909697in}{0.805574in}}%
\pgfpathlineto{\pgfqpoint{2.024995in}{0.789260in}}%
\pgfpathlineto{\pgfqpoint{2.140292in}{0.775634in}}%
\pgfpathlineto{\pgfqpoint{2.255590in}{0.764082in}}%
\pgfpathlineto{\pgfqpoint{2.370887in}{0.754164in}}%
\pgfpathlineto{\pgfqpoint{2.486185in}{0.745557in}}%
\pgfpathlineto{\pgfqpoint{2.601483in}{0.738016in}}%
\pgfpathlineto{\pgfqpoint{2.716780in}{0.731355in}}%
\pgfpathlineto{\pgfqpoint{2.832078in}{0.725428in}}%
\pgfpathlineto{\pgfqpoint{2.947375in}{0.720121in}}%
\pgfpathlineto{\pgfqpoint{3.062673in}{0.715341in}}%
\pgfusepath{stroke}%
\end{pgfscope}%
\begin{pgfscope}%
\pgfsetrectcap%
\pgfsetmiterjoin%
\pgfsetlinewidth{0.803000pt}%
\definecolor{currentstroke}{rgb}{0.000000,0.000000,0.000000}%
\pgfsetstrokecolor{currentstroke}%
\pgfsetdash{}{0pt}%
\pgfpathmoveto{\pgfqpoint{0.762486in}{0.611944in}}%
\pgfpathlineto{\pgfqpoint{0.762486in}{2.886667in}}%
\pgfusepath{stroke}%
\end{pgfscope}%
\begin{pgfscope}%
\pgfsetrectcap%
\pgfsetmiterjoin%
\pgfsetlinewidth{0.803000pt}%
\definecolor{currentstroke}{rgb}{0.000000,0.000000,0.000000}%
\pgfsetstrokecolor{currentstroke}%
\pgfsetdash{}{0pt}%
\pgfpathmoveto{\pgfqpoint{3.172206in}{0.611944in}}%
\pgfpathlineto{\pgfqpoint{3.172206in}{2.886667in}}%
\pgfusepath{stroke}%
\end{pgfscope}%
\begin{pgfscope}%
\pgfsetrectcap%
\pgfsetmiterjoin%
\pgfsetlinewidth{0.803000pt}%
\definecolor{currentstroke}{rgb}{0.000000,0.000000,0.000000}%
\pgfsetstrokecolor{currentstroke}%
\pgfsetdash{}{0pt}%
\pgfpathmoveto{\pgfqpoint{0.762486in}{0.611944in}}%
\pgfpathlineto{\pgfqpoint{3.172206in}{0.611944in}}%
\pgfusepath{stroke}%
\end{pgfscope}%
\begin{pgfscope}%
\pgfsetrectcap%
\pgfsetmiterjoin%
\pgfsetlinewidth{0.803000pt}%
\definecolor{currentstroke}{rgb}{0.000000,0.000000,0.000000}%
\pgfsetstrokecolor{currentstroke}%
\pgfsetdash{}{0pt}%
\pgfpathmoveto{\pgfqpoint{0.762486in}{2.886667in}}%
\pgfpathlineto{\pgfqpoint{3.172206in}{2.886667in}}%
\pgfusepath{stroke}%
\end{pgfscope}%
\begin{pgfscope}%
\definecolor{textcolor}{rgb}{0.000000,0.000000,0.000000}%
\pgfsetstrokecolor{textcolor}%
\pgfsetfillcolor{textcolor}%
\pgftext[x=1.967346in,y=2.970000in,,base]{\color{textcolor}{\sffamily\fontsize{10.000000}{12.000000}\selectfont\catcode`\^=\active\def^{\ifmmode\sp\else\^{}\fi}\catcode`\%=\active\def
\end{pgfscope}%
\begin{pgfscope}%
\pgfsetbuttcap%
\pgfsetmiterjoin%
\definecolor{currentfill}{rgb}{1.000000,1.000000,1.000000}%
\pgfsetfillcolor{currentfill}%
\pgfsetfillopacity{0.800000}%
\pgfsetlinewidth{0.240900pt}%
\definecolor{currentstroke}{rgb}{0.800000,0.800000,0.800000}%
\pgfsetstrokecolor{currentstroke}%
\pgfsetstrokeopacity{0.800000}%
\pgfsetdash{}{0pt}%
\pgfpathmoveto{\pgfqpoint{2.203341in}{2.121234in}}%
\pgfpathlineto{\pgfqpoint{3.094428in}{2.121234in}}%
\pgfpathquadraticcurveto{\pgfqpoint{3.116650in}{2.121234in}}{\pgfqpoint{3.116650in}{2.143456in}}%
\pgfpathlineto{\pgfqpoint{3.116650in}{2.808889in}}%
\pgfpathquadraticcurveto{\pgfqpoint{3.116650in}{2.831111in}}{\pgfqpoint{3.094428in}{2.831111in}}%
\pgfpathlineto{\pgfqpoint{2.203341in}{2.831111in}}%
\pgfpathquadraticcurveto{\pgfqpoint{2.181118in}{2.831111in}}{\pgfqpoint{2.181118in}{2.808889in}}%
\pgfpathlineto{\pgfqpoint{2.181118in}{2.143456in}}%
\pgfpathquadraticcurveto{\pgfqpoint{2.181118in}{2.121234in}}{\pgfqpoint{2.203341in}{2.121234in}}%
\pgfpathlineto{\pgfqpoint{2.203341in}{2.121234in}}%
\pgfpathclose%
\pgfusepath{stroke,fill}%
\end{pgfscope}%
\begin{pgfscope}%
\pgfsetrectcap%
\pgfsetroundjoin%
\pgfsetlinewidth{1.405250pt}%
\definecolor{currentstroke}{rgb}{0.121569,0.466667,0.705882}%
\pgfsetstrokecolor{currentstroke}%
\pgfsetdash{}{0pt}%
\pgfpathmoveto{\pgfqpoint{2.225563in}{2.747778in}}%
\pgfpathlineto{\pgfqpoint{2.336674in}{2.747778in}}%
\pgfpathlineto{\pgfqpoint{2.447785in}{2.747778in}}%
\pgfusepath{stroke}%
\end{pgfscope}%
\begin{pgfscope}%
\definecolor{textcolor}{rgb}{0.000000,0.000000,0.000000}%
\pgfsetstrokecolor{textcolor}%
\pgfsetfillcolor{textcolor}%
\pgftext[x=2.536674in,y=2.708889in,left,base]{\color{textcolor}{\sffamily\fontsize{8.000000}{9.600000}\selectfont\catcode`\^=\active\def^{\ifmmode\sp\else\^{}\fi}\catcode`\%=\active\def
\end{pgfscope}%
\begin{pgfscope}%
\pgfsetrectcap%
\pgfsetroundjoin%
\pgfsetlinewidth{1.405250pt}%
\definecolor{currentstroke}{rgb}{1.000000,0.498039,0.054902}%
\pgfsetstrokecolor{currentstroke}%
\pgfsetdash{}{0pt}%
\pgfpathmoveto{\pgfqpoint{2.225563in}{2.578642in}}%
\pgfpathlineto{\pgfqpoint{2.336674in}{2.578642in}}%
\pgfpathlineto{\pgfqpoint{2.447785in}{2.578642in}}%
\pgfusepath{stroke}%
\end{pgfscope}%
\begin{pgfscope}%
\definecolor{textcolor}{rgb}{0.000000,0.000000,0.000000}%
\pgfsetstrokecolor{textcolor}%
\pgfsetfillcolor{textcolor}%
\pgftext[x=2.536674in,y=2.539753in,left,base]{\color{textcolor}{\sffamily\fontsize{8.000000}{9.600000}\selectfont\catcode`\^=\active\def^{\ifmmode\sp\else\^{}\fi}\catcode`\%=\active\def
\end{pgfscope}%
\begin{pgfscope}%
\pgfsetrectcap%
\pgfsetroundjoin%
\pgfsetlinewidth{1.405250pt}%
\definecolor{currentstroke}{rgb}{0.172549,0.627451,0.172549}%
\pgfsetstrokecolor{currentstroke}%
\pgfsetdash{}{0pt}%
\pgfpathmoveto{\pgfqpoint{2.225563in}{2.409506in}}%
\pgfpathlineto{\pgfqpoint{2.336674in}{2.409506in}}%
\pgfpathlineto{\pgfqpoint{2.447785in}{2.409506in}}%
\pgfusepath{stroke}%
\end{pgfscope}%
\begin{pgfscope}%
\definecolor{textcolor}{rgb}{0.000000,0.000000,0.000000}%
\pgfsetstrokecolor{textcolor}%
\pgfsetfillcolor{textcolor}%
\pgftext[x=2.536674in,y=2.370617in,left,base]{\color{textcolor}{\sffamily\fontsize{8.000000}{9.600000}\selectfont\catcode`\^=\active\def^{\ifmmode\sp\else\^{}\fi}\catcode`\%=\active\def
\end{pgfscope}%
\begin{pgfscope}%
\pgfsetrectcap%
\pgfsetroundjoin%
\pgfsetlinewidth{1.405250pt}%
\definecolor{currentstroke}{rgb}{0.839216,0.152941,0.156863}%
\pgfsetstrokecolor{currentstroke}%
\pgfsetdash{}{0pt}%
\pgfpathmoveto{\pgfqpoint{2.225563in}{2.240370in}}%
\pgfpathlineto{\pgfqpoint{2.336674in}{2.240370in}}%
\pgfpathlineto{\pgfqpoint{2.447785in}{2.240370in}}%
\pgfusepath{stroke}%
\end{pgfscope}%
\begin{pgfscope}%
\definecolor{textcolor}{rgb}{0.000000,0.000000,0.000000}%
\pgfsetstrokecolor{textcolor}%
\pgfsetfillcolor{textcolor}%
\pgftext[x=2.536674in,y=2.201481in,left,base]{\color{textcolor}{\sffamily\fontsize{8.000000}{9.600000}\selectfont\catcode`\^=\active\def^{\ifmmode\sp\else\^{}\fi}\catcode`\%=\active\def
\end{pgfscope}%
\begin{pgfscope}%
\pgfsetbuttcap%
\pgfsetmiterjoin%
\definecolor{currentfill}{rgb}{1.000000,1.000000,1.000000}%
\pgfsetfillcolor{currentfill}%
\pgfsetlinewidth{0.000000pt}%
\definecolor{currentstroke}{rgb}{0.000000,0.000000,0.000000}%
\pgfsetstrokecolor{currentstroke}%
\pgfsetstrokeopacity{0.000000}%
\pgfsetdash{}{0pt}%
\pgfpathmoveto{\pgfqpoint{3.937486in}{0.611944in}}%
\pgfpathlineto{\pgfqpoint{6.347206in}{0.611944in}}%
\pgfpathlineto{\pgfqpoint{6.347206in}{2.886667in}}%
\pgfpathlineto{\pgfqpoint{3.937486in}{2.886667in}}%
\pgfpathlineto{\pgfqpoint{3.937486in}{0.611944in}}%
\pgfpathclose%
\pgfusepath{fill}%
\end{pgfscope}%
\begin{pgfscope}%
\pgfsetbuttcap%
\pgfsetroundjoin%
\definecolor{currentfill}{rgb}{0.000000,0.000000,0.000000}%
\pgfsetfillcolor{currentfill}%
\pgfsetlinewidth{0.803000pt}%
\definecolor{currentstroke}{rgb}{0.000000,0.000000,0.000000}%
\pgfsetstrokecolor{currentstroke}%
\pgfsetdash{}{0pt}%
\pgfsys@defobject{currentmarker}{\pgfqpoint{0.000000in}{-0.048611in}}{\pgfqpoint{0.000000in}{0.000000in}}{%
\pgfpathmoveto{\pgfqpoint{0.000000in}{0.000000in}}%
\pgfpathlineto{\pgfqpoint{0.000000in}{-0.048611in}}%
\pgfusepath{stroke,fill}%
}%
\begin{pgfscope}%
\pgfsys@transformshift{4.508209in}{0.611944in}%
\pgfsys@useobject{currentmarker}{}%
\end{pgfscope}%
\end{pgfscope}%
\begin{pgfscope}%
\definecolor{textcolor}{rgb}{0.000000,0.000000,0.000000}%
\pgfsetstrokecolor{textcolor}%
\pgfsetfillcolor{textcolor}%
\pgftext[x=4.508209in,y=0.485556in,,top]{\color{textcolor}{\sffamily\fontsize{8.000000}{9.600000}\selectfont\catcode`\^=\active\def^{\ifmmode\sp\else\^{}\fi}\catcode`\%=\active\def
\end{pgfscope}%
\begin{pgfscope}%
\pgfsetbuttcap%
\pgfsetroundjoin%
\definecolor{currentfill}{rgb}{0.000000,0.000000,0.000000}%
\pgfsetfillcolor{currentfill}%
\pgfsetlinewidth{0.803000pt}%
\definecolor{currentstroke}{rgb}{0.000000,0.000000,0.000000}%
\pgfsetstrokecolor{currentstroke}%
\pgfsetdash{}{0pt}%
\pgfsys@defobject{currentmarker}{\pgfqpoint{0.000000in}{-0.048611in}}{\pgfqpoint{0.000000in}{0.000000in}}{%
\pgfpathmoveto{\pgfqpoint{0.000000in}{0.000000in}}%
\pgfpathlineto{\pgfqpoint{0.000000in}{-0.048611in}}%
\pgfusepath{stroke,fill}%
}%
\begin{pgfscope}%
\pgfsys@transformshift{5.084697in}{0.611944in}%
\pgfsys@useobject{currentmarker}{}%
\end{pgfscope}%
\end{pgfscope}%
\begin{pgfscope}%
\definecolor{textcolor}{rgb}{0.000000,0.000000,0.000000}%
\pgfsetstrokecolor{textcolor}%
\pgfsetfillcolor{textcolor}%
\pgftext[x=5.084697in,y=0.485556in,,top]{\color{textcolor}{\sffamily\fontsize{8.000000}{9.600000}\selectfont\catcode`\^=\active\def^{\ifmmode\sp\else\^{}\fi}\catcode`\%=\active\def
\end{pgfscope}%
\begin{pgfscope}%
\pgfsetbuttcap%
\pgfsetroundjoin%
\definecolor{currentfill}{rgb}{0.000000,0.000000,0.000000}%
\pgfsetfillcolor{currentfill}%
\pgfsetlinewidth{0.803000pt}%
\definecolor{currentstroke}{rgb}{0.000000,0.000000,0.000000}%
\pgfsetstrokecolor{currentstroke}%
\pgfsetdash{}{0pt}%
\pgfsys@defobject{currentmarker}{\pgfqpoint{0.000000in}{-0.048611in}}{\pgfqpoint{0.000000in}{0.000000in}}{%
\pgfpathmoveto{\pgfqpoint{0.000000in}{0.000000in}}%
\pgfpathlineto{\pgfqpoint{0.000000in}{-0.048611in}}%
\pgfusepath{stroke,fill}%
}%
\begin{pgfscope}%
\pgfsys@transformshift{5.661185in}{0.611944in}%
\pgfsys@useobject{currentmarker}{}%
\end{pgfscope}%
\end{pgfscope}%
\begin{pgfscope}%
\definecolor{textcolor}{rgb}{0.000000,0.000000,0.000000}%
\pgfsetstrokecolor{textcolor}%
\pgfsetfillcolor{textcolor}%
\pgftext[x=5.661185in,y=0.485556in,,top]{\color{textcolor}{\sffamily\fontsize{8.000000}{9.600000}\selectfont\catcode`\^=\active\def^{\ifmmode\sp\else\^{}\fi}\catcode`\%=\active\def
\end{pgfscope}%
\begin{pgfscope}%
\pgfsetbuttcap%
\pgfsetroundjoin%
\definecolor{currentfill}{rgb}{0.000000,0.000000,0.000000}%
\pgfsetfillcolor{currentfill}%
\pgfsetlinewidth{0.803000pt}%
\definecolor{currentstroke}{rgb}{0.000000,0.000000,0.000000}%
\pgfsetstrokecolor{currentstroke}%
\pgfsetdash{}{0pt}%
\pgfsys@defobject{currentmarker}{\pgfqpoint{0.000000in}{-0.048611in}}{\pgfqpoint{0.000000in}{0.000000in}}{%
\pgfpathmoveto{\pgfqpoint{0.000000in}{0.000000in}}%
\pgfpathlineto{\pgfqpoint{0.000000in}{-0.048611in}}%
\pgfusepath{stroke,fill}%
}%
\begin{pgfscope}%
\pgfsys@transformshift{6.237673in}{0.611944in}%
\pgfsys@useobject{currentmarker}{}%
\end{pgfscope}%
\end{pgfscope}%
\begin{pgfscope}%
\definecolor{textcolor}{rgb}{0.000000,0.000000,0.000000}%
\pgfsetstrokecolor{textcolor}%
\pgfsetfillcolor{textcolor}%
\pgftext[x=6.237673in,y=0.485556in,,top]{\color{textcolor}{\sffamily\fontsize{8.000000}{9.600000}\selectfont\catcode`\^=\active\def^{\ifmmode\sp\else\^{}\fi}\catcode`\%=\active\def
\end{pgfscope}%
\begin{pgfscope}%
\definecolor{textcolor}{rgb}{0.000000,0.000000,0.000000}%
\pgfsetstrokecolor{textcolor}%
\pgfsetfillcolor{textcolor}%
\pgftext[x=5.142346in,y=0.331234in,,top]{\color{textcolor}{\sffamily\fontsize{10.000000}{12.000000}\selectfont\catcode`\^=\active\def^{\ifmmode\sp\else\^{}\fi}\catcode`\%=\active\def
\end{pgfscope}%
\begin{pgfscope}%
\pgfsetbuttcap%
\pgfsetroundjoin%
\definecolor{currentfill}{rgb}{0.000000,0.000000,0.000000}%
\pgfsetfillcolor{currentfill}%
\pgfsetlinewidth{0.803000pt}%
\definecolor{currentstroke}{rgb}{0.000000,0.000000,0.000000}%
\pgfsetstrokecolor{currentstroke}%
\pgfsetdash{}{0pt}%
\pgfsys@defobject{currentmarker}{\pgfqpoint{-0.048611in}{0.000000in}}{\pgfqpoint{-0.000000in}{0.000000in}}{%
\pgfpathmoveto{\pgfqpoint{-0.000000in}{0.000000in}}%
\pgfpathlineto{\pgfqpoint{-0.048611in}{0.000000in}}%
\pgfusepath{stroke,fill}%
}%
\begin{pgfscope}%
\pgfsys@transformshift{3.937486in}{0.667106in}%
\pgfsys@useobject{currentmarker}{}%
\end{pgfscope}%
\end{pgfscope}%
\begin{pgfscope}%
\definecolor{textcolor}{rgb}{0.000000,0.000000,0.000000}%
\pgfsetstrokecolor{textcolor}%
\pgfsetfillcolor{textcolor}%
\pgftext[x=3.660246in, y=0.628526in, left, base]{\color{textcolor}{\sffamily\fontsize{8.000000}{9.600000}\selectfont\catcode`\^=\active\def^{\ifmmode\sp\else\^{}\fi}\catcode`\%=\active\def
\end{pgfscope}%
\begin{pgfscope}%
\pgfsetbuttcap%
\pgfsetroundjoin%
\definecolor{currentfill}{rgb}{0.000000,0.000000,0.000000}%
\pgfsetfillcolor{currentfill}%
\pgfsetlinewidth{0.803000pt}%
\definecolor{currentstroke}{rgb}{0.000000,0.000000,0.000000}%
\pgfsetstrokecolor{currentstroke}%
\pgfsetdash{}{0pt}%
\pgfsys@defobject{currentmarker}{\pgfqpoint{-0.048611in}{0.000000in}}{\pgfqpoint{-0.000000in}{0.000000in}}{%
\pgfpathmoveto{\pgfqpoint{-0.000000in}{0.000000in}}%
\pgfpathlineto{\pgfqpoint{-0.048611in}{0.000000in}}%
\pgfusepath{stroke,fill}%
}%
\begin{pgfscope}%
\pgfsys@transformshift{3.937486in}{1.035647in}%
\pgfsys@useobject{currentmarker}{}%
\end{pgfscope}%
\end{pgfscope}%
\begin{pgfscope}%
\definecolor{textcolor}{rgb}{0.000000,0.000000,0.000000}%
\pgfsetstrokecolor{textcolor}%
\pgfsetfillcolor{textcolor}%
\pgftext[x=3.660246in, y=0.997066in, left, base]{\color{textcolor}{\sffamily\fontsize{8.000000}{9.600000}\selectfont\catcode`\^=\active\def^{\ifmmode\sp\else\^{}\fi}\catcode`\%=\active\def
\end{pgfscope}%
\begin{pgfscope}%
\pgfsetbuttcap%
\pgfsetroundjoin%
\definecolor{currentfill}{rgb}{0.000000,0.000000,0.000000}%
\pgfsetfillcolor{currentfill}%
\pgfsetlinewidth{0.803000pt}%
\definecolor{currentstroke}{rgb}{0.000000,0.000000,0.000000}%
\pgfsetstrokecolor{currentstroke}%
\pgfsetdash{}{0pt}%
\pgfsys@defobject{currentmarker}{\pgfqpoint{-0.048611in}{0.000000in}}{\pgfqpoint{-0.000000in}{0.000000in}}{%
\pgfpathmoveto{\pgfqpoint{-0.000000in}{0.000000in}}%
\pgfpathlineto{\pgfqpoint{-0.048611in}{0.000000in}}%
\pgfusepath{stroke,fill}%
}%
\begin{pgfscope}%
\pgfsys@transformshift{3.937486in}{1.404187in}%
\pgfsys@useobject{currentmarker}{}%
\end{pgfscope}%
\end{pgfscope}%
\begin{pgfscope}%
\definecolor{textcolor}{rgb}{0.000000,0.000000,0.000000}%
\pgfsetstrokecolor{textcolor}%
\pgfsetfillcolor{textcolor}%
\pgftext[x=3.660246in, y=1.365607in, left, base]{\color{textcolor}{\sffamily\fontsize{8.000000}{9.600000}\selectfont\catcode`\^=\active\def^{\ifmmode\sp\else\^{}\fi}\catcode`\%=\active\def
\end{pgfscope}%
\begin{pgfscope}%
\pgfsetbuttcap%
\pgfsetroundjoin%
\definecolor{currentfill}{rgb}{0.000000,0.000000,0.000000}%
\pgfsetfillcolor{currentfill}%
\pgfsetlinewidth{0.803000pt}%
\definecolor{currentstroke}{rgb}{0.000000,0.000000,0.000000}%
\pgfsetstrokecolor{currentstroke}%
\pgfsetdash{}{0pt}%
\pgfsys@defobject{currentmarker}{\pgfqpoint{-0.048611in}{0.000000in}}{\pgfqpoint{-0.000000in}{0.000000in}}{%
\pgfpathmoveto{\pgfqpoint{-0.000000in}{0.000000in}}%
\pgfpathlineto{\pgfqpoint{-0.048611in}{0.000000in}}%
\pgfusepath{stroke,fill}%
}%
\begin{pgfscope}%
\pgfsys@transformshift{3.937486in}{1.772728in}%
\pgfsys@useobject{currentmarker}{}%
\end{pgfscope}%
\end{pgfscope}%
\begin{pgfscope}%
\definecolor{textcolor}{rgb}{0.000000,0.000000,0.000000}%
\pgfsetstrokecolor{textcolor}%
\pgfsetfillcolor{textcolor}%
\pgftext[x=3.660246in, y=1.734148in, left, base]{\color{textcolor}{\sffamily\fontsize{8.000000}{9.600000}\selectfont\catcode`\^=\active\def^{\ifmmode\sp\else\^{}\fi}\catcode`\%=\active\def
\end{pgfscope}%
\begin{pgfscope}%
\pgfsetbuttcap%
\pgfsetroundjoin%
\definecolor{currentfill}{rgb}{0.000000,0.000000,0.000000}%
\pgfsetfillcolor{currentfill}%
\pgfsetlinewidth{0.803000pt}%
\definecolor{currentstroke}{rgb}{0.000000,0.000000,0.000000}%
\pgfsetstrokecolor{currentstroke}%
\pgfsetdash{}{0pt}%
\pgfsys@defobject{currentmarker}{\pgfqpoint{-0.048611in}{0.000000in}}{\pgfqpoint{-0.000000in}{0.000000in}}{%
\pgfpathmoveto{\pgfqpoint{-0.000000in}{0.000000in}}%
\pgfpathlineto{\pgfqpoint{-0.048611in}{0.000000in}}%
\pgfusepath{stroke,fill}%
}%
\begin{pgfscope}%
\pgfsys@transformshift{3.937486in}{2.141269in}%
\pgfsys@useobject{currentmarker}{}%
\end{pgfscope}%
\end{pgfscope}%
\begin{pgfscope}%
\definecolor{textcolor}{rgb}{0.000000,0.000000,0.000000}%
\pgfsetstrokecolor{textcolor}%
\pgfsetfillcolor{textcolor}%
\pgftext[x=3.660246in, y=2.102688in, left, base]{\color{textcolor}{\sffamily\fontsize{8.000000}{9.600000}\selectfont\catcode`\^=\active\def^{\ifmmode\sp\else\^{}\fi}\catcode`\%=\active\def
\end{pgfscope}%
\begin{pgfscope}%
\pgfsetbuttcap%
\pgfsetroundjoin%
\definecolor{currentfill}{rgb}{0.000000,0.000000,0.000000}%
\pgfsetfillcolor{currentfill}%
\pgfsetlinewidth{0.803000pt}%
\definecolor{currentstroke}{rgb}{0.000000,0.000000,0.000000}%
\pgfsetstrokecolor{currentstroke}%
\pgfsetdash{}{0pt}%
\pgfsys@defobject{currentmarker}{\pgfqpoint{-0.048611in}{0.000000in}}{\pgfqpoint{-0.000000in}{0.000000in}}{%
\pgfpathmoveto{\pgfqpoint{-0.000000in}{0.000000in}}%
\pgfpathlineto{\pgfqpoint{-0.048611in}{0.000000in}}%
\pgfusepath{stroke,fill}%
}%
\begin{pgfscope}%
\pgfsys@transformshift{3.937486in}{2.509809in}%
\pgfsys@useobject{currentmarker}{}%
\end{pgfscope}%
\end{pgfscope}%
\begin{pgfscope}%
\definecolor{textcolor}{rgb}{0.000000,0.000000,0.000000}%
\pgfsetstrokecolor{textcolor}%
\pgfsetfillcolor{textcolor}%
\pgftext[x=3.660246in, y=2.471229in, left, base]{\color{textcolor}{\sffamily\fontsize{8.000000}{9.600000}\selectfont\catcode`\^=\active\def^{\ifmmode\sp\else\^{}\fi}\catcode`\%=\active\def
\end{pgfscope}%
\begin{pgfscope}%
\pgfsetbuttcap%
\pgfsetroundjoin%
\definecolor{currentfill}{rgb}{0.000000,0.000000,0.000000}%
\pgfsetfillcolor{currentfill}%
\pgfsetlinewidth{0.803000pt}%
\definecolor{currentstroke}{rgb}{0.000000,0.000000,0.000000}%
\pgfsetstrokecolor{currentstroke}%
\pgfsetdash{}{0pt}%
\pgfsys@defobject{currentmarker}{\pgfqpoint{-0.048611in}{0.000000in}}{\pgfqpoint{-0.000000in}{0.000000in}}{%
\pgfpathmoveto{\pgfqpoint{-0.000000in}{0.000000in}}%
\pgfpathlineto{\pgfqpoint{-0.048611in}{0.000000in}}%
\pgfusepath{stroke,fill}%
}%
\begin{pgfscope}%
\pgfsys@transformshift{3.937486in}{2.878350in}%
\pgfsys@useobject{currentmarker}{}%
\end{pgfscope}%
\end{pgfscope}%
\begin{pgfscope}%
\definecolor{textcolor}{rgb}{0.000000,0.000000,0.000000}%
\pgfsetstrokecolor{textcolor}%
\pgfsetfillcolor{textcolor}%
\pgftext[x=3.660246in, y=2.839770in, left, base]{\color{textcolor}{\sffamily\fontsize{8.000000}{9.600000}\selectfont\catcode`\^=\active\def^{\ifmmode\sp\else\^{}\fi}\catcode`\%=\active\def
\end{pgfscope}%
\begin{pgfscope}%
\definecolor{textcolor}{rgb}{0.000000,0.000000,0.000000}%
\pgfsetstrokecolor{textcolor}%
\pgfsetfillcolor{textcolor}%
\pgftext[x=3.604691in,y=1.749306in,,bottom,rotate=90.000000]{\color{textcolor}{\sffamily\fontsize{10.000000}{12.000000}\selectfont\catcode`\^=\active\def^{\ifmmode\sp\else\^{}\fi}\catcode`\%=\active\def
\end{pgfscope}%
\begin{pgfscope}%
\pgfpathrectangle{\pgfqpoint{3.937486in}{0.611944in}}{\pgfqpoint{2.409719in}{2.274722in}}%
\pgfusepath{clip}%
\pgfsetrectcap%
\pgfsetroundjoin%
\pgfsetlinewidth{1.405250pt}%
\definecolor{currentstroke}{rgb}{0.121569,0.466667,0.705882}%
\pgfsetstrokecolor{currentstroke}%
\pgfsetdash{}{0pt}%
\pgfpathmoveto{\pgfqpoint{4.047019in}{1.547550in}}%
\pgfpathlineto{\pgfqpoint{4.162316in}{1.128892in}}%
\pgfpathlineto{\pgfqpoint{4.277614in}{0.979953in}}%
\pgfpathlineto{\pgfqpoint{4.392912in}{0.903641in}}%
\pgfpathlineto{\pgfqpoint{4.508209in}{0.857253in}}%
\pgfpathlineto{\pgfqpoint{4.623507in}{0.826076in}}%
\pgfpathlineto{\pgfqpoint{4.738804in}{0.803681in}}%
\pgfpathlineto{\pgfqpoint{4.854102in}{0.786817in}}%
\pgfpathlineto{\pgfqpoint{4.969399in}{0.773659in}}%
\pgfpathlineto{\pgfqpoint{5.084697in}{0.763108in}}%
\pgfpathlineto{\pgfqpoint{5.199995in}{0.754457in}}%
\pgfpathlineto{\pgfqpoint{5.315292in}{0.747237in}}%
\pgfpathlineto{\pgfqpoint{5.430590in}{0.741119in}}%
\pgfpathlineto{\pgfqpoint{5.545887in}{0.735869in}}%
\pgfpathlineto{\pgfqpoint{5.661185in}{0.731315in}}%
\pgfpathlineto{\pgfqpoint{5.776483in}{0.727326in}}%
\pgfpathlineto{\pgfqpoint{5.891780in}{0.723804in}}%
\pgfpathlineto{\pgfqpoint{6.007078in}{0.720671in}}%
\pgfpathlineto{\pgfqpoint{6.122375in}{0.717867in}}%
\pgfpathlineto{\pgfqpoint{6.237673in}{0.715341in}}%
\pgfusepath{stroke}%
\end{pgfscope}%
\begin{pgfscope}%
\pgfpathrectangle{\pgfqpoint{3.937486in}{0.611944in}}{\pgfqpoint{2.409719in}{2.274722in}}%
\pgfusepath{clip}%
\pgfsetrectcap%
\pgfsetroundjoin%
\pgfsetlinewidth{1.405250pt}%
\definecolor{currentstroke}{rgb}{1.000000,0.498039,0.054902}%
\pgfsetstrokecolor{currentstroke}%
\pgfsetdash{}{0pt}%
\pgfpathmoveto{\pgfqpoint{4.047019in}{2.062300in}}%
\pgfpathlineto{\pgfqpoint{4.162316in}{1.422039in}}%
\pgfpathlineto{\pgfqpoint{4.277614in}{1.184030in}}%
\pgfpathlineto{\pgfqpoint{4.392912in}{1.060045in}}%
\pgfpathlineto{\pgfqpoint{4.508209in}{0.984011in}}%
\pgfpathlineto{\pgfqpoint{4.623507in}{0.932624in}}%
\pgfpathlineto{\pgfqpoint{4.738804in}{0.895575in}}%
\pgfpathlineto{\pgfqpoint{4.854102in}{0.867597in}}%
\pgfpathlineto{\pgfqpoint{4.969399in}{0.845723in}}%
\pgfpathlineto{\pgfqpoint{5.084697in}{0.828152in}}%
\pgfpathlineto{\pgfqpoint{5.199995in}{0.813728in}}%
\pgfpathlineto{\pgfqpoint{5.315292in}{0.801676in}}%
\pgfpathlineto{\pgfqpoint{5.430590in}{0.791454in}}%
\pgfpathlineto{\pgfqpoint{5.545887in}{0.782675in}}%
\pgfpathlineto{\pgfqpoint{5.661185in}{0.775055in}}%
\pgfpathlineto{\pgfqpoint{5.776483in}{0.768377in}}%
\pgfpathlineto{\pgfqpoint{5.891780in}{0.762477in}}%
\pgfpathlineto{\pgfqpoint{6.007078in}{0.757226in}}%
\pgfpathlineto{\pgfqpoint{6.122375in}{0.752524in}}%
\pgfpathlineto{\pgfqpoint{6.237673in}{0.748288in}}%
\pgfusepath{stroke}%
\end{pgfscope}%
\begin{pgfscope}%
\pgfpathrectangle{\pgfqpoint{3.937486in}{0.611944in}}{\pgfqpoint{2.409719in}{2.274722in}}%
\pgfusepath{clip}%
\pgfsetrectcap%
\pgfsetroundjoin%
\pgfsetlinewidth{1.405250pt}%
\definecolor{currentstroke}{rgb}{0.172549,0.627451,0.172549}%
\pgfsetstrokecolor{currentstroke}%
\pgfsetdash{}{0pt}%
\pgfpathmoveto{\pgfqpoint{4.047019in}{2.458741in}}%
\pgfpathlineto{\pgfqpoint{4.162316in}{1.662742in}}%
\pgfpathlineto{\pgfqpoint{4.277614in}{1.355157in}}%
\pgfpathlineto{\pgfqpoint{4.392912in}{1.192577in}}%
\pgfpathlineto{\pgfqpoint{4.508209in}{1.092097in}}%
\pgfpathlineto{\pgfqpoint{4.623507in}{1.023858in}}%
\pgfpathlineto{\pgfqpoint{4.738804in}{0.974494in}}%
\pgfpathlineto{\pgfqpoint{4.854102in}{0.937127in}}%
\pgfpathlineto{\pgfqpoint{4.969399in}{0.907859in}}%
\pgfpathlineto{\pgfqpoint{5.084697in}{0.884314in}}%
\pgfpathlineto{\pgfqpoint{5.199995in}{0.864963in}}%
\pgfpathlineto{\pgfqpoint{5.315292in}{0.848778in}}%
\pgfpathlineto{\pgfqpoint{5.430590in}{0.835040in}}%
\pgfpathlineto{\pgfqpoint{5.545887in}{0.823234in}}%
\pgfpathlineto{\pgfqpoint{5.661185in}{0.812978in}}%
\pgfpathlineto{\pgfqpoint{5.776483in}{0.803987in}}%
\pgfpathlineto{\pgfqpoint{5.891780in}{0.796040in}}%
\pgfpathlineto{\pgfqpoint{6.007078in}{0.788964in}}%
\pgfpathlineto{\pgfqpoint{6.122375in}{0.782625in}}%
\pgfpathlineto{\pgfqpoint{6.237673in}{0.776913in}}%
\pgfusepath{stroke}%
\end{pgfscope}%
\begin{pgfscope}%
\pgfpathrectangle{\pgfqpoint{3.937486in}{0.611944in}}{\pgfqpoint{2.409719in}{2.274722in}}%
\pgfusepath{clip}%
\pgfsetrectcap%
\pgfsetroundjoin%
\pgfsetlinewidth{1.405250pt}%
\definecolor{currentstroke}{rgb}{0.839216,0.152941,0.156863}%
\pgfsetstrokecolor{currentstroke}%
\pgfsetdash{}{0pt}%
\pgfpathmoveto{\pgfqpoint{4.047019in}{2.783270in}}%
\pgfpathlineto{\pgfqpoint{4.162316in}{1.871568in}}%
\pgfpathlineto{\pgfqpoint{4.277614in}{1.506453in}}%
\pgfpathlineto{\pgfqpoint{4.392912in}{1.310856in}}%
\pgfpathlineto{\pgfqpoint{4.508209in}{1.189101in}}%
\pgfpathlineto{\pgfqpoint{4.623507in}{1.106045in}}%
\pgfpathlineto{\pgfqpoint{4.738804in}{1.045778in}}%
\pgfpathlineto{\pgfqpoint{4.854102in}{1.000056in}}%
\pgfpathlineto{\pgfqpoint{4.969399in}{0.964182in}}%
\pgfpathlineto{\pgfqpoint{5.084697in}{0.935285in}}%
\pgfpathlineto{\pgfqpoint{5.199995in}{0.911510in}}%
\pgfpathlineto{\pgfqpoint{5.315292in}{0.891606in}}%
\pgfpathlineto{\pgfqpoint{5.430590in}{0.874700in}}%
\pgfpathlineto{\pgfqpoint{5.545887in}{0.860161in}}%
\pgfpathlineto{\pgfqpoint{5.661185in}{0.847526in}}%
\pgfpathlineto{\pgfqpoint{5.776483in}{0.836442in}}%
\pgfpathlineto{\pgfqpoint{5.891780in}{0.826641in}}%
\pgfpathlineto{\pgfqpoint{6.007078in}{0.817913in}}%
\pgfpathlineto{\pgfqpoint{6.122375in}{0.810090in}}%
\pgfpathlineto{\pgfqpoint{6.237673in}{0.803039in}}%
\pgfusepath{stroke}%
\end{pgfscope}%
\begin{pgfscope}%
\pgfsetrectcap%
\pgfsetmiterjoin%
\pgfsetlinewidth{0.803000pt}%
\definecolor{currentstroke}{rgb}{0.000000,0.000000,0.000000}%
\pgfsetstrokecolor{currentstroke}%
\pgfsetdash{}{0pt}%
\pgfpathmoveto{\pgfqpoint{3.937486in}{0.611944in}}%
\pgfpathlineto{\pgfqpoint{3.937486in}{2.886667in}}%
\pgfusepath{stroke}%
\end{pgfscope}%
\begin{pgfscope}%
\pgfsetrectcap%
\pgfsetmiterjoin%
\pgfsetlinewidth{0.803000pt}%
\definecolor{currentstroke}{rgb}{0.000000,0.000000,0.000000}%
\pgfsetstrokecolor{currentstroke}%
\pgfsetdash{}{0pt}%
\pgfpathmoveto{\pgfqpoint{6.347206in}{0.611944in}}%
\pgfpathlineto{\pgfqpoint{6.347206in}{2.886667in}}%
\pgfusepath{stroke}%
\end{pgfscope}%
\begin{pgfscope}%
\pgfsetrectcap%
\pgfsetmiterjoin%
\pgfsetlinewidth{0.803000pt}%
\definecolor{currentstroke}{rgb}{0.000000,0.000000,0.000000}%
\pgfsetstrokecolor{currentstroke}%
\pgfsetdash{}{0pt}%
\pgfpathmoveto{\pgfqpoint{3.937486in}{0.611944in}}%
\pgfpathlineto{\pgfqpoint{6.347206in}{0.611944in}}%
\pgfusepath{stroke}%
\end{pgfscope}%
\begin{pgfscope}%
\pgfsetrectcap%
\pgfsetmiterjoin%
\pgfsetlinewidth{0.803000pt}%
\definecolor{currentstroke}{rgb}{0.000000,0.000000,0.000000}%
\pgfsetstrokecolor{currentstroke}%
\pgfsetdash{}{0pt}%
\pgfpathmoveto{\pgfqpoint{3.937486in}{2.886667in}}%
\pgfpathlineto{\pgfqpoint{6.347206in}{2.886667in}}%
\pgfusepath{stroke}%
\end{pgfscope}%
\begin{pgfscope}%
\definecolor{textcolor}{rgb}{0.000000,0.000000,0.000000}%
\pgfsetstrokecolor{textcolor}%
\pgfsetfillcolor{textcolor}%
\pgftext[x=5.142346in,y=2.970000in,,base]{\color{textcolor}{\sffamily\fontsize{10.000000}{12.000000}\selectfont\catcode`\^=\active\def^{\ifmmode\sp\else\^{}\fi}\catcode`\%=\active\def
\end{pgfscope}%
\begin{pgfscope}%
\pgfsetbuttcap%
\pgfsetmiterjoin%
\definecolor{currentfill}{rgb}{1.000000,1.000000,1.000000}%
\pgfsetfillcolor{currentfill}%
\pgfsetfillopacity{0.800000}%
\pgfsetlinewidth{0.240900pt}%
\definecolor{currentstroke}{rgb}{0.800000,0.800000,0.800000}%
\pgfsetstrokecolor{currentstroke}%
\pgfsetstrokeopacity{0.800000}%
\pgfsetdash{}{0pt}%
\pgfpathmoveto{\pgfqpoint{5.107328in}{2.165679in}}%
\pgfpathlineto{\pgfqpoint{6.269428in}{2.165679in}}%
\pgfpathquadraticcurveto{\pgfqpoint{6.291650in}{2.165679in}}{\pgfqpoint{6.291650in}{2.187902in}}%
\pgfpathlineto{\pgfqpoint{6.291650in}{2.808889in}}%
\pgfpathquadraticcurveto{\pgfqpoint{6.291650in}{2.831111in}}{\pgfqpoint{6.269428in}{2.831111in}}%
\pgfpathlineto{\pgfqpoint{5.107328in}{2.831111in}}%
\pgfpathquadraticcurveto{\pgfqpoint{5.085106in}{2.831111in}}{\pgfqpoint{5.085106in}{2.808889in}}%
\pgfpathlineto{\pgfqpoint{5.085106in}{2.187902in}}%
\pgfpathquadraticcurveto{\pgfqpoint{5.085106in}{2.165679in}}{\pgfqpoint{5.107328in}{2.165679in}}%
\pgfpathlineto{\pgfqpoint{5.107328in}{2.165679in}}%
\pgfpathclose%
\pgfusepath{stroke,fill}%
\end{pgfscope}%
\begin{pgfscope}%
\pgfsetrectcap%
\pgfsetroundjoin%
\pgfsetlinewidth{1.405250pt}%
\definecolor{currentstroke}{rgb}{0.121569,0.466667,0.705882}%
\pgfsetstrokecolor{currentstroke}%
\pgfsetdash{}{0pt}%
\pgfpathmoveto{\pgfqpoint{5.129550in}{2.747778in}}%
\pgfpathlineto{\pgfqpoint{5.240661in}{2.747778in}}%
\pgfpathlineto{\pgfqpoint{5.351773in}{2.747778in}}%
\pgfusepath{stroke}%
\end{pgfscope}%
\begin{pgfscope}%
\definecolor{textcolor}{rgb}{0.000000,0.000000,0.000000}%
\pgfsetstrokecolor{textcolor}%
\pgfsetfillcolor{textcolor}%
\pgftext[x=5.440661in,y=2.708889in,left,base]{\color{textcolor}{\sffamily\fontsize{8.000000}{9.600000}\selectfont\catcode`\^=\active\def^{\ifmmode\sp\else\^{}\fi}\catcode`\%=\active\def
\end{pgfscope}%
\begin{pgfscope}%
\pgfsetrectcap%
\pgfsetroundjoin%
\pgfsetlinewidth{1.405250pt}%
\definecolor{currentstroke}{rgb}{1.000000,0.498039,0.054902}%
\pgfsetstrokecolor{currentstroke}%
\pgfsetdash{}{0pt}%
\pgfpathmoveto{\pgfqpoint{5.129550in}{2.589753in}}%
\pgfpathlineto{\pgfqpoint{5.240661in}{2.589753in}}%
\pgfpathlineto{\pgfqpoint{5.351773in}{2.589753in}}%
\pgfusepath{stroke}%
\end{pgfscope}%
\begin{pgfscope}%
\definecolor{textcolor}{rgb}{0.000000,0.000000,0.000000}%
\pgfsetstrokecolor{textcolor}%
\pgfsetfillcolor{textcolor}%
\pgftext[x=5.440661in,y=2.550864in,left,base]{\color{textcolor}{\sffamily\fontsize{8.000000}{9.600000}\selectfont\catcode`\^=\active\def^{\ifmmode\sp\else\^{}\fi}\catcode`\%=\active\def
\end{pgfscope}%
\begin{pgfscope}%
\pgfsetrectcap%
\pgfsetroundjoin%
\pgfsetlinewidth{1.405250pt}%
\definecolor{currentstroke}{rgb}{0.172549,0.627451,0.172549}%
\pgfsetstrokecolor{currentstroke}%
\pgfsetdash{}{0pt}%
\pgfpathmoveto{\pgfqpoint{5.129550in}{2.431729in}}%
\pgfpathlineto{\pgfqpoint{5.240661in}{2.431729in}}%
\pgfpathlineto{\pgfqpoint{5.351773in}{2.431729in}}%
\pgfusepath{stroke}%
\end{pgfscope}%
\begin{pgfscope}%
\definecolor{textcolor}{rgb}{0.000000,0.000000,0.000000}%
\pgfsetstrokecolor{textcolor}%
\pgfsetfillcolor{textcolor}%
\pgftext[x=5.440661in,y=2.392840in,left,base]{\color{textcolor}{\sffamily\fontsize{8.000000}{9.600000}\selectfont\catcode`\^=\active\def^{\ifmmode\sp\else\^{}\fi}\catcode`\%=\active\def
\end{pgfscope}%
\begin{pgfscope}%
\pgfsetrectcap%
\pgfsetroundjoin%
\pgfsetlinewidth{1.405250pt}%
\definecolor{currentstroke}{rgb}{0.839216,0.152941,0.156863}%
\pgfsetstrokecolor{currentstroke}%
\pgfsetdash{}{0pt}%
\pgfpathmoveto{\pgfqpoint{5.129550in}{2.273704in}}%
\pgfpathlineto{\pgfqpoint{5.240661in}{2.273704in}}%
\pgfpathlineto{\pgfqpoint{5.351773in}{2.273704in}}%
\pgfusepath{stroke}%
\end{pgfscope}%
\begin{pgfscope}%
\definecolor{textcolor}{rgb}{0.000000,0.000000,0.000000}%
\pgfsetstrokecolor{textcolor}%
\pgfsetfillcolor{textcolor}%
\pgftext[x=5.440661in,y=2.234815in,left,base]{\color{textcolor}{\sffamily\fontsize{8.000000}{9.600000}\selectfont\catcode`\^=\active\def^{\ifmmode\sp\else\^{}\fi}\catcode`\%=\active\def
\end{pgfscope}%
\end{pgfpicture}%
\makeatother%
\endgroup%

%% file: figs/pml-1d-delta-ppw.pgf
\begingroup%
\makeatletter%
\begin{pgfpicture}%
\pgfpathrectangle{\pgfpointorigin}{\pgfqpoint{6.500000in}{3.250000in}}%
\pgfusepath{use as bounding box, clip}%
\begin{pgfscope}%
\pgfsetbuttcap%
\pgfsetmiterjoin%
\definecolor{currentfill}{rgb}{1.000000,1.000000,1.000000}%
\pgfsetfillcolor{currentfill}%
\pgfsetlinewidth{0.000000pt}%
\definecolor{currentstroke}{rgb}{1.000000,1.000000,1.000000}%
\pgfsetstrokecolor{currentstroke}%
\pgfsetdash{}{0pt}%
\pgfpathmoveto{\pgfqpoint{0.000000in}{0.000000in}}%
\pgfpathlineto{\pgfqpoint{6.500000in}{0.000000in}}%
\pgfpathlineto{\pgfqpoint{6.500000in}{3.250000in}}%
\pgfpathlineto{\pgfqpoint{0.000000in}{3.250000in}}%
\pgfpathlineto{\pgfqpoint{0.000000in}{0.000000in}}%
\pgfpathclose%
\pgfusepath{fill}%
\end{pgfscope}%
\begin{pgfscope}%
\pgfsetbuttcap%
\pgfsetmiterjoin%
\definecolor{currentfill}{rgb}{1.000000,1.000000,1.000000}%
\pgfsetfillcolor{currentfill}%
\pgfsetlinewidth{0.000000pt}%
\definecolor{currentstroke}{rgb}{0.000000,0.000000,0.000000}%
\pgfsetstrokecolor{currentstroke}%
\pgfsetstrokeopacity{0.000000}%
\pgfsetdash{}{0pt}%
\pgfpathmoveto{\pgfqpoint{0.833736in}{0.611944in}}%
\pgfpathlineto{\pgfqpoint{3.172206in}{0.611944in}}%
\pgfpathlineto{\pgfqpoint{3.172206in}{2.886667in}}%
\pgfpathlineto{\pgfqpoint{0.833736in}{2.886667in}}%
\pgfpathlineto{\pgfqpoint{0.833736in}{0.611944in}}%
\pgfpathclose%
\pgfusepath{fill}%
\end{pgfscope}%
\begin{pgfscope}%
\pgfsetbuttcap%
\pgfsetroundjoin%
\definecolor{currentfill}{rgb}{0.000000,0.000000,0.000000}%
\pgfsetfillcolor{currentfill}%
\pgfsetlinewidth{0.803000pt}%
\definecolor{currentstroke}{rgb}{0.000000,0.000000,0.000000}%
\pgfsetstrokecolor{currentstroke}%
\pgfsetdash{}{0pt}%
\pgfsys@defobject{currentmarker}{\pgfqpoint{0.000000in}{-0.048611in}}{\pgfqpoint{0.000000in}{0.000000in}}{%
\pgfpathmoveto{\pgfqpoint{0.000000in}{0.000000in}}%
\pgfpathlineto{\pgfqpoint{0.000000in}{-0.048611in}}%
\pgfusepath{stroke,fill}%
}%
\begin{pgfscope}%
\pgfsys@transformshift{0.940030in}{0.611944in}%
\pgfsys@useobject{currentmarker}{}%
\end{pgfscope}%
\end{pgfscope}%
\begin{pgfscope}%
\definecolor{textcolor}{rgb}{0.000000,0.000000,0.000000}%
\pgfsetstrokecolor{textcolor}%
\pgfsetfillcolor{textcolor}%
\pgftext[x=0.940030in,y=0.485556in,,top]{\color{textcolor}{\sffamily\fontsize{8.000000}{9.600000}\selectfont\catcode`\^=\active\def^{\ifmmode\sp\else\^{}\fi}\catcode`\%=\active\def
\end{pgfscope}%
\begin{pgfscope}%
\pgfsetbuttcap%
\pgfsetroundjoin%
\definecolor{currentfill}{rgb}{0.000000,0.000000,0.000000}%
\pgfsetfillcolor{currentfill}%
\pgfsetlinewidth{0.803000pt}%
\definecolor{currentstroke}{rgb}{0.000000,0.000000,0.000000}%
\pgfsetstrokecolor{currentstroke}%
\pgfsetdash{}{0pt}%
\pgfsys@defobject{currentmarker}{\pgfqpoint{0.000000in}{-0.048611in}}{\pgfqpoint{0.000000in}{0.000000in}}{%
\pgfpathmoveto{\pgfqpoint{0.000000in}{0.000000in}}%
\pgfpathlineto{\pgfqpoint{0.000000in}{-0.048611in}}%
\pgfusepath{stroke,fill}%
}%
\begin{pgfscope}%
\pgfsys@transformshift{1.365206in}{0.611944in}%
\pgfsys@useobject{currentmarker}{}%
\end{pgfscope}%
\end{pgfscope}%
\begin{pgfscope}%
\definecolor{textcolor}{rgb}{0.000000,0.000000,0.000000}%
\pgfsetstrokecolor{textcolor}%
\pgfsetfillcolor{textcolor}%
\pgftext[x=1.365206in,y=0.485556in,,top]{\color{textcolor}{\sffamily\fontsize{8.000000}{9.600000}\selectfont\catcode`\^=\active\def^{\ifmmode\sp\else\^{}\fi}\catcode`\%=\active\def
\end{pgfscope}%
\begin{pgfscope}%
\pgfsetbuttcap%
\pgfsetroundjoin%
\definecolor{currentfill}{rgb}{0.000000,0.000000,0.000000}%
\pgfsetfillcolor{currentfill}%
\pgfsetlinewidth{0.803000pt}%
\definecolor{currentstroke}{rgb}{0.000000,0.000000,0.000000}%
\pgfsetstrokecolor{currentstroke}%
\pgfsetdash{}{0pt}%
\pgfsys@defobject{currentmarker}{\pgfqpoint{0.000000in}{-0.048611in}}{\pgfqpoint{0.000000in}{0.000000in}}{%
\pgfpathmoveto{\pgfqpoint{0.000000in}{0.000000in}}%
\pgfpathlineto{\pgfqpoint{0.000000in}{-0.048611in}}%
\pgfusepath{stroke,fill}%
}%
\begin{pgfscope}%
\pgfsys@transformshift{1.790383in}{0.611944in}%
\pgfsys@useobject{currentmarker}{}%
\end{pgfscope}%
\end{pgfscope}%
\begin{pgfscope}%
\definecolor{textcolor}{rgb}{0.000000,0.000000,0.000000}%
\pgfsetstrokecolor{textcolor}%
\pgfsetfillcolor{textcolor}%
\pgftext[x=1.790383in,y=0.485556in,,top]{\color{textcolor}{\sffamily\fontsize{8.000000}{9.600000}\selectfont\catcode`\^=\active\def^{\ifmmode\sp\else\^{}\fi}\catcode`\%=\active\def
\end{pgfscope}%
\begin{pgfscope}%
\pgfsetbuttcap%
\pgfsetroundjoin%
\definecolor{currentfill}{rgb}{0.000000,0.000000,0.000000}%
\pgfsetfillcolor{currentfill}%
\pgfsetlinewidth{0.803000pt}%
\definecolor{currentstroke}{rgb}{0.000000,0.000000,0.000000}%
\pgfsetstrokecolor{currentstroke}%
\pgfsetdash{}{0pt}%
\pgfsys@defobject{currentmarker}{\pgfqpoint{0.000000in}{-0.048611in}}{\pgfqpoint{0.000000in}{0.000000in}}{%
\pgfpathmoveto{\pgfqpoint{0.000000in}{0.000000in}}%
\pgfpathlineto{\pgfqpoint{0.000000in}{-0.048611in}}%
\pgfusepath{stroke,fill}%
}%
\begin{pgfscope}%
\pgfsys@transformshift{2.215559in}{0.611944in}%
\pgfsys@useobject{currentmarker}{}%
\end{pgfscope}%
\end{pgfscope}%
\begin{pgfscope}%
\definecolor{textcolor}{rgb}{0.000000,0.000000,0.000000}%
\pgfsetstrokecolor{textcolor}%
\pgfsetfillcolor{textcolor}%
\pgftext[x=2.215559in,y=0.485556in,,top]{\color{textcolor}{\sffamily\fontsize{8.000000}{9.600000}\selectfont\catcode`\^=\active\def^{\ifmmode\sp\else\^{}\fi}\catcode`\%=\active\def
\end{pgfscope}%
\begin{pgfscope}%
\pgfsetbuttcap%
\pgfsetroundjoin%
\definecolor{currentfill}{rgb}{0.000000,0.000000,0.000000}%
\pgfsetfillcolor{currentfill}%
\pgfsetlinewidth{0.803000pt}%
\definecolor{currentstroke}{rgb}{0.000000,0.000000,0.000000}%
\pgfsetstrokecolor{currentstroke}%
\pgfsetdash{}{0pt}%
\pgfsys@defobject{currentmarker}{\pgfqpoint{0.000000in}{-0.048611in}}{\pgfqpoint{0.000000in}{0.000000in}}{%
\pgfpathmoveto{\pgfqpoint{0.000000in}{0.000000in}}%
\pgfpathlineto{\pgfqpoint{0.000000in}{-0.048611in}}%
\pgfusepath{stroke,fill}%
}%
\begin{pgfscope}%
\pgfsys@transformshift{2.640735in}{0.611944in}%
\pgfsys@useobject{currentmarker}{}%
\end{pgfscope}%
\end{pgfscope}%
\begin{pgfscope}%
\definecolor{textcolor}{rgb}{0.000000,0.000000,0.000000}%
\pgfsetstrokecolor{textcolor}%
\pgfsetfillcolor{textcolor}%
\pgftext[x=2.640735in,y=0.485556in,,top]{\color{textcolor}{\sffamily\fontsize{8.000000}{9.600000}\selectfont\catcode`\^=\active\def^{\ifmmode\sp\else\^{}\fi}\catcode`\%=\active\def
\end{pgfscope}%
\begin{pgfscope}%
\pgfsetbuttcap%
\pgfsetroundjoin%
\definecolor{currentfill}{rgb}{0.000000,0.000000,0.000000}%
\pgfsetfillcolor{currentfill}%
\pgfsetlinewidth{0.803000pt}%
\definecolor{currentstroke}{rgb}{0.000000,0.000000,0.000000}%
\pgfsetstrokecolor{currentstroke}%
\pgfsetdash{}{0pt}%
\pgfsys@defobject{currentmarker}{\pgfqpoint{0.000000in}{-0.048611in}}{\pgfqpoint{0.000000in}{0.000000in}}{%
\pgfpathmoveto{\pgfqpoint{0.000000in}{0.000000in}}%
\pgfpathlineto{\pgfqpoint{0.000000in}{-0.048611in}}%
\pgfusepath{stroke,fill}%
}%
\begin{pgfscope}%
\pgfsys@transformshift{3.065912in}{0.611944in}%
\pgfsys@useobject{currentmarker}{}%
\end{pgfscope}%
\end{pgfscope}%
\begin{pgfscope}%
\definecolor{textcolor}{rgb}{0.000000,0.000000,0.000000}%
\pgfsetstrokecolor{textcolor}%
\pgfsetfillcolor{textcolor}%
\pgftext[x=3.065912in,y=0.485556in,,top]{\color{textcolor}{\sffamily\fontsize{8.000000}{9.600000}\selectfont\catcode`\^=\active\def^{\ifmmode\sp\else\^{}\fi}\catcode`\%=\active\def
\end{pgfscope}%
\begin{pgfscope}%
\definecolor{textcolor}{rgb}{0.000000,0.000000,0.000000}%
\pgfsetstrokecolor{textcolor}%
\pgfsetfillcolor{textcolor}%
\pgftext[x=2.002971in,y=0.331234in,,top]{\color{textcolor}{\sffamily\fontsize{10.000000}{12.000000}\selectfont\catcode`\^=\active\def^{\ifmmode\sp\else\^{}\fi}\catcode`\%=\active\def
\end{pgfscope}%
\begin{pgfscope}%
\pgfsetbuttcap%
\pgfsetroundjoin%
\definecolor{currentfill}{rgb}{0.000000,0.000000,0.000000}%
\pgfsetfillcolor{currentfill}%
\pgfsetlinewidth{0.803000pt}%
\definecolor{currentstroke}{rgb}{0.000000,0.000000,0.000000}%
\pgfsetstrokecolor{currentstroke}%
\pgfsetdash{}{0pt}%
\pgfsys@defobject{currentmarker}{\pgfqpoint{-0.048611in}{0.000000in}}{\pgfqpoint{-0.000000in}{0.000000in}}{%
\pgfpathmoveto{\pgfqpoint{-0.000000in}{0.000000in}}%
\pgfpathlineto{\pgfqpoint{-0.048611in}{0.000000in}}%
\pgfusepath{stroke,fill}%
}%
\begin{pgfscope}%
\pgfsys@transformshift{0.833736in}{0.847496in}%
\pgfsys@useobject{currentmarker}{}%
\end{pgfscope}%
\end{pgfscope}%
\begin{pgfscope}%
\definecolor{textcolor}{rgb}{0.000000,0.000000,0.000000}%
\pgfsetstrokecolor{textcolor}%
\pgfsetfillcolor{textcolor}%
\pgftext[x=0.497468in, y=0.808916in, left, base]{\color{textcolor}{\sffamily\fontsize{8.000000}{9.600000}\selectfont\catcode`\^=\active\def^{\ifmmode\sp\else\^{}\fi}\catcode`\%=\active\def
\end{pgfscope}%
\begin{pgfscope}%
\pgfsetbuttcap%
\pgfsetroundjoin%
\definecolor{currentfill}{rgb}{0.000000,0.000000,0.000000}%
\pgfsetfillcolor{currentfill}%
\pgfsetlinewidth{0.803000pt}%
\definecolor{currentstroke}{rgb}{0.000000,0.000000,0.000000}%
\pgfsetstrokecolor{currentstroke}%
\pgfsetdash{}{0pt}%
\pgfsys@defobject{currentmarker}{\pgfqpoint{-0.048611in}{0.000000in}}{\pgfqpoint{-0.000000in}{0.000000in}}{%
\pgfpathmoveto{\pgfqpoint{-0.000000in}{0.000000in}}%
\pgfpathlineto{\pgfqpoint{-0.048611in}{0.000000in}}%
\pgfusepath{stroke,fill}%
}%
\begin{pgfscope}%
\pgfsys@transformshift{0.833736in}{1.192514in}%
\pgfsys@useobject{currentmarker}{}%
\end{pgfscope}%
\end{pgfscope}%
\begin{pgfscope}%
\definecolor{textcolor}{rgb}{0.000000,0.000000,0.000000}%
\pgfsetstrokecolor{textcolor}%
\pgfsetfillcolor{textcolor}%
\pgftext[x=0.497468in, y=1.153934in, left, base]{\color{textcolor}{\sffamily\fontsize{8.000000}{9.600000}\selectfont\catcode`\^=\active\def^{\ifmmode\sp\else\^{}\fi}\catcode`\%=\active\def
\end{pgfscope}%
\begin{pgfscope}%
\pgfsetbuttcap%
\pgfsetroundjoin%
\definecolor{currentfill}{rgb}{0.000000,0.000000,0.000000}%
\pgfsetfillcolor{currentfill}%
\pgfsetlinewidth{0.803000pt}%
\definecolor{currentstroke}{rgb}{0.000000,0.000000,0.000000}%
\pgfsetstrokecolor{currentstroke}%
\pgfsetdash{}{0pt}%
\pgfsys@defobject{currentmarker}{\pgfqpoint{-0.048611in}{0.000000in}}{\pgfqpoint{-0.000000in}{0.000000in}}{%
\pgfpathmoveto{\pgfqpoint{-0.000000in}{0.000000in}}%
\pgfpathlineto{\pgfqpoint{-0.048611in}{0.000000in}}%
\pgfusepath{stroke,fill}%
}%
\begin{pgfscope}%
\pgfsys@transformshift{0.833736in}{1.537532in}%
\pgfsys@useobject{currentmarker}{}%
\end{pgfscope}%
\end{pgfscope}%
\begin{pgfscope}%
\definecolor{textcolor}{rgb}{0.000000,0.000000,0.000000}%
\pgfsetstrokecolor{textcolor}%
\pgfsetfillcolor{textcolor}%
\pgftext[x=0.497468in, y=1.498951in, left, base]{\color{textcolor}{\sffamily\fontsize{8.000000}{9.600000}\selectfont\catcode`\^=\active\def^{\ifmmode\sp\else\^{}\fi}\catcode`\%=\active\def
\end{pgfscope}%
\begin{pgfscope}%
\pgfsetbuttcap%
\pgfsetroundjoin%
\definecolor{currentfill}{rgb}{0.000000,0.000000,0.000000}%
\pgfsetfillcolor{currentfill}%
\pgfsetlinewidth{0.803000pt}%
\definecolor{currentstroke}{rgb}{0.000000,0.000000,0.000000}%
\pgfsetstrokecolor{currentstroke}%
\pgfsetdash{}{0pt}%
\pgfsys@defobject{currentmarker}{\pgfqpoint{-0.048611in}{0.000000in}}{\pgfqpoint{-0.000000in}{0.000000in}}{%
\pgfpathmoveto{\pgfqpoint{-0.000000in}{0.000000in}}%
\pgfpathlineto{\pgfqpoint{-0.048611in}{0.000000in}}%
\pgfusepath{stroke,fill}%
}%
\begin{pgfscope}%
\pgfsys@transformshift{0.833736in}{1.882549in}%
\pgfsys@useobject{currentmarker}{}%
\end{pgfscope}%
\end{pgfscope}%
\begin{pgfscope}%
\definecolor{textcolor}{rgb}{0.000000,0.000000,0.000000}%
\pgfsetstrokecolor{textcolor}%
\pgfsetfillcolor{textcolor}%
\pgftext[x=0.497468in, y=1.843969in, left, base]{\color{textcolor}{\sffamily\fontsize{8.000000}{9.600000}\selectfont\catcode`\^=\active\def^{\ifmmode\sp\else\^{}\fi}\catcode`\%=\active\def
\end{pgfscope}%
\begin{pgfscope}%
\pgfsetbuttcap%
\pgfsetroundjoin%
\definecolor{currentfill}{rgb}{0.000000,0.000000,0.000000}%
\pgfsetfillcolor{currentfill}%
\pgfsetlinewidth{0.803000pt}%
\definecolor{currentstroke}{rgb}{0.000000,0.000000,0.000000}%
\pgfsetstrokecolor{currentstroke}%
\pgfsetdash{}{0pt}%
\pgfsys@defobject{currentmarker}{\pgfqpoint{-0.048611in}{0.000000in}}{\pgfqpoint{-0.000000in}{0.000000in}}{%
\pgfpathmoveto{\pgfqpoint{-0.000000in}{0.000000in}}%
\pgfpathlineto{\pgfqpoint{-0.048611in}{0.000000in}}%
\pgfusepath{stroke,fill}%
}%
\begin{pgfscope}%
\pgfsys@transformshift{0.833736in}{2.227567in}%
\pgfsys@useobject{currentmarker}{}%
\end{pgfscope}%
\end{pgfscope}%
\begin{pgfscope}%
\definecolor{textcolor}{rgb}{0.000000,0.000000,0.000000}%
\pgfsetstrokecolor{textcolor}%
\pgfsetfillcolor{textcolor}%
\pgftext[x=0.497468in, y=2.188987in, left, base]{\color{textcolor}{\sffamily\fontsize{8.000000}{9.600000}\selectfont\catcode`\^=\active\def^{\ifmmode\sp\else\^{}\fi}\catcode`\%=\active\def
\end{pgfscope}%
\begin{pgfscope}%
\pgfsetbuttcap%
\pgfsetroundjoin%
\definecolor{currentfill}{rgb}{0.000000,0.000000,0.000000}%
\pgfsetfillcolor{currentfill}%
\pgfsetlinewidth{0.803000pt}%
\definecolor{currentstroke}{rgb}{0.000000,0.000000,0.000000}%
\pgfsetstrokecolor{currentstroke}%
\pgfsetdash{}{0pt}%
\pgfsys@defobject{currentmarker}{\pgfqpoint{-0.048611in}{0.000000in}}{\pgfqpoint{-0.000000in}{0.000000in}}{%
\pgfpathmoveto{\pgfqpoint{-0.000000in}{0.000000in}}%
\pgfpathlineto{\pgfqpoint{-0.048611in}{0.000000in}}%
\pgfusepath{stroke,fill}%
}%
\begin{pgfscope}%
\pgfsys@transformshift{0.833736in}{2.572584in}%
\pgfsys@useobject{currentmarker}{}%
\end{pgfscope}%
\end{pgfscope}%
\begin{pgfscope}%
\definecolor{textcolor}{rgb}{0.000000,0.000000,0.000000}%
\pgfsetstrokecolor{textcolor}%
\pgfsetfillcolor{textcolor}%
\pgftext[x=0.497468in, y=2.534004in, left, base]{\color{textcolor}{\sffamily\fontsize{8.000000}{9.600000}\selectfont\catcode`\^=\active\def^{\ifmmode\sp\else\^{}\fi}\catcode`\%=\active\def
\end{pgfscope}%
\begin{pgfscope}%
\definecolor{textcolor}{rgb}{0.000000,0.000000,0.000000}%
\pgfsetstrokecolor{textcolor}%
\pgfsetfillcolor{textcolor}%
\pgftext[x=0.441912in,y=1.749306in,,bottom,rotate=90.000000]{\color{textcolor}{\sffamily\fontsize{10.000000}{12.000000}\selectfont\catcode`\^=\active\def^{\ifmmode\sp\else\^{}\fi}\catcode`\%=\active\def
\end{pgfscope}%
\begin{pgfscope}%
\pgfpathrectangle{\pgfqpoint{0.833736in}{0.611944in}}{\pgfqpoint{2.338469in}{2.274722in}}%
\pgfusepath{clip}%
\pgfsetrectcap%
\pgfsetroundjoin%
\pgfsetlinewidth{1.405250pt}%
\definecolor{currentstroke}{rgb}{0.121569,0.466667,0.705882}%
\pgfsetstrokecolor{currentstroke}%
\pgfsetdash{}{0pt}%
\pgfpathmoveto{\pgfqpoint{0.940030in}{0.881396in}}%
\pgfpathlineto{\pgfqpoint{1.152618in}{0.857126in}}%
\pgfpathlineto{\pgfqpoint{1.365206in}{0.835355in}}%
\pgfpathlineto{\pgfqpoint{1.577795in}{0.815664in}}%
\pgfpathlineto{\pgfqpoint{1.790383in}{0.797724in}}%
\pgfpathlineto{\pgfqpoint{2.002971in}{0.781276in}}%
\pgfpathlineto{\pgfqpoint{2.215559in}{0.766113in}}%
\pgfpathlineto{\pgfqpoint{2.428147in}{0.752065in}}%
\pgfpathlineto{\pgfqpoint{2.640735in}{0.738995in}}%
\pgfpathlineto{\pgfqpoint{2.853323in}{0.726785in}}%
\pgfpathlineto{\pgfqpoint{3.065912in}{0.715341in}}%
\pgfusepath{stroke}%
\end{pgfscope}%
\begin{pgfscope}%
\pgfpathrectangle{\pgfqpoint{0.833736in}{0.611944in}}{\pgfqpoint{2.338469in}{2.274722in}}%
\pgfusepath{clip}%
\pgfsetrectcap%
\pgfsetroundjoin%
\pgfsetlinewidth{1.405250pt}%
\definecolor{currentstroke}{rgb}{1.000000,0.498039,0.054902}%
\pgfsetstrokecolor{currentstroke}%
\pgfsetdash{}{0pt}%
\pgfpathmoveto{\pgfqpoint{0.940030in}{1.600848in}}%
\pgfpathlineto{\pgfqpoint{1.152618in}{1.559830in}}%
\pgfpathlineto{\pgfqpoint{1.365206in}{1.523152in}}%
\pgfpathlineto{\pgfqpoint{1.577795in}{1.490052in}}%
\pgfpathlineto{\pgfqpoint{1.790383in}{1.459948in}}%
\pgfpathlineto{\pgfqpoint{2.002971in}{1.432384in}}%
\pgfpathlineto{\pgfqpoint{2.215559in}{1.406998in}}%
\pgfpathlineto{\pgfqpoint{2.428147in}{1.383500in}}%
\pgfpathlineto{\pgfqpoint{2.640735in}{1.361650in}}%
\pgfpathlineto{\pgfqpoint{2.853323in}{1.341251in}}%
\pgfpathlineto{\pgfqpoint{3.065912in}{1.322139in}}%
\pgfusepath{stroke}%
\end{pgfscope}%
\begin{pgfscope}%
\pgfpathrectangle{\pgfqpoint{0.833736in}{0.611944in}}{\pgfqpoint{2.338469in}{2.274722in}}%
\pgfusepath{clip}%
\pgfsetrectcap%
\pgfsetroundjoin%
\pgfsetlinewidth{1.405250pt}%
\definecolor{currentstroke}{rgb}{0.172549,0.627451,0.172549}%
\pgfsetstrokecolor{currentstroke}%
\pgfsetdash{}{0pt}%
\pgfpathmoveto{\pgfqpoint{0.940030in}{2.783270in}}%
\pgfpathlineto{\pgfqpoint{1.152618in}{2.715693in}}%
\pgfpathlineto{\pgfqpoint{1.365206in}{2.655279in}}%
\pgfpathlineto{\pgfqpoint{1.577795in}{2.600765in}}%
\pgfpathlineto{\pgfqpoint{1.790383in}{2.551185in}}%
\pgfpathlineto{\pgfqpoint{2.002971in}{2.505788in}}%
\pgfpathlineto{\pgfqpoint{2.215559in}{2.463977in}}%
\pgfpathlineto{\pgfqpoint{2.428147in}{2.425270in}}%
\pgfpathlineto{\pgfqpoint{2.640735in}{2.389275in}}%
\pgfpathlineto{\pgfqpoint{2.853323in}{2.355666in}}%
\pgfpathlineto{\pgfqpoint{3.065912in}{2.324174in}}%
\pgfusepath{stroke}%
\end{pgfscope}%
\begin{pgfscope}%
\pgfsetrectcap%
\pgfsetmiterjoin%
\pgfsetlinewidth{0.803000pt}%
\definecolor{currentstroke}{rgb}{0.000000,0.000000,0.000000}%
\pgfsetstrokecolor{currentstroke}%
\pgfsetdash{}{0pt}%
\pgfpathmoveto{\pgfqpoint{0.833736in}{0.611944in}}%
\pgfpathlineto{\pgfqpoint{0.833736in}{2.886667in}}%
\pgfusepath{stroke}%
\end{pgfscope}%
\begin{pgfscope}%
\pgfsetrectcap%
\pgfsetmiterjoin%
\pgfsetlinewidth{0.803000pt}%
\definecolor{currentstroke}{rgb}{0.000000,0.000000,0.000000}%
\pgfsetstrokecolor{currentstroke}%
\pgfsetdash{}{0pt}%
\pgfpathmoveto{\pgfqpoint{3.172206in}{0.611944in}}%
\pgfpathlineto{\pgfqpoint{3.172206in}{2.886667in}}%
\pgfusepath{stroke}%
\end{pgfscope}%
\begin{pgfscope}%
\pgfsetrectcap%
\pgfsetmiterjoin%
\pgfsetlinewidth{0.803000pt}%
\definecolor{currentstroke}{rgb}{0.000000,0.000000,0.000000}%
\pgfsetstrokecolor{currentstroke}%
\pgfsetdash{}{0pt}%
\pgfpathmoveto{\pgfqpoint{0.833736in}{0.611944in}}%
\pgfpathlineto{\pgfqpoint{3.172206in}{0.611944in}}%
\pgfusepath{stroke}%
\end{pgfscope}%
\begin{pgfscope}%
\pgfsetrectcap%
\pgfsetmiterjoin%
\pgfsetlinewidth{0.803000pt}%
\definecolor{currentstroke}{rgb}{0.000000,0.000000,0.000000}%
\pgfsetstrokecolor{currentstroke}%
\pgfsetdash{}{0pt}%
\pgfpathmoveto{\pgfqpoint{0.833736in}{2.886667in}}%
\pgfpathlineto{\pgfqpoint{3.172206in}{2.886667in}}%
\pgfusepath{stroke}%
\end{pgfscope}%
\begin{pgfscope}%
\definecolor{textcolor}{rgb}{0.000000,0.000000,0.000000}%
\pgfsetstrokecolor{textcolor}%
\pgfsetfillcolor{textcolor}%
\pgftext[x=2.002971in,y=2.970000in,,base]{\color{textcolor}{\sffamily\fontsize{10.000000}{12.000000}\selectfont\catcode`\^=\active\def^{\ifmmode\sp\else\^{}\fi}\catcode`\%=\active\def
\end{pgfscope}%
\begin{pgfscope}%
\pgfsetbuttcap%
\pgfsetmiterjoin%
\definecolor{currentfill}{rgb}{1.000000,1.000000,1.000000}%
\pgfsetfillcolor{currentfill}%
\pgfsetfillopacity{0.800000}%
\pgfsetlinewidth{0.240900pt}%
\definecolor{currentstroke}{rgb}{0.800000,0.800000,0.800000}%
\pgfsetstrokecolor{currentstroke}%
\pgfsetstrokeopacity{0.800000}%
\pgfsetdash{}{0pt}%
\pgfpathmoveto{\pgfqpoint{2.315178in}{1.500231in}}%
\pgfpathlineto{\pgfqpoint{3.094428in}{1.500231in}}%
\pgfpathquadraticcurveto{\pgfqpoint{3.116650in}{1.500231in}}{\pgfqpoint{3.116650in}{1.522454in}}%
\pgfpathlineto{\pgfqpoint{3.116650in}{1.976158in}}%
\pgfpathquadraticcurveto{\pgfqpoint{3.116650in}{1.998380in}}{\pgfqpoint{3.094428in}{1.998380in}}%
\pgfpathlineto{\pgfqpoint{2.315178in}{1.998380in}}%
\pgfpathquadraticcurveto{\pgfqpoint{2.292956in}{1.998380in}}{\pgfqpoint{2.292956in}{1.976158in}}%
\pgfpathlineto{\pgfqpoint{2.292956in}{1.522454in}}%
\pgfpathquadraticcurveto{\pgfqpoint{2.292956in}{1.500231in}}{\pgfqpoint{2.315178in}{1.500231in}}%
\pgfpathlineto{\pgfqpoint{2.315178in}{1.500231in}}%
\pgfpathclose%
\pgfusepath{stroke,fill}%
\end{pgfscope}%
\begin{pgfscope}%
\pgfsetrectcap%
\pgfsetroundjoin%
\pgfsetlinewidth{1.405250pt}%
\definecolor{currentstroke}{rgb}{0.121569,0.466667,0.705882}%
\pgfsetstrokecolor{currentstroke}%
\pgfsetdash{}{0pt}%
\pgfpathmoveto{\pgfqpoint{2.337400in}{1.915046in}}%
\pgfpathlineto{\pgfqpoint{2.448511in}{1.915046in}}%
\pgfpathlineto{\pgfqpoint{2.559623in}{1.915046in}}%
\pgfusepath{stroke}%
\end{pgfscope}%
\begin{pgfscope}%
\definecolor{textcolor}{rgb}{0.000000,0.000000,0.000000}%
\pgfsetstrokecolor{textcolor}%
\pgfsetfillcolor{textcolor}%
\pgftext[x=2.648511in,y=1.876157in,left,base]{\color{textcolor}{\sffamily\fontsize{8.000000}{9.600000}\selectfont\catcode`\^=\active\def^{\ifmmode\sp\else\^{}\fi}\catcode`\%=\active\def
\end{pgfscope}%
\begin{pgfscope}%
\pgfsetrectcap%
\pgfsetroundjoin%
\pgfsetlinewidth{1.405250pt}%
\definecolor{currentstroke}{rgb}{1.000000,0.498039,0.054902}%
\pgfsetstrokecolor{currentstroke}%
\pgfsetdash{}{0pt}%
\pgfpathmoveto{\pgfqpoint{2.337400in}{1.760108in}}%
\pgfpathlineto{\pgfqpoint{2.448511in}{1.760108in}}%
\pgfpathlineto{\pgfqpoint{2.559623in}{1.760108in}}%
\pgfusepath{stroke}%
\end{pgfscope}%
\begin{pgfscope}%
\definecolor{textcolor}{rgb}{0.000000,0.000000,0.000000}%
\pgfsetstrokecolor{textcolor}%
\pgfsetfillcolor{textcolor}%
\pgftext[x=2.648511in,y=1.721219in,left,base]{\color{textcolor}{\sffamily\fontsize{8.000000}{9.600000}\selectfont\catcode`\^=\active\def^{\ifmmode\sp\else\^{}\fi}\catcode`\%=\active\def
\end{pgfscope}%
\begin{pgfscope}%
\pgfsetrectcap%
\pgfsetroundjoin%
\pgfsetlinewidth{1.405250pt}%
\definecolor{currentstroke}{rgb}{0.172549,0.627451,0.172549}%
\pgfsetstrokecolor{currentstroke}%
\pgfsetdash{}{0pt}%
\pgfpathmoveto{\pgfqpoint{2.337400in}{1.605170in}}%
\pgfpathlineto{\pgfqpoint{2.448511in}{1.605170in}}%
\pgfpathlineto{\pgfqpoint{2.559623in}{1.605170in}}%
\pgfusepath{stroke}%
\end{pgfscope}%
\begin{pgfscope}%
\definecolor{textcolor}{rgb}{0.000000,0.000000,0.000000}%
\pgfsetstrokecolor{textcolor}%
\pgfsetfillcolor{textcolor}%
\pgftext[x=2.648511in,y=1.566281in,left,base]{\color{textcolor}{\sffamily\fontsize{8.000000}{9.600000}\selectfont\catcode`\^=\active\def^{\ifmmode\sp\else\^{}\fi}\catcode`\%=\active\def
\end{pgfscope}%
\begin{pgfscope}%
\pgfsetbuttcap%
\pgfsetmiterjoin%
\definecolor{currentfill}{rgb}{1.000000,1.000000,1.000000}%
\pgfsetfillcolor{currentfill}%
\pgfsetlinewidth{0.000000pt}%
\definecolor{currentstroke}{rgb}{0.000000,0.000000,0.000000}%
\pgfsetstrokecolor{currentstroke}%
\pgfsetstrokeopacity{0.000000}%
\pgfsetdash{}{0pt}%
\pgfpathmoveto{\pgfqpoint{4.008736in}{0.611944in}}%
\pgfpathlineto{\pgfqpoint{6.347206in}{0.611944in}}%
\pgfpathlineto{\pgfqpoint{6.347206in}{2.886667in}}%
\pgfpathlineto{\pgfqpoint{4.008736in}{2.886667in}}%
\pgfpathlineto{\pgfqpoint{4.008736in}{0.611944in}}%
\pgfpathclose%
\pgfusepath{fill}%
\end{pgfscope}%
\begin{pgfscope}%
\pgfsetbuttcap%
\pgfsetroundjoin%
\definecolor{currentfill}{rgb}{0.000000,0.000000,0.000000}%
\pgfsetfillcolor{currentfill}%
\pgfsetlinewidth{0.803000pt}%
\definecolor{currentstroke}{rgb}{0.000000,0.000000,0.000000}%
\pgfsetstrokecolor{currentstroke}%
\pgfsetdash{}{0pt}%
\pgfsys@defobject{currentmarker}{\pgfqpoint{0.000000in}{-0.048611in}}{\pgfqpoint{0.000000in}{0.000000in}}{%
\pgfpathmoveto{\pgfqpoint{0.000000in}{0.000000in}}%
\pgfpathlineto{\pgfqpoint{0.000000in}{-0.048611in}}%
\pgfusepath{stroke,fill}%
}%
\begin{pgfscope}%
\pgfsys@transformshift{4.115030in}{0.611944in}%
\pgfsys@useobject{currentmarker}{}%
\end{pgfscope}%
\end{pgfscope}%
\begin{pgfscope}%
\definecolor{textcolor}{rgb}{0.000000,0.000000,0.000000}%
\pgfsetstrokecolor{textcolor}%
\pgfsetfillcolor{textcolor}%
\pgftext[x=4.115030in,y=0.485556in,,top]{\color{textcolor}{\sffamily\fontsize{8.000000}{9.600000}\selectfont\catcode`\^=\active\def^{\ifmmode\sp\else\^{}\fi}\catcode`\%=\active\def
\end{pgfscope}%
\begin{pgfscope}%
\pgfsetbuttcap%
\pgfsetroundjoin%
\definecolor{currentfill}{rgb}{0.000000,0.000000,0.000000}%
\pgfsetfillcolor{currentfill}%
\pgfsetlinewidth{0.803000pt}%
\definecolor{currentstroke}{rgb}{0.000000,0.000000,0.000000}%
\pgfsetstrokecolor{currentstroke}%
\pgfsetdash{}{0pt}%
\pgfsys@defobject{currentmarker}{\pgfqpoint{0.000000in}{-0.048611in}}{\pgfqpoint{0.000000in}{0.000000in}}{%
\pgfpathmoveto{\pgfqpoint{0.000000in}{0.000000in}}%
\pgfpathlineto{\pgfqpoint{0.000000in}{-0.048611in}}%
\pgfusepath{stroke,fill}%
}%
\begin{pgfscope}%
\pgfsys@transformshift{4.540206in}{0.611944in}%
\pgfsys@useobject{currentmarker}{}%
\end{pgfscope}%
\end{pgfscope}%
\begin{pgfscope}%
\definecolor{textcolor}{rgb}{0.000000,0.000000,0.000000}%
\pgfsetstrokecolor{textcolor}%
\pgfsetfillcolor{textcolor}%
\pgftext[x=4.540206in,y=0.485556in,,top]{\color{textcolor}{\sffamily\fontsize{8.000000}{9.600000}\selectfont\catcode`\^=\active\def^{\ifmmode\sp\else\^{}\fi}\catcode`\%=\active\def
\end{pgfscope}%
\begin{pgfscope}%
\pgfsetbuttcap%
\pgfsetroundjoin%
\definecolor{currentfill}{rgb}{0.000000,0.000000,0.000000}%
\pgfsetfillcolor{currentfill}%
\pgfsetlinewidth{0.803000pt}%
\definecolor{currentstroke}{rgb}{0.000000,0.000000,0.000000}%
\pgfsetstrokecolor{currentstroke}%
\pgfsetdash{}{0pt}%
\pgfsys@defobject{currentmarker}{\pgfqpoint{0.000000in}{-0.048611in}}{\pgfqpoint{0.000000in}{0.000000in}}{%
\pgfpathmoveto{\pgfqpoint{0.000000in}{0.000000in}}%
\pgfpathlineto{\pgfqpoint{0.000000in}{-0.048611in}}%
\pgfusepath{stroke,fill}%
}%
\begin{pgfscope}%
\pgfsys@transformshift{4.965383in}{0.611944in}%
\pgfsys@useobject{currentmarker}{}%
\end{pgfscope}%
\end{pgfscope}%
\begin{pgfscope}%
\definecolor{textcolor}{rgb}{0.000000,0.000000,0.000000}%
\pgfsetstrokecolor{textcolor}%
\pgfsetfillcolor{textcolor}%
\pgftext[x=4.965383in,y=0.485556in,,top]{\color{textcolor}{\sffamily\fontsize{8.000000}{9.600000}\selectfont\catcode`\^=\active\def^{\ifmmode\sp\else\^{}\fi}\catcode`\%=\active\def
\end{pgfscope}%
\begin{pgfscope}%
\pgfsetbuttcap%
\pgfsetroundjoin%
\definecolor{currentfill}{rgb}{0.000000,0.000000,0.000000}%
\pgfsetfillcolor{currentfill}%
\pgfsetlinewidth{0.803000pt}%
\definecolor{currentstroke}{rgb}{0.000000,0.000000,0.000000}%
\pgfsetstrokecolor{currentstroke}%
\pgfsetdash{}{0pt}%
\pgfsys@defobject{currentmarker}{\pgfqpoint{0.000000in}{-0.048611in}}{\pgfqpoint{0.000000in}{0.000000in}}{%
\pgfpathmoveto{\pgfqpoint{0.000000in}{0.000000in}}%
\pgfpathlineto{\pgfqpoint{0.000000in}{-0.048611in}}%
\pgfusepath{stroke,fill}%
}%
\begin{pgfscope}%
\pgfsys@transformshift{5.390559in}{0.611944in}%
\pgfsys@useobject{currentmarker}{}%
\end{pgfscope}%
\end{pgfscope}%
\begin{pgfscope}%
\definecolor{textcolor}{rgb}{0.000000,0.000000,0.000000}%
\pgfsetstrokecolor{textcolor}%
\pgfsetfillcolor{textcolor}%
\pgftext[x=5.390559in,y=0.485556in,,top]{\color{textcolor}{\sffamily\fontsize{8.000000}{9.600000}\selectfont\catcode`\^=\active\def^{\ifmmode\sp\else\^{}\fi}\catcode`\%=\active\def
\end{pgfscope}%
\begin{pgfscope}%
\pgfsetbuttcap%
\pgfsetroundjoin%
\definecolor{currentfill}{rgb}{0.000000,0.000000,0.000000}%
\pgfsetfillcolor{currentfill}%
\pgfsetlinewidth{0.803000pt}%
\definecolor{currentstroke}{rgb}{0.000000,0.000000,0.000000}%
\pgfsetstrokecolor{currentstroke}%
\pgfsetdash{}{0pt}%
\pgfsys@defobject{currentmarker}{\pgfqpoint{0.000000in}{-0.048611in}}{\pgfqpoint{0.000000in}{0.000000in}}{%
\pgfpathmoveto{\pgfqpoint{0.000000in}{0.000000in}}%
\pgfpathlineto{\pgfqpoint{0.000000in}{-0.048611in}}%
\pgfusepath{stroke,fill}%
}%
\begin{pgfscope}%
\pgfsys@transformshift{5.815735in}{0.611944in}%
\pgfsys@useobject{currentmarker}{}%
\end{pgfscope}%
\end{pgfscope}%
\begin{pgfscope}%
\definecolor{textcolor}{rgb}{0.000000,0.000000,0.000000}%
\pgfsetstrokecolor{textcolor}%
\pgfsetfillcolor{textcolor}%
\pgftext[x=5.815735in,y=0.485556in,,top]{\color{textcolor}{\sffamily\fontsize{8.000000}{9.600000}\selectfont\catcode`\^=\active\def^{\ifmmode\sp\else\^{}\fi}\catcode`\%=\active\def
\end{pgfscope}%
\begin{pgfscope}%
\pgfsetbuttcap%
\pgfsetroundjoin%
\definecolor{currentfill}{rgb}{0.000000,0.000000,0.000000}%
\pgfsetfillcolor{currentfill}%
\pgfsetlinewidth{0.803000pt}%
\definecolor{currentstroke}{rgb}{0.000000,0.000000,0.000000}%
\pgfsetstrokecolor{currentstroke}%
\pgfsetdash{}{0pt}%
\pgfsys@defobject{currentmarker}{\pgfqpoint{0.000000in}{-0.048611in}}{\pgfqpoint{0.000000in}{0.000000in}}{%
\pgfpathmoveto{\pgfqpoint{0.000000in}{0.000000in}}%
\pgfpathlineto{\pgfqpoint{0.000000in}{-0.048611in}}%
\pgfusepath{stroke,fill}%
}%
\begin{pgfscope}%
\pgfsys@transformshift{6.240912in}{0.611944in}%
\pgfsys@useobject{currentmarker}{}%
\end{pgfscope}%
\end{pgfscope}%
\begin{pgfscope}%
\definecolor{textcolor}{rgb}{0.000000,0.000000,0.000000}%
\pgfsetstrokecolor{textcolor}%
\pgfsetfillcolor{textcolor}%
\pgftext[x=6.240912in,y=0.485556in,,top]{\color{textcolor}{\sffamily\fontsize{8.000000}{9.600000}\selectfont\catcode`\^=\active\def^{\ifmmode\sp\else\^{}\fi}\catcode`\%=\active\def
\end{pgfscope}%
\begin{pgfscope}%
\definecolor{textcolor}{rgb}{0.000000,0.000000,0.000000}%
\pgfsetstrokecolor{textcolor}%
\pgfsetfillcolor{textcolor}%
\pgftext[x=5.177971in,y=0.331234in,,top]{\color{textcolor}{\sffamily\fontsize{10.000000}{12.000000}\selectfont\catcode`\^=\active\def^{\ifmmode\sp\else\^{}\fi}\catcode`\%=\active\def
\end{pgfscope}%
\begin{pgfscope}%
\pgfsetbuttcap%
\pgfsetroundjoin%
\definecolor{currentfill}{rgb}{0.000000,0.000000,0.000000}%
\pgfsetfillcolor{currentfill}%
\pgfsetlinewidth{0.803000pt}%
\definecolor{currentstroke}{rgb}{0.000000,0.000000,0.000000}%
\pgfsetstrokecolor{currentstroke}%
\pgfsetdash{}{0pt}%
\pgfsys@defobject{currentmarker}{\pgfqpoint{-0.048611in}{0.000000in}}{\pgfqpoint{-0.000000in}{0.000000in}}{%
\pgfpathmoveto{\pgfqpoint{-0.000000in}{0.000000in}}%
\pgfpathlineto{\pgfqpoint{-0.048611in}{0.000000in}}%
\pgfusepath{stroke,fill}%
}%
\begin{pgfscope}%
\pgfsys@transformshift{4.008736in}{0.611944in}%
\pgfsys@useobject{currentmarker}{}%
\end{pgfscope}%
\end{pgfscope}%
\begin{pgfscope}%
\definecolor{textcolor}{rgb}{0.000000,0.000000,0.000000}%
\pgfsetstrokecolor{textcolor}%
\pgfsetfillcolor{textcolor}%
\pgftext[x=3.672468in, y=0.573364in, left, base]{\color{textcolor}{\sffamily\fontsize{8.000000}{9.600000}\selectfont\catcode`\^=\active\def^{\ifmmode\sp\else\^{}\fi}\catcode`\%=\active\def
\end{pgfscope}%
\begin{pgfscope}%
\pgfsetbuttcap%
\pgfsetroundjoin%
\definecolor{currentfill}{rgb}{0.000000,0.000000,0.000000}%
\pgfsetfillcolor{currentfill}%
\pgfsetlinewidth{0.803000pt}%
\definecolor{currentstroke}{rgb}{0.000000,0.000000,0.000000}%
\pgfsetstrokecolor{currentstroke}%
\pgfsetdash{}{0pt}%
\pgfsys@defobject{currentmarker}{\pgfqpoint{-0.048611in}{0.000000in}}{\pgfqpoint{-0.000000in}{0.000000in}}{%
\pgfpathmoveto{\pgfqpoint{-0.000000in}{0.000000in}}%
\pgfpathlineto{\pgfqpoint{-0.048611in}{0.000000in}}%
\pgfusepath{stroke,fill}%
}%
\begin{pgfscope}%
\pgfsys@transformshift{4.008736in}{0.936905in}%
\pgfsys@useobject{currentmarker}{}%
\end{pgfscope}%
\end{pgfscope}%
\begin{pgfscope}%
\definecolor{textcolor}{rgb}{0.000000,0.000000,0.000000}%
\pgfsetstrokecolor{textcolor}%
\pgfsetfillcolor{textcolor}%
\pgftext[x=3.672468in, y=0.898324in, left, base]{\color{textcolor}{\sffamily\fontsize{8.000000}{9.600000}\selectfont\catcode`\^=\active\def^{\ifmmode\sp\else\^{}\fi}\catcode`\%=\active\def
\end{pgfscope}%
\begin{pgfscope}%
\pgfsetbuttcap%
\pgfsetroundjoin%
\definecolor{currentfill}{rgb}{0.000000,0.000000,0.000000}%
\pgfsetfillcolor{currentfill}%
\pgfsetlinewidth{0.803000pt}%
\definecolor{currentstroke}{rgb}{0.000000,0.000000,0.000000}%
\pgfsetstrokecolor{currentstroke}%
\pgfsetdash{}{0pt}%
\pgfsys@defobject{currentmarker}{\pgfqpoint{-0.048611in}{0.000000in}}{\pgfqpoint{-0.000000in}{0.000000in}}{%
\pgfpathmoveto{\pgfqpoint{-0.000000in}{0.000000in}}%
\pgfpathlineto{\pgfqpoint{-0.048611in}{0.000000in}}%
\pgfusepath{stroke,fill}%
}%
\begin{pgfscope}%
\pgfsys@transformshift{4.008736in}{1.261865in}%
\pgfsys@useobject{currentmarker}{}%
\end{pgfscope}%
\end{pgfscope}%
\begin{pgfscope}%
\definecolor{textcolor}{rgb}{0.000000,0.000000,0.000000}%
\pgfsetstrokecolor{textcolor}%
\pgfsetfillcolor{textcolor}%
\pgftext[x=3.672468in, y=1.223285in, left, base]{\color{textcolor}{\sffamily\fontsize{8.000000}{9.600000}\selectfont\catcode`\^=\active\def^{\ifmmode\sp\else\^{}\fi}\catcode`\%=\active\def
\end{pgfscope}%
\begin{pgfscope}%
\pgfsetbuttcap%
\pgfsetroundjoin%
\definecolor{currentfill}{rgb}{0.000000,0.000000,0.000000}%
\pgfsetfillcolor{currentfill}%
\pgfsetlinewidth{0.803000pt}%
\definecolor{currentstroke}{rgb}{0.000000,0.000000,0.000000}%
\pgfsetstrokecolor{currentstroke}%
\pgfsetdash{}{0pt}%
\pgfsys@defobject{currentmarker}{\pgfqpoint{-0.048611in}{0.000000in}}{\pgfqpoint{-0.000000in}{0.000000in}}{%
\pgfpathmoveto{\pgfqpoint{-0.000000in}{0.000000in}}%
\pgfpathlineto{\pgfqpoint{-0.048611in}{0.000000in}}%
\pgfusepath{stroke,fill}%
}%
\begin{pgfscope}%
\pgfsys@transformshift{4.008736in}{1.586825in}%
\pgfsys@useobject{currentmarker}{}%
\end{pgfscope}%
\end{pgfscope}%
\begin{pgfscope}%
\definecolor{textcolor}{rgb}{0.000000,0.000000,0.000000}%
\pgfsetstrokecolor{textcolor}%
\pgfsetfillcolor{textcolor}%
\pgftext[x=3.672468in, y=1.548245in, left, base]{\color{textcolor}{\sffamily\fontsize{8.000000}{9.600000}\selectfont\catcode`\^=\active\def^{\ifmmode\sp\else\^{}\fi}\catcode`\%=\active\def
\end{pgfscope}%
\begin{pgfscope}%
\pgfsetbuttcap%
\pgfsetroundjoin%
\definecolor{currentfill}{rgb}{0.000000,0.000000,0.000000}%
\pgfsetfillcolor{currentfill}%
\pgfsetlinewidth{0.803000pt}%
\definecolor{currentstroke}{rgb}{0.000000,0.000000,0.000000}%
\pgfsetstrokecolor{currentstroke}%
\pgfsetdash{}{0pt}%
\pgfsys@defobject{currentmarker}{\pgfqpoint{-0.048611in}{0.000000in}}{\pgfqpoint{-0.000000in}{0.000000in}}{%
\pgfpathmoveto{\pgfqpoint{-0.000000in}{0.000000in}}%
\pgfpathlineto{\pgfqpoint{-0.048611in}{0.000000in}}%
\pgfusepath{stroke,fill}%
}%
\begin{pgfscope}%
\pgfsys@transformshift{4.008736in}{1.911786in}%
\pgfsys@useobject{currentmarker}{}%
\end{pgfscope}%
\end{pgfscope}%
\begin{pgfscope}%
\definecolor{textcolor}{rgb}{0.000000,0.000000,0.000000}%
\pgfsetstrokecolor{textcolor}%
\pgfsetfillcolor{textcolor}%
\pgftext[x=3.672468in, y=1.873205in, left, base]{\color{textcolor}{\sffamily\fontsize{8.000000}{9.600000}\selectfont\catcode`\^=\active\def^{\ifmmode\sp\else\^{}\fi}\catcode`\%=\active\def
\end{pgfscope}%
\begin{pgfscope}%
\pgfsetbuttcap%
\pgfsetroundjoin%
\definecolor{currentfill}{rgb}{0.000000,0.000000,0.000000}%
\pgfsetfillcolor{currentfill}%
\pgfsetlinewidth{0.803000pt}%
\definecolor{currentstroke}{rgb}{0.000000,0.000000,0.000000}%
\pgfsetstrokecolor{currentstroke}%
\pgfsetdash{}{0pt}%
\pgfsys@defobject{currentmarker}{\pgfqpoint{-0.048611in}{0.000000in}}{\pgfqpoint{-0.000000in}{0.000000in}}{%
\pgfpathmoveto{\pgfqpoint{-0.000000in}{0.000000in}}%
\pgfpathlineto{\pgfqpoint{-0.048611in}{0.000000in}}%
\pgfusepath{stroke,fill}%
}%
\begin{pgfscope}%
\pgfsys@transformshift{4.008736in}{2.236746in}%
\pgfsys@useobject{currentmarker}{}%
\end{pgfscope}%
\end{pgfscope}%
\begin{pgfscope}%
\definecolor{textcolor}{rgb}{0.000000,0.000000,0.000000}%
\pgfsetstrokecolor{textcolor}%
\pgfsetfillcolor{textcolor}%
\pgftext[x=3.672468in, y=2.198166in, left, base]{\color{textcolor}{\sffamily\fontsize{8.000000}{9.600000}\selectfont\catcode`\^=\active\def^{\ifmmode\sp\else\^{}\fi}\catcode`\%=\active\def
\end{pgfscope}%
\begin{pgfscope}%
\pgfsetbuttcap%
\pgfsetroundjoin%
\definecolor{currentfill}{rgb}{0.000000,0.000000,0.000000}%
\pgfsetfillcolor{currentfill}%
\pgfsetlinewidth{0.803000pt}%
\definecolor{currentstroke}{rgb}{0.000000,0.000000,0.000000}%
\pgfsetstrokecolor{currentstroke}%
\pgfsetdash{}{0pt}%
\pgfsys@defobject{currentmarker}{\pgfqpoint{-0.048611in}{0.000000in}}{\pgfqpoint{-0.000000in}{0.000000in}}{%
\pgfpathmoveto{\pgfqpoint{-0.000000in}{0.000000in}}%
\pgfpathlineto{\pgfqpoint{-0.048611in}{0.000000in}}%
\pgfusepath{stroke,fill}%
}%
\begin{pgfscope}%
\pgfsys@transformshift{4.008736in}{2.561706in}%
\pgfsys@useobject{currentmarker}{}%
\end{pgfscope}%
\end{pgfscope}%
\begin{pgfscope}%
\definecolor{textcolor}{rgb}{0.000000,0.000000,0.000000}%
\pgfsetstrokecolor{textcolor}%
\pgfsetfillcolor{textcolor}%
\pgftext[x=3.672468in, y=2.523126in, left, base]{\color{textcolor}{\sffamily\fontsize{8.000000}{9.600000}\selectfont\catcode`\^=\active\def^{\ifmmode\sp\else\^{}\fi}\catcode`\%=\active\def
\end{pgfscope}%
\begin{pgfscope}%
\pgfsetbuttcap%
\pgfsetroundjoin%
\definecolor{currentfill}{rgb}{0.000000,0.000000,0.000000}%
\pgfsetfillcolor{currentfill}%
\pgfsetlinewidth{0.803000pt}%
\definecolor{currentstroke}{rgb}{0.000000,0.000000,0.000000}%
\pgfsetstrokecolor{currentstroke}%
\pgfsetdash{}{0pt}%
\pgfsys@defobject{currentmarker}{\pgfqpoint{-0.048611in}{0.000000in}}{\pgfqpoint{-0.000000in}{0.000000in}}{%
\pgfpathmoveto{\pgfqpoint{-0.000000in}{0.000000in}}%
\pgfpathlineto{\pgfqpoint{-0.048611in}{0.000000in}}%
\pgfusepath{stroke,fill}%
}%
\begin{pgfscope}%
\pgfsys@transformshift{4.008736in}{2.886667in}%
\pgfsys@useobject{currentmarker}{}%
\end{pgfscope}%
\end{pgfscope}%
\begin{pgfscope}%
\definecolor{textcolor}{rgb}{0.000000,0.000000,0.000000}%
\pgfsetstrokecolor{textcolor}%
\pgfsetfillcolor{textcolor}%
\pgftext[x=3.672468in, y=2.848086in, left, base]{\color{textcolor}{\sffamily\fontsize{8.000000}{9.600000}\selectfont\catcode`\^=\active\def^{\ifmmode\sp\else\^{}\fi}\catcode`\%=\active\def
\end{pgfscope}%
\begin{pgfscope}%
\definecolor{textcolor}{rgb}{0.000000,0.000000,0.000000}%
\pgfsetstrokecolor{textcolor}%
\pgfsetfillcolor{textcolor}%
\pgftext[x=3.616912in,y=1.749306in,,bottom,rotate=90.000000]{\color{textcolor}{\sffamily\fontsize{10.000000}{12.000000}\selectfont\catcode`\^=\active\def^{\ifmmode\sp\else\^{}\fi}\catcode`\%=\active\def
\end{pgfscope}%
\begin{pgfscope}%
\pgfpathrectangle{\pgfqpoint{4.008736in}{0.611944in}}{\pgfqpoint{2.338469in}{2.274722in}}%
\pgfusepath{clip}%
\pgfsetrectcap%
\pgfsetroundjoin%
\pgfsetlinewidth{1.405250pt}%
\definecolor{currentstroke}{rgb}{0.121569,0.466667,0.705882}%
\pgfsetstrokecolor{currentstroke}%
\pgfsetdash{}{0pt}%
\pgfpathmoveto{\pgfqpoint{4.115030in}{1.293794in}}%
\pgfpathlineto{\pgfqpoint{4.327618in}{1.203227in}}%
\pgfpathlineto{\pgfqpoint{4.540206in}{1.127609in}}%
\pgfpathlineto{\pgfqpoint{4.752795in}{1.063520in}}%
\pgfpathlineto{\pgfqpoint{4.965383in}{1.008511in}}%
\pgfpathlineto{\pgfqpoint{5.177971in}{0.960780in}}%
\pgfpathlineto{\pgfqpoint{5.390559in}{0.918972in}}%
\pgfpathlineto{\pgfqpoint{5.603147in}{0.882049in}}%
\pgfpathlineto{\pgfqpoint{5.815735in}{0.849202in}}%
\pgfpathlineto{\pgfqpoint{6.028323in}{0.819792in}}%
\pgfpathlineto{\pgfqpoint{6.240912in}{0.793305in}}%
\pgfusepath{stroke}%
\end{pgfscope}%
\begin{pgfscope}%
\pgfpathrectangle{\pgfqpoint{4.008736in}{0.611944in}}{\pgfqpoint{2.338469in}{2.274722in}}%
\pgfusepath{clip}%
\pgfsetrectcap%
\pgfsetroundjoin%
\pgfsetlinewidth{1.405250pt}%
\definecolor{currentstroke}{rgb}{1.000000,0.498039,0.054902}%
\pgfsetstrokecolor{currentstroke}%
\pgfsetdash{}{0pt}%
\pgfpathmoveto{\pgfqpoint{4.115030in}{1.111960in}}%
\pgfpathlineto{\pgfqpoint{4.327618in}{1.037605in}}%
\pgfpathlineto{\pgfqpoint{4.540206in}{0.975544in}}%
\pgfpathlineto{\pgfqpoint{4.752795in}{0.922962in}}%
\pgfpathlineto{\pgfqpoint{4.965383in}{0.877841in}}%
\pgfpathlineto{\pgfqpoint{5.177971in}{0.838698in}}%
\pgfpathlineto{\pgfqpoint{5.390559in}{0.804419in}}%
\pgfpathlineto{\pgfqpoint{5.603147in}{0.774151in}}%
\pgfpathlineto{\pgfqpoint{5.815735in}{0.747228in}}%
\pgfpathlineto{\pgfqpoint{6.028323in}{0.723125in}}%
\pgfpathlineto{\pgfqpoint{6.240912in}{0.701421in}}%
\pgfusepath{stroke}%
\end{pgfscope}%
\begin{pgfscope}%
\pgfpathrectangle{\pgfqpoint{4.008736in}{0.611944in}}{\pgfqpoint{2.338469in}{2.274722in}}%
\pgfusepath{clip}%
\pgfsetrectcap%
\pgfsetroundjoin%
\pgfsetlinewidth{1.405250pt}%
\definecolor{currentstroke}{rgb}{0.172549,0.627451,0.172549}%
\pgfsetstrokecolor{currentstroke}%
\pgfsetdash{}{0pt}%
\pgfpathmoveto{\pgfqpoint{4.115030in}{1.002804in}}%
\pgfpathlineto{\pgfqpoint{4.327618in}{0.938212in}}%
\pgfpathlineto{\pgfqpoint{4.540206in}{0.884313in}}%
\pgfpathlineto{\pgfqpoint{4.752795in}{0.838653in}}%
\pgfpathlineto{\pgfqpoint{4.965383in}{0.799478in}}%
\pgfpathlineto{\pgfqpoint{5.177971in}{0.765497in}}%
\pgfpathlineto{\pgfqpoint{5.390559in}{0.735743in}}%
\pgfpathlineto{\pgfqpoint{5.603147in}{0.709472in}}%
\pgfpathlineto{\pgfqpoint{5.815735in}{0.686107in}}%
\pgfpathlineto{\pgfqpoint{6.028323in}{0.665191in}}%
\pgfpathlineto{\pgfqpoint{6.240912in}{0.646356in}}%
\pgfusepath{stroke}%
\end{pgfscope}%
\begin{pgfscope}%
\pgfsetrectcap%
\pgfsetmiterjoin%
\pgfsetlinewidth{0.803000pt}%
\definecolor{currentstroke}{rgb}{0.000000,0.000000,0.000000}%
\pgfsetstrokecolor{currentstroke}%
\pgfsetdash{}{0pt}%
\pgfpathmoveto{\pgfqpoint{4.008736in}{0.611944in}}%
\pgfpathlineto{\pgfqpoint{4.008736in}{2.886667in}}%
\pgfusepath{stroke}%
\end{pgfscope}%
\begin{pgfscope}%
\pgfsetrectcap%
\pgfsetmiterjoin%
\pgfsetlinewidth{0.803000pt}%
\definecolor{currentstroke}{rgb}{0.000000,0.000000,0.000000}%
\pgfsetstrokecolor{currentstroke}%
\pgfsetdash{}{0pt}%
\pgfpathmoveto{\pgfqpoint{6.347206in}{0.611944in}}%
\pgfpathlineto{\pgfqpoint{6.347206in}{2.886667in}}%
\pgfusepath{stroke}%
\end{pgfscope}%
\begin{pgfscope}%
\pgfsetrectcap%
\pgfsetmiterjoin%
\pgfsetlinewidth{0.803000pt}%
\definecolor{currentstroke}{rgb}{0.000000,0.000000,0.000000}%
\pgfsetstrokecolor{currentstroke}%
\pgfsetdash{}{0pt}%
\pgfpathmoveto{\pgfqpoint{4.008736in}{0.611944in}}%
\pgfpathlineto{\pgfqpoint{6.347206in}{0.611944in}}%
\pgfusepath{stroke}%
\end{pgfscope}%
\begin{pgfscope}%
\pgfsetrectcap%
\pgfsetmiterjoin%
\pgfsetlinewidth{0.803000pt}%
\definecolor{currentstroke}{rgb}{0.000000,0.000000,0.000000}%
\pgfsetstrokecolor{currentstroke}%
\pgfsetdash{}{0pt}%
\pgfpathmoveto{\pgfqpoint{4.008736in}{2.886667in}}%
\pgfpathlineto{\pgfqpoint{6.347206in}{2.886667in}}%
\pgfusepath{stroke}%
\end{pgfscope}%
\begin{pgfscope}%
\definecolor{textcolor}{rgb}{0.000000,0.000000,0.000000}%
\pgfsetstrokecolor{textcolor}%
\pgfsetfillcolor{textcolor}%
\pgftext[x=5.177971in,y=2.970000in,,base]{\color{textcolor}{\sffamily\fontsize{10.000000}{12.000000}\selectfont\catcode`\^=\active\def^{\ifmmode\sp\else\^{}\fi}\catcode`\%=\active\def
\end{pgfscope}%
\begin{pgfscope}%
\pgfsetbuttcap%
\pgfsetmiterjoin%
\definecolor{currentfill}{rgb}{1.000000,1.000000,1.000000}%
\pgfsetfillcolor{currentfill}%
\pgfsetfillopacity{0.800000}%
\pgfsetlinewidth{0.240900pt}%
\definecolor{currentstroke}{rgb}{0.800000,0.800000,0.800000}%
\pgfsetstrokecolor{currentstroke}%
\pgfsetstrokeopacity{0.800000}%
\pgfsetdash{}{0pt}%
\pgfpathmoveto{\pgfqpoint{5.461319in}{2.323704in}}%
\pgfpathlineto{\pgfqpoint{6.269428in}{2.323704in}}%
\pgfpathquadraticcurveto{\pgfqpoint{6.291650in}{2.323704in}}{\pgfqpoint{6.291650in}{2.345926in}}%
\pgfpathlineto{\pgfqpoint{6.291650in}{2.808889in}}%
\pgfpathquadraticcurveto{\pgfqpoint{6.291650in}{2.831111in}}{\pgfqpoint{6.269428in}{2.831111in}}%
\pgfpathlineto{\pgfqpoint{5.461319in}{2.831111in}}%
\pgfpathquadraticcurveto{\pgfqpoint{5.439097in}{2.831111in}}{\pgfqpoint{5.439097in}{2.808889in}}%
\pgfpathlineto{\pgfqpoint{5.439097in}{2.345926in}}%
\pgfpathquadraticcurveto{\pgfqpoint{5.439097in}{2.323704in}}{\pgfqpoint{5.461319in}{2.323704in}}%
\pgfpathlineto{\pgfqpoint{5.461319in}{2.323704in}}%
\pgfpathclose%
\pgfusepath{stroke,fill}%
\end{pgfscope}%
\begin{pgfscope}%
\pgfsetrectcap%
\pgfsetroundjoin%
\pgfsetlinewidth{1.405250pt}%
\definecolor{currentstroke}{rgb}{0.121569,0.466667,0.705882}%
\pgfsetstrokecolor{currentstroke}%
\pgfsetdash{}{0pt}%
\pgfpathmoveto{\pgfqpoint{5.483542in}{2.747778in}}%
\pgfpathlineto{\pgfqpoint{5.594653in}{2.747778in}}%
\pgfpathlineto{\pgfqpoint{5.705764in}{2.747778in}}%
\pgfusepath{stroke}%
\end{pgfscope}%
\begin{pgfscope}%
\definecolor{textcolor}{rgb}{0.000000,0.000000,0.000000}%
\pgfsetstrokecolor{textcolor}%
\pgfsetfillcolor{textcolor}%
\pgftext[x=5.794653in,y=2.708889in,left,base]{\color{textcolor}{\sffamily\fontsize{8.000000}{9.600000}\selectfont\catcode`\^=\active\def^{\ifmmode\sp\else\^{}\fi}\catcode`\%=\active\def
\end{pgfscope}%
\begin{pgfscope}%
\pgfsetrectcap%
\pgfsetroundjoin%
\pgfsetlinewidth{1.405250pt}%
\definecolor{currentstroke}{rgb}{1.000000,0.498039,0.054902}%
\pgfsetstrokecolor{currentstroke}%
\pgfsetdash{}{0pt}%
\pgfpathmoveto{\pgfqpoint{5.483542in}{2.589753in}}%
\pgfpathlineto{\pgfqpoint{5.594653in}{2.589753in}}%
\pgfpathlineto{\pgfqpoint{5.705764in}{2.589753in}}%
\pgfusepath{stroke}%
\end{pgfscope}%
\begin{pgfscope}%
\definecolor{textcolor}{rgb}{0.000000,0.000000,0.000000}%
\pgfsetstrokecolor{textcolor}%
\pgfsetfillcolor{textcolor}%
\pgftext[x=5.794653in,y=2.550864in,left,base]{\color{textcolor}{\sffamily\fontsize{8.000000}{9.600000}\selectfont\catcode`\^=\active\def^{\ifmmode\sp\else\^{}\fi}\catcode`\%=\active\def
\end{pgfscope}%
\begin{pgfscope}%
\pgfsetrectcap%
\pgfsetroundjoin%
\pgfsetlinewidth{1.405250pt}%
\definecolor{currentstroke}{rgb}{0.172549,0.627451,0.172549}%
\pgfsetstrokecolor{currentstroke}%
\pgfsetdash{}{0pt}%
\pgfpathmoveto{\pgfqpoint{5.483542in}{2.431729in}}%
\pgfpathlineto{\pgfqpoint{5.594653in}{2.431729in}}%
\pgfpathlineto{\pgfqpoint{5.705764in}{2.431729in}}%
\pgfusepath{stroke}%
\end{pgfscope}%
\begin{pgfscope}%
\definecolor{textcolor}{rgb}{0.000000,0.000000,0.000000}%
\pgfsetstrokecolor{textcolor}%
\pgfsetfillcolor{textcolor}%
\pgftext[x=5.794653in,y=2.392840in,left,base]{\color{textcolor}{\sffamily\fontsize{8.000000}{9.600000}\selectfont\catcode`\^=\active\def^{\ifmmode\sp\else\^{}\fi}\catcode`\%=\active\def
\end{pgfscope}%
\end{pgfpicture}%
\makeatother%
\endgroup%

%% file: figs/plot-weak.pgf
\begingroup%
\makeatletter%
\begin{pgfpicture}%
\pgfpathrectangle{\pgfpointorigin}{\pgfqpoint{5.200000in}{2.600000in}}%
\pgfusepath{use as bounding box, clip}%
\begin{pgfscope}%
\pgfsetbuttcap%
\pgfsetmiterjoin%
\definecolor{currentfill}{rgb}{1.000000,1.000000,1.000000}%
\pgfsetfillcolor{currentfill}%
\pgfsetlinewidth{0.000000pt}%
\definecolor{currentstroke}{rgb}{1.000000,1.000000,1.000000}%
\pgfsetstrokecolor{currentstroke}%
\pgfsetdash{}{0pt}%
\pgfpathmoveto{\pgfqpoint{0.000000in}{0.000000in}}%
\pgfpathlineto{\pgfqpoint{5.200000in}{0.000000in}}%
\pgfpathlineto{\pgfqpoint{5.200000in}{2.600000in}}%
\pgfpathlineto{\pgfqpoint{0.000000in}{2.600000in}}%
\pgfpathlineto{\pgfqpoint{0.000000in}{0.000000in}}%
\pgfpathclose%
\pgfusepath{fill}%
\end{pgfscope}%
\begin{pgfscope}%
\pgfsetbuttcap%
\pgfsetmiterjoin%
\definecolor{currentfill}{rgb}{1.000000,1.000000,1.000000}%
\pgfsetfillcolor{currentfill}%
\pgfsetlinewidth{0.000000pt}%
\definecolor{currentstroke}{rgb}{0.000000,0.000000,0.000000}%
\pgfsetstrokecolor{currentstroke}%
\pgfsetstrokeopacity{0.000000}%
\pgfsetdash{}{0pt}%
\pgfpathmoveto{\pgfqpoint{0.614444in}{0.611944in}}%
\pgfpathlineto{\pgfqpoint{4.585556in}{0.611944in}}%
\pgfpathlineto{\pgfqpoint{4.585556in}{2.410000in}}%
\pgfpathlineto{\pgfqpoint{0.614444in}{2.410000in}}%
\pgfpathlineto{\pgfqpoint{0.614444in}{0.611944in}}%
\pgfpathclose%
\pgfusepath{fill}%
\end{pgfscope}%
\begin{pgfscope}%
\pgfsetbuttcap%
\pgfsetroundjoin%
\definecolor{currentfill}{rgb}{0.000000,0.000000,0.000000}%
\pgfsetfillcolor{currentfill}%
\pgfsetlinewidth{0.803000pt}%
\definecolor{currentstroke}{rgb}{0.000000,0.000000,0.000000}%
\pgfsetstrokecolor{currentstroke}%
\pgfsetdash{}{0pt}%
\pgfsys@defobject{currentmarker}{\pgfqpoint{0.000000in}{-0.048611in}}{\pgfqpoint{0.000000in}{0.000000in}}{%
\pgfpathmoveto{\pgfqpoint{0.000000in}{0.000000in}}%
\pgfpathlineto{\pgfqpoint{0.000000in}{-0.048611in}}%
\pgfusepath{stroke,fill}%
}%
\begin{pgfscope}%
\pgfsys@transformshift{0.898095in}{0.611944in}%
\pgfsys@useobject{currentmarker}{}%
\end{pgfscope}%
\end{pgfscope}%
\begin{pgfscope}%
\definecolor{textcolor}{rgb}{0.000000,0.000000,0.000000}%
\pgfsetstrokecolor{textcolor}%
\pgfsetfillcolor{textcolor}%
\pgftext[x=0.898095in,y=0.485556in,,top]{\color{textcolor}{\sffamily\fontsize{8.000000}{9.600000}\selectfont\catcode`\^=\active\def^{\ifmmode\sp\else\^{}\fi}\catcode`\%=\active\def
\end{pgfscope}%
\begin{pgfscope}%
\pgfsetbuttcap%
\pgfsetroundjoin%
\definecolor{currentfill}{rgb}{0.000000,0.000000,0.000000}%
\pgfsetfillcolor{currentfill}%
\pgfsetlinewidth{0.803000pt}%
\definecolor{currentstroke}{rgb}{0.000000,0.000000,0.000000}%
\pgfsetstrokecolor{currentstroke}%
\pgfsetdash{}{0pt}%
\pgfsys@defobject{currentmarker}{\pgfqpoint{0.000000in}{-0.048611in}}{\pgfqpoint{0.000000in}{0.000000in}}{%
\pgfpathmoveto{\pgfqpoint{0.000000in}{0.000000in}}%
\pgfpathlineto{\pgfqpoint{0.000000in}{-0.048611in}}%
\pgfusepath{stroke,fill}%
}%
\begin{pgfscope}%
\pgfsys@transformshift{1.465397in}{0.611944in}%
\pgfsys@useobject{currentmarker}{}%
\end{pgfscope}%
\end{pgfscope}%
\begin{pgfscope}%
\definecolor{textcolor}{rgb}{0.000000,0.000000,0.000000}%
\pgfsetstrokecolor{textcolor}%
\pgfsetfillcolor{textcolor}%
\pgftext[x=1.465397in,y=0.485556in,,top]{\color{textcolor}{\sffamily\fontsize{8.000000}{9.600000}\selectfont\catcode`\^=\active\def^{\ifmmode\sp\else\^{}\fi}\catcode`\%=\active\def
\end{pgfscope}%
\begin{pgfscope}%
\pgfsetbuttcap%
\pgfsetroundjoin%
\definecolor{currentfill}{rgb}{0.000000,0.000000,0.000000}%
\pgfsetfillcolor{currentfill}%
\pgfsetlinewidth{0.803000pt}%
\definecolor{currentstroke}{rgb}{0.000000,0.000000,0.000000}%
\pgfsetstrokecolor{currentstroke}%
\pgfsetdash{}{0pt}%
\pgfsys@defobject{currentmarker}{\pgfqpoint{0.000000in}{-0.048611in}}{\pgfqpoint{0.000000in}{0.000000in}}{%
\pgfpathmoveto{\pgfqpoint{0.000000in}{0.000000in}}%
\pgfpathlineto{\pgfqpoint{0.000000in}{-0.048611in}}%
\pgfusepath{stroke,fill}%
}%
\begin{pgfscope}%
\pgfsys@transformshift{2.032698in}{0.611944in}%
\pgfsys@useobject{currentmarker}{}%
\end{pgfscope}%
\end{pgfscope}%
\begin{pgfscope}%
\definecolor{textcolor}{rgb}{0.000000,0.000000,0.000000}%
\pgfsetstrokecolor{textcolor}%
\pgfsetfillcolor{textcolor}%
\pgftext[x=2.032698in,y=0.485556in,,top]{\color{textcolor}{\sffamily\fontsize{8.000000}{9.600000}\selectfont\catcode`\^=\active\def^{\ifmmode\sp\else\^{}\fi}\catcode`\%=\active\def
\end{pgfscope}%
\begin{pgfscope}%
\pgfsetbuttcap%
\pgfsetroundjoin%
\definecolor{currentfill}{rgb}{0.000000,0.000000,0.000000}%
\pgfsetfillcolor{currentfill}%
\pgfsetlinewidth{0.803000pt}%
\definecolor{currentstroke}{rgb}{0.000000,0.000000,0.000000}%
\pgfsetstrokecolor{currentstroke}%
\pgfsetdash{}{0pt}%
\pgfsys@defobject{currentmarker}{\pgfqpoint{0.000000in}{-0.048611in}}{\pgfqpoint{0.000000in}{0.000000in}}{%
\pgfpathmoveto{\pgfqpoint{0.000000in}{0.000000in}}%
\pgfpathlineto{\pgfqpoint{0.000000in}{-0.048611in}}%
\pgfusepath{stroke,fill}%
}%
\begin{pgfscope}%
\pgfsys@transformshift{2.600000in}{0.611944in}%
\pgfsys@useobject{currentmarker}{}%
\end{pgfscope}%
\end{pgfscope}%
\begin{pgfscope}%
\definecolor{textcolor}{rgb}{0.000000,0.000000,0.000000}%
\pgfsetstrokecolor{textcolor}%
\pgfsetfillcolor{textcolor}%
\pgftext[x=2.600000in,y=0.485556in,,top]{\color{textcolor}{\sffamily\fontsize{8.000000}{9.600000}\selectfont\catcode`\^=\active\def^{\ifmmode\sp\else\^{}\fi}\catcode`\%=\active\def
\end{pgfscope}%
\begin{pgfscope}%
\pgfsetbuttcap%
\pgfsetroundjoin%
\definecolor{currentfill}{rgb}{0.000000,0.000000,0.000000}%
\pgfsetfillcolor{currentfill}%
\pgfsetlinewidth{0.803000pt}%
\definecolor{currentstroke}{rgb}{0.000000,0.000000,0.000000}%
\pgfsetstrokecolor{currentstroke}%
\pgfsetdash{}{0pt}%
\pgfsys@defobject{currentmarker}{\pgfqpoint{0.000000in}{-0.048611in}}{\pgfqpoint{0.000000in}{0.000000in}}{%
\pgfpathmoveto{\pgfqpoint{0.000000in}{0.000000in}}%
\pgfpathlineto{\pgfqpoint{0.000000in}{-0.048611in}}%
\pgfusepath{stroke,fill}%
}%
\begin{pgfscope}%
\pgfsys@transformshift{3.167302in}{0.611944in}%
\pgfsys@useobject{currentmarker}{}%
\end{pgfscope}%
\end{pgfscope}%
\begin{pgfscope}%
\definecolor{textcolor}{rgb}{0.000000,0.000000,0.000000}%
\pgfsetstrokecolor{textcolor}%
\pgfsetfillcolor{textcolor}%
\pgftext[x=3.167302in,y=0.485556in,,top]{\color{textcolor}{\sffamily\fontsize{8.000000}{9.600000}\selectfont\catcode`\^=\active\def^{\ifmmode\sp\else\^{}\fi}\catcode`\%=\active\def
\end{pgfscope}%
\begin{pgfscope}%
\pgfsetbuttcap%
\pgfsetroundjoin%
\definecolor{currentfill}{rgb}{0.000000,0.000000,0.000000}%
\pgfsetfillcolor{currentfill}%
\pgfsetlinewidth{0.803000pt}%
\definecolor{currentstroke}{rgb}{0.000000,0.000000,0.000000}%
\pgfsetstrokecolor{currentstroke}%
\pgfsetdash{}{0pt}%
\pgfsys@defobject{currentmarker}{\pgfqpoint{0.000000in}{-0.048611in}}{\pgfqpoint{0.000000in}{0.000000in}}{%
\pgfpathmoveto{\pgfqpoint{0.000000in}{0.000000in}}%
\pgfpathlineto{\pgfqpoint{0.000000in}{-0.048611in}}%
\pgfusepath{stroke,fill}%
}%
\begin{pgfscope}%
\pgfsys@transformshift{3.734603in}{0.611944in}%
\pgfsys@useobject{currentmarker}{}%
\end{pgfscope}%
\end{pgfscope}%
\begin{pgfscope}%
\definecolor{textcolor}{rgb}{0.000000,0.000000,0.000000}%
\pgfsetstrokecolor{textcolor}%
\pgfsetfillcolor{textcolor}%
\pgftext[x=3.734603in,y=0.485556in,,top]{\color{textcolor}{\sffamily\fontsize{8.000000}{9.600000}\selectfont\catcode`\^=\active\def^{\ifmmode\sp\else\^{}\fi}\catcode`\%=\active\def
\end{pgfscope}%
\begin{pgfscope}%
\pgfsetbuttcap%
\pgfsetroundjoin%
\definecolor{currentfill}{rgb}{0.000000,0.000000,0.000000}%
\pgfsetfillcolor{currentfill}%
\pgfsetlinewidth{0.803000pt}%
\definecolor{currentstroke}{rgb}{0.000000,0.000000,0.000000}%
\pgfsetstrokecolor{currentstroke}%
\pgfsetdash{}{0pt}%
\pgfsys@defobject{currentmarker}{\pgfqpoint{0.000000in}{-0.048611in}}{\pgfqpoint{0.000000in}{0.000000in}}{%
\pgfpathmoveto{\pgfqpoint{0.000000in}{0.000000in}}%
\pgfpathlineto{\pgfqpoint{0.000000in}{-0.048611in}}%
\pgfusepath{stroke,fill}%
}%
\begin{pgfscope}%
\pgfsys@transformshift{4.301905in}{0.611944in}%
\pgfsys@useobject{currentmarker}{}%
\end{pgfscope}%
\end{pgfscope}%
\begin{pgfscope}%
\definecolor{textcolor}{rgb}{0.000000,0.000000,0.000000}%
\pgfsetstrokecolor{textcolor}%
\pgfsetfillcolor{textcolor}%
\pgftext[x=4.301905in,y=0.485556in,,top]{\color{textcolor}{\sffamily\fontsize{8.000000}{9.600000}\selectfont\catcode`\^=\active\def^{\ifmmode\sp\else\^{}\fi}\catcode`\%=\active\def
\end{pgfscope}%
\begin{pgfscope}%
\definecolor{textcolor}{rgb}{0.000000,0.000000,0.000000}%
\pgfsetstrokecolor{textcolor}%
\pgfsetfillcolor{textcolor}%
\pgftext[x=2.600000in,y=0.331234in,,top]{\color{textcolor}{\sffamily\fontsize{10.000000}{12.000000}\selectfont\catcode`\^=\active\def^{\ifmmode\sp\else\^{}\fi}\catcode`\%=\active\def
\end{pgfscope}%
\begin{pgfscope}%
\pgfsetbuttcap%
\pgfsetroundjoin%
\definecolor{currentfill}{rgb}{0.000000,0.000000,0.000000}%
\pgfsetfillcolor{currentfill}%
\pgfsetlinewidth{0.803000pt}%
\definecolor{currentstroke}{rgb}{0.000000,0.000000,0.000000}%
\pgfsetstrokecolor{currentstroke}%
\pgfsetdash{}{0pt}%
\pgfsys@defobject{currentmarker}{\pgfqpoint{-0.048611in}{0.000000in}}{\pgfqpoint{-0.000000in}{0.000000in}}{%
\pgfpathmoveto{\pgfqpoint{-0.000000in}{0.000000in}}%
\pgfpathlineto{\pgfqpoint{-0.048611in}{0.000000in}}%
\pgfusepath{stroke,fill}%
}%
\begin{pgfscope}%
\pgfsys@transformshift{0.614444in}{0.611944in}%
\pgfsys@useobject{currentmarker}{}%
\end{pgfscope}%
\end{pgfscope}%
\begin{pgfscope}%
\definecolor{textcolor}{rgb}{0.839216,0.152941,0.156863}%
\pgfsetstrokecolor{textcolor}%
\pgfsetfillcolor{textcolor}%
\pgftext[x=0.429027in, y=0.573364in, left, base]{\color{textcolor}{\sffamily\fontsize{8.000000}{9.600000}\selectfont\catcode`\^=\active\def^{\ifmmode\sp\else\^{}\fi}\catcode`\%=\active\def
\end{pgfscope}%
\begin{pgfscope}%
\pgfsetbuttcap%
\pgfsetroundjoin%
\definecolor{currentfill}{rgb}{0.000000,0.000000,0.000000}%
\pgfsetfillcolor{currentfill}%
\pgfsetlinewidth{0.803000pt}%
\definecolor{currentstroke}{rgb}{0.000000,0.000000,0.000000}%
\pgfsetstrokecolor{currentstroke}%
\pgfsetdash{}{0pt}%
\pgfsys@defobject{currentmarker}{\pgfqpoint{-0.048611in}{0.000000in}}{\pgfqpoint{-0.000000in}{0.000000in}}{%
\pgfpathmoveto{\pgfqpoint{-0.000000in}{0.000000in}}%
\pgfpathlineto{\pgfqpoint{-0.048611in}{0.000000in}}%
\pgfusepath{stroke,fill}%
}%
\begin{pgfscope}%
\pgfsys@transformshift{0.614444in}{0.836701in}%
\pgfsys@useobject{currentmarker}{}%
\end{pgfscope}%
\end{pgfscope}%
\begin{pgfscope}%
\definecolor{textcolor}{rgb}{0.839216,0.152941,0.156863}%
\pgfsetstrokecolor{textcolor}%
\pgfsetfillcolor{textcolor}%
\pgftext[x=0.369998in, y=0.798121in, left, base]{\color{textcolor}{\sffamily\fontsize{8.000000}{9.600000}\selectfont\catcode`\^=\active\def^{\ifmmode\sp\else\^{}\fi}\catcode`\%=\active\def
\end{pgfscope}%
\begin{pgfscope}%
\pgfsetbuttcap%
\pgfsetroundjoin%
\definecolor{currentfill}{rgb}{0.000000,0.000000,0.000000}%
\pgfsetfillcolor{currentfill}%
\pgfsetlinewidth{0.803000pt}%
\definecolor{currentstroke}{rgb}{0.000000,0.000000,0.000000}%
\pgfsetstrokecolor{currentstroke}%
\pgfsetdash{}{0pt}%
\pgfsys@defobject{currentmarker}{\pgfqpoint{-0.048611in}{0.000000in}}{\pgfqpoint{-0.000000in}{0.000000in}}{%
\pgfpathmoveto{\pgfqpoint{-0.000000in}{0.000000in}}%
\pgfpathlineto{\pgfqpoint{-0.048611in}{0.000000in}}%
\pgfusepath{stroke,fill}%
}%
\begin{pgfscope}%
\pgfsys@transformshift{0.614444in}{1.061458in}%
\pgfsys@useobject{currentmarker}{}%
\end{pgfscope}%
\end{pgfscope}%
\begin{pgfscope}%
\definecolor{textcolor}{rgb}{0.839216,0.152941,0.156863}%
\pgfsetstrokecolor{textcolor}%
\pgfsetfillcolor{textcolor}%
\pgftext[x=0.369998in, y=1.022878in, left, base]{\color{textcolor}{\sffamily\fontsize{8.000000}{9.600000}\selectfont\catcode`\^=\active\def^{\ifmmode\sp\else\^{}\fi}\catcode`\%=\active\def
\end{pgfscope}%
\begin{pgfscope}%
\pgfsetbuttcap%
\pgfsetroundjoin%
\definecolor{currentfill}{rgb}{0.000000,0.000000,0.000000}%
\pgfsetfillcolor{currentfill}%
\pgfsetlinewidth{0.803000pt}%
\definecolor{currentstroke}{rgb}{0.000000,0.000000,0.000000}%
\pgfsetstrokecolor{currentstroke}%
\pgfsetdash{}{0pt}%
\pgfsys@defobject{currentmarker}{\pgfqpoint{-0.048611in}{0.000000in}}{\pgfqpoint{-0.000000in}{0.000000in}}{%
\pgfpathmoveto{\pgfqpoint{-0.000000in}{0.000000in}}%
\pgfpathlineto{\pgfqpoint{-0.048611in}{0.000000in}}%
\pgfusepath{stroke,fill}%
}%
\begin{pgfscope}%
\pgfsys@transformshift{0.614444in}{1.286215in}%
\pgfsys@useobject{currentmarker}{}%
\end{pgfscope}%
\end{pgfscope}%
\begin{pgfscope}%
\definecolor{textcolor}{rgb}{0.839216,0.152941,0.156863}%
\pgfsetstrokecolor{textcolor}%
\pgfsetfillcolor{textcolor}%
\pgftext[x=0.369998in, y=1.247635in, left, base]{\color{textcolor}{\sffamily\fontsize{8.000000}{9.600000}\selectfont\catcode`\^=\active\def^{\ifmmode\sp\else\^{}\fi}\catcode`\%=\active\def
\end{pgfscope}%
\begin{pgfscope}%
\pgfsetbuttcap%
\pgfsetroundjoin%
\definecolor{currentfill}{rgb}{0.000000,0.000000,0.000000}%
\pgfsetfillcolor{currentfill}%
\pgfsetlinewidth{0.803000pt}%
\definecolor{currentstroke}{rgb}{0.000000,0.000000,0.000000}%
\pgfsetstrokecolor{currentstroke}%
\pgfsetdash{}{0pt}%
\pgfsys@defobject{currentmarker}{\pgfqpoint{-0.048611in}{0.000000in}}{\pgfqpoint{-0.000000in}{0.000000in}}{%
\pgfpathmoveto{\pgfqpoint{-0.000000in}{0.000000in}}%
\pgfpathlineto{\pgfqpoint{-0.048611in}{0.000000in}}%
\pgfusepath{stroke,fill}%
}%
\begin{pgfscope}%
\pgfsys@transformshift{0.614444in}{1.510972in}%
\pgfsys@useobject{currentmarker}{}%
\end{pgfscope}%
\end{pgfscope}%
\begin{pgfscope}%
\definecolor{textcolor}{rgb}{0.839216,0.152941,0.156863}%
\pgfsetstrokecolor{textcolor}%
\pgfsetfillcolor{textcolor}%
\pgftext[x=0.369998in, y=1.472392in, left, base]{\color{textcolor}{\sffamily\fontsize{8.000000}{9.600000}\selectfont\catcode`\^=\active\def^{\ifmmode\sp\else\^{}\fi}\catcode`\%=\active\def
\end{pgfscope}%
\begin{pgfscope}%
\pgfsetbuttcap%
\pgfsetroundjoin%
\definecolor{currentfill}{rgb}{0.000000,0.000000,0.000000}%
\pgfsetfillcolor{currentfill}%
\pgfsetlinewidth{0.803000pt}%
\definecolor{currentstroke}{rgb}{0.000000,0.000000,0.000000}%
\pgfsetstrokecolor{currentstroke}%
\pgfsetdash{}{0pt}%
\pgfsys@defobject{currentmarker}{\pgfqpoint{-0.048611in}{0.000000in}}{\pgfqpoint{-0.000000in}{0.000000in}}{%
\pgfpathmoveto{\pgfqpoint{-0.000000in}{0.000000in}}%
\pgfpathlineto{\pgfqpoint{-0.048611in}{0.000000in}}%
\pgfusepath{stroke,fill}%
}%
\begin{pgfscope}%
\pgfsys@transformshift{0.614444in}{1.735729in}%
\pgfsys@useobject{currentmarker}{}%
\end{pgfscope}%
\end{pgfscope}%
\begin{pgfscope}%
\definecolor{textcolor}{rgb}{0.839216,0.152941,0.156863}%
\pgfsetstrokecolor{textcolor}%
\pgfsetfillcolor{textcolor}%
\pgftext[x=0.369998in, y=1.697149in, left, base]{\color{textcolor}{\sffamily\fontsize{8.000000}{9.600000}\selectfont\catcode`\^=\active\def^{\ifmmode\sp\else\^{}\fi}\catcode`\%=\active\def
\end{pgfscope}%
\begin{pgfscope}%
\pgfsetbuttcap%
\pgfsetroundjoin%
\definecolor{currentfill}{rgb}{0.000000,0.000000,0.000000}%
\pgfsetfillcolor{currentfill}%
\pgfsetlinewidth{0.803000pt}%
\definecolor{currentstroke}{rgb}{0.000000,0.000000,0.000000}%
\pgfsetstrokecolor{currentstroke}%
\pgfsetdash{}{0pt}%
\pgfsys@defobject{currentmarker}{\pgfqpoint{-0.048611in}{0.000000in}}{\pgfqpoint{-0.000000in}{0.000000in}}{%
\pgfpathmoveto{\pgfqpoint{-0.000000in}{0.000000in}}%
\pgfpathlineto{\pgfqpoint{-0.048611in}{0.000000in}}%
\pgfusepath{stroke,fill}%
}%
\begin{pgfscope}%
\pgfsys@transformshift{0.614444in}{1.960486in}%
\pgfsys@useobject{currentmarker}{}%
\end{pgfscope}%
\end{pgfscope}%
\begin{pgfscope}%
\definecolor{textcolor}{rgb}{0.839216,0.152941,0.156863}%
\pgfsetstrokecolor{textcolor}%
\pgfsetfillcolor{textcolor}%
\pgftext[x=0.369998in, y=1.921906in, left, base]{\color{textcolor}{\sffamily\fontsize{8.000000}{9.600000}\selectfont\catcode`\^=\active\def^{\ifmmode\sp\else\^{}\fi}\catcode`\%=\active\def
\end{pgfscope}%
\begin{pgfscope}%
\pgfsetbuttcap%
\pgfsetroundjoin%
\definecolor{currentfill}{rgb}{0.000000,0.000000,0.000000}%
\pgfsetfillcolor{currentfill}%
\pgfsetlinewidth{0.803000pt}%
\definecolor{currentstroke}{rgb}{0.000000,0.000000,0.000000}%
\pgfsetstrokecolor{currentstroke}%
\pgfsetdash{}{0pt}%
\pgfsys@defobject{currentmarker}{\pgfqpoint{-0.048611in}{0.000000in}}{\pgfqpoint{-0.000000in}{0.000000in}}{%
\pgfpathmoveto{\pgfqpoint{-0.000000in}{0.000000in}}%
\pgfpathlineto{\pgfqpoint{-0.048611in}{0.000000in}}%
\pgfusepath{stroke,fill}%
}%
\begin{pgfscope}%
\pgfsys@transformshift{0.614444in}{2.185243in}%
\pgfsys@useobject{currentmarker}{}%
\end{pgfscope}%
\end{pgfscope}%
\begin{pgfscope}%
\definecolor{textcolor}{rgb}{0.839216,0.152941,0.156863}%
\pgfsetstrokecolor{textcolor}%
\pgfsetfillcolor{textcolor}%
\pgftext[x=0.369998in, y=2.146663in, left, base]{\color{textcolor}{\sffamily\fontsize{8.000000}{9.600000}\selectfont\catcode`\^=\active\def^{\ifmmode\sp\else\^{}\fi}\catcode`\%=\active\def
\end{pgfscope}%
\begin{pgfscope}%
\pgfsetbuttcap%
\pgfsetroundjoin%
\definecolor{currentfill}{rgb}{0.000000,0.000000,0.000000}%
\pgfsetfillcolor{currentfill}%
\pgfsetlinewidth{0.803000pt}%
\definecolor{currentstroke}{rgb}{0.000000,0.000000,0.000000}%
\pgfsetstrokecolor{currentstroke}%
\pgfsetdash{}{0pt}%
\pgfsys@defobject{currentmarker}{\pgfqpoint{-0.048611in}{0.000000in}}{\pgfqpoint{-0.000000in}{0.000000in}}{%
\pgfpathmoveto{\pgfqpoint{-0.000000in}{0.000000in}}%
\pgfpathlineto{\pgfqpoint{-0.048611in}{0.000000in}}%
\pgfusepath{stroke,fill}%
}%
\begin{pgfscope}%
\pgfsys@transformshift{0.614444in}{2.410000in}%
\pgfsys@useobject{currentmarker}{}%
\end{pgfscope}%
\end{pgfscope}%
\begin{pgfscope}%
\definecolor{textcolor}{rgb}{0.839216,0.152941,0.156863}%
\pgfsetstrokecolor{textcolor}%
\pgfsetfillcolor{textcolor}%
\pgftext[x=0.369998in, y=2.371420in, left, base]{\color{textcolor}{\sffamily\fontsize{8.000000}{9.600000}\selectfont\catcode`\^=\active\def^{\ifmmode\sp\else\^{}\fi}\catcode`\%=\active\def
\end{pgfscope}%
\begin{pgfscope}%
\definecolor{textcolor}{rgb}{0.000000,0.000000,0.000000}%
\pgfsetstrokecolor{textcolor}%
\pgfsetfillcolor{textcolor}%
\pgftext[x=0.314443in,y=1.510972in,,bottom,rotate=90.000000]{\color{textcolor}{\sffamily\fontsize{10.000000}{12.000000}\selectfont\catcode`\^=\active\def^{\ifmmode\sp\else\^{}\fi}\catcode`\%=\active\def
\end{pgfscope}%
\begin{pgfscope}%
\pgfpathrectangle{\pgfqpoint{0.614444in}{0.611944in}}{\pgfqpoint{3.971111in}{1.798056in}}%
\pgfusepath{clip}%
\pgfsetrectcap%
\pgfsetroundjoin%
\pgfsetlinewidth{1.405250pt}%
\definecolor{currentstroke}{rgb}{0.839216,0.152941,0.156863}%
\pgfsetstrokecolor{currentstroke}%
\pgfsetdash{}{0pt}%
\pgfpathmoveto{\pgfqpoint{0.898095in}{1.689834in}}%
\pgfpathlineto{\pgfqpoint{1.465397in}{1.458447in}}%
\pgfpathlineto{\pgfqpoint{2.032698in}{1.608697in}}%
\pgfpathlineto{\pgfqpoint{2.600000in}{1.579253in}}%
\pgfpathlineto{\pgfqpoint{3.167302in}{1.625980in}}%
\pgfpathlineto{\pgfqpoint{3.734603in}{1.734515in}}%
\pgfpathlineto{\pgfqpoint{4.301905in}{1.686620in}}%
\pgfusepath{stroke}%
\end{pgfscope}%
\begin{pgfscope}%
\pgfpathrectangle{\pgfqpoint{0.614444in}{0.611944in}}{\pgfqpoint{3.971111in}{1.798056in}}%
\pgfusepath{clip}%
\pgfsetbuttcap%
\pgfsetroundjoin%
\definecolor{currentfill}{rgb}{0.839216,0.152941,0.156863}%
\pgfsetfillcolor{currentfill}%
\pgfsetlinewidth{0.000000pt}%
\definecolor{currentstroke}{rgb}{0.839216,0.152941,0.156863}%
\pgfsetstrokecolor{currentstroke}%
\pgfsetdash{}{0pt}%
\pgfsys@defobject{currentmarker}{\pgfqpoint{-0.038889in}{-0.038889in}}{\pgfqpoint{0.038889in}{0.038889in}}{%
\pgfpathmoveto{\pgfqpoint{0.000000in}{-0.038889in}}%
\pgfpathcurveto{\pgfqpoint{0.010313in}{-0.038889in}}{\pgfqpoint{0.020206in}{-0.034791in}}{\pgfqpoint{0.027499in}{-0.027499in}}%
\pgfpathcurveto{\pgfqpoint{0.034791in}{-0.020206in}}{\pgfqpoint{0.038889in}{-0.010313in}}{\pgfqpoint{0.038889in}{0.000000in}}%
\pgfpathcurveto{\pgfqpoint{0.038889in}{0.010313in}}{\pgfqpoint{0.034791in}{0.020206in}}{\pgfqpoint{0.027499in}{0.027499in}}%
\pgfpathcurveto{\pgfqpoint{0.020206in}{0.034791in}}{\pgfqpoint{0.010313in}{0.038889in}}{\pgfqpoint{0.000000in}{0.038889in}}%
\pgfpathcurveto{\pgfqpoint{-0.010313in}{0.038889in}}{\pgfqpoint{-0.020206in}{0.034791in}}{\pgfqpoint{-0.027499in}{0.027499in}}%
\pgfpathcurveto{\pgfqpoint{-0.034791in}{0.020206in}}{\pgfqpoint{-0.038889in}{0.010313in}}{\pgfqpoint{-0.038889in}{0.000000in}}%
\pgfpathcurveto{\pgfqpoint{-0.038889in}{-0.010313in}}{\pgfqpoint{-0.034791in}{-0.020206in}}{\pgfqpoint{-0.027499in}{-0.027499in}}%
\pgfpathcurveto{\pgfqpoint{-0.020206in}{-0.034791in}}{\pgfqpoint{-0.010313in}{-0.038889in}}{\pgfqpoint{0.000000in}{-0.038889in}}%
\pgfpathlineto{\pgfqpoint{0.000000in}{-0.038889in}}%
\pgfpathclose%
\pgfusepath{fill}%
}%
\begin{pgfscope}%
\pgfsys@transformshift{0.898095in}{1.689834in}%
\pgfsys@useobject{currentmarker}{}%
\end{pgfscope}%
\begin{pgfscope}%
\pgfsys@transformshift{1.465397in}{1.458447in}%
\pgfsys@useobject{currentmarker}{}%
\end{pgfscope}%
\begin{pgfscope}%
\pgfsys@transformshift{2.032698in}{1.608697in}%
\pgfsys@useobject{currentmarker}{}%
\end{pgfscope}%
\begin{pgfscope}%
\pgfsys@transformshift{2.600000in}{1.579253in}%
\pgfsys@useobject{currentmarker}{}%
\end{pgfscope}%
\begin{pgfscope}%
\pgfsys@transformshift{3.167302in}{1.625980in}%
\pgfsys@useobject{currentmarker}{}%
\end{pgfscope}%
\begin{pgfscope}%
\pgfsys@transformshift{3.734603in}{1.734515in}%
\pgfsys@useobject{currentmarker}{}%
\end{pgfscope}%
\begin{pgfscope}%
\pgfsys@transformshift{4.301905in}{1.686620in}%
\pgfsys@useobject{currentmarker}{}%
\end{pgfscope}%
\end{pgfscope}%
\begin{pgfscope}%
\pgfsetrectcap%
\pgfsetmiterjoin%
\pgfsetlinewidth{0.803000pt}%
\definecolor{currentstroke}{rgb}{0.000000,0.000000,0.000000}%
\pgfsetstrokecolor{currentstroke}%
\pgfsetdash{}{0pt}%
\pgfpathmoveto{\pgfqpoint{0.614444in}{0.611944in}}%
\pgfpathlineto{\pgfqpoint{0.614444in}{2.410000in}}%
\pgfusepath{stroke}%
\end{pgfscope}%
\begin{pgfscope}%
\pgfsetrectcap%
\pgfsetmiterjoin%
\pgfsetlinewidth{0.803000pt}%
\definecolor{currentstroke}{rgb}{0.000000,0.000000,0.000000}%
\pgfsetstrokecolor{currentstroke}%
\pgfsetdash{}{0pt}%
\pgfpathmoveto{\pgfqpoint{4.585556in}{0.611944in}}%
\pgfpathlineto{\pgfqpoint{4.585556in}{2.410000in}}%
\pgfusepath{stroke}%
\end{pgfscope}%
\begin{pgfscope}%
\pgfsetrectcap%
\pgfsetmiterjoin%
\pgfsetlinewidth{0.803000pt}%
\definecolor{currentstroke}{rgb}{0.000000,0.000000,0.000000}%
\pgfsetstrokecolor{currentstroke}%
\pgfsetdash{}{0pt}%
\pgfpathmoveto{\pgfqpoint{0.614444in}{0.611944in}}%
\pgfpathlineto{\pgfqpoint{4.585556in}{0.611944in}}%
\pgfusepath{stroke}%
\end{pgfscope}%
\begin{pgfscope}%
\pgfsetrectcap%
\pgfsetmiterjoin%
\pgfsetlinewidth{0.803000pt}%
\definecolor{currentstroke}{rgb}{0.000000,0.000000,0.000000}%
\pgfsetstrokecolor{currentstroke}%
\pgfsetdash{}{0pt}%
\pgfpathmoveto{\pgfqpoint{0.614444in}{2.410000in}}%
\pgfpathlineto{\pgfqpoint{4.585556in}{2.410000in}}%
\pgfusepath{stroke}%
\end{pgfscope}%
\begin{pgfscope}%
\pgfsetbuttcap%
\pgfsetroundjoin%
\definecolor{currentfill}{rgb}{0.000000,0.000000,0.000000}%
\pgfsetfillcolor{currentfill}%
\pgfsetlinewidth{0.803000pt}%
\definecolor{currentstroke}{rgb}{0.000000,0.000000,0.000000}%
\pgfsetstrokecolor{currentstroke}%
\pgfsetdash{}{0pt}%
\pgfsys@defobject{currentmarker}{\pgfqpoint{0.000000in}{0.000000in}}{\pgfqpoint{0.048611in}{0.000000in}}{%
\pgfpathmoveto{\pgfqpoint{0.000000in}{0.000000in}}%
\pgfpathlineto{\pgfqpoint{0.048611in}{0.000000in}}%
\pgfusepath{stroke,fill}%
}%
\begin{pgfscope}%
\pgfsys@transformshift{4.585556in}{0.667728in}%
\pgfsys@useobject{currentmarker}{}%
\end{pgfscope}%
\end{pgfscope}%
\begin{pgfscope}%
\definecolor{textcolor}{rgb}{0.172549,0.627451,0.172549}%
\pgfsetstrokecolor{textcolor}%
\pgfsetfillcolor{textcolor}%
\pgftext[x=4.711944in, y=0.629148in, left, base]{\color{textcolor}{\sffamily\fontsize{8.000000}{9.600000}\selectfont\catcode`\^=\active\def^{\ifmmode\sp\else\^{}\fi}\catcode`\%=\active\def
\end{pgfscope}%
\begin{pgfscope}%
\pgfsetbuttcap%
\pgfsetroundjoin%
\definecolor{currentfill}{rgb}{0.000000,0.000000,0.000000}%
\pgfsetfillcolor{currentfill}%
\pgfsetlinewidth{0.803000pt}%
\definecolor{currentstroke}{rgb}{0.000000,0.000000,0.000000}%
\pgfsetstrokecolor{currentstroke}%
\pgfsetdash{}{0pt}%
\pgfsys@defobject{currentmarker}{\pgfqpoint{0.000000in}{0.000000in}}{\pgfqpoint{0.048611in}{0.000000in}}{%
\pgfpathmoveto{\pgfqpoint{0.000000in}{0.000000in}}%
\pgfpathlineto{\pgfqpoint{0.048611in}{0.000000in}}%
\pgfusepath{stroke,fill}%
}%
\begin{pgfscope}%
\pgfsys@transformshift{4.585556in}{0.927188in}%
\pgfsys@useobject{currentmarker}{}%
\end{pgfscope}%
\end{pgfscope}%
\begin{pgfscope}%
\definecolor{textcolor}{rgb}{0.172549,0.627451,0.172549}%
\pgfsetstrokecolor{textcolor}%
\pgfsetfillcolor{textcolor}%
\pgftext[x=4.711944in, y=0.888608in, left, base]{\color{textcolor}{\sffamily\fontsize{8.000000}{9.600000}\selectfont\catcode`\^=\active\def^{\ifmmode\sp\else\^{}\fi}\catcode`\%=\active\def
\end{pgfscope}%
\begin{pgfscope}%
\pgfsetbuttcap%
\pgfsetroundjoin%
\definecolor{currentfill}{rgb}{0.000000,0.000000,0.000000}%
\pgfsetfillcolor{currentfill}%
\pgfsetlinewidth{0.803000pt}%
\definecolor{currentstroke}{rgb}{0.000000,0.000000,0.000000}%
\pgfsetstrokecolor{currentstroke}%
\pgfsetdash{}{0pt}%
\pgfsys@defobject{currentmarker}{\pgfqpoint{0.000000in}{0.000000in}}{\pgfqpoint{0.048611in}{0.000000in}}{%
\pgfpathmoveto{\pgfqpoint{0.000000in}{0.000000in}}%
\pgfpathlineto{\pgfqpoint{0.048611in}{0.000000in}}%
\pgfusepath{stroke,fill}%
}%
\begin{pgfscope}%
\pgfsys@transformshift{4.585556in}{1.186648in}%
\pgfsys@useobject{currentmarker}{}%
\end{pgfscope}%
\end{pgfscope}%
\begin{pgfscope}%
\definecolor{textcolor}{rgb}{0.172549,0.627451,0.172549}%
\pgfsetstrokecolor{textcolor}%
\pgfsetfillcolor{textcolor}%
\pgftext[x=4.711944in, y=1.148067in, left, base]{\color{textcolor}{\sffamily\fontsize{8.000000}{9.600000}\selectfont\catcode`\^=\active\def^{\ifmmode\sp\else\^{}\fi}\catcode`\%=\active\def
\end{pgfscope}%
\begin{pgfscope}%
\pgfsetbuttcap%
\pgfsetroundjoin%
\definecolor{currentfill}{rgb}{0.000000,0.000000,0.000000}%
\pgfsetfillcolor{currentfill}%
\pgfsetlinewidth{0.803000pt}%
\definecolor{currentstroke}{rgb}{0.000000,0.000000,0.000000}%
\pgfsetstrokecolor{currentstroke}%
\pgfsetdash{}{0pt}%
\pgfsys@defobject{currentmarker}{\pgfqpoint{0.000000in}{0.000000in}}{\pgfqpoint{0.048611in}{0.000000in}}{%
\pgfpathmoveto{\pgfqpoint{0.000000in}{0.000000in}}%
\pgfpathlineto{\pgfqpoint{0.048611in}{0.000000in}}%
\pgfusepath{stroke,fill}%
}%
\begin{pgfscope}%
\pgfsys@transformshift{4.585556in}{1.446107in}%
\pgfsys@useobject{currentmarker}{}%
\end{pgfscope}%
\end{pgfscope}%
\begin{pgfscope}%
\definecolor{textcolor}{rgb}{0.172549,0.627451,0.172549}%
\pgfsetstrokecolor{textcolor}%
\pgfsetfillcolor{textcolor}%
\pgftext[x=4.711944in, y=1.407527in, left, base]{\color{textcolor}{\sffamily\fontsize{8.000000}{9.600000}\selectfont\catcode`\^=\active\def^{\ifmmode\sp\else\^{}\fi}\catcode`\%=\active\def
\end{pgfscope}%
\begin{pgfscope}%
\pgfsetbuttcap%
\pgfsetroundjoin%
\definecolor{currentfill}{rgb}{0.000000,0.000000,0.000000}%
\pgfsetfillcolor{currentfill}%
\pgfsetlinewidth{0.803000pt}%
\definecolor{currentstroke}{rgb}{0.000000,0.000000,0.000000}%
\pgfsetstrokecolor{currentstroke}%
\pgfsetdash{}{0pt}%
\pgfsys@defobject{currentmarker}{\pgfqpoint{0.000000in}{0.000000in}}{\pgfqpoint{0.048611in}{0.000000in}}{%
\pgfpathmoveto{\pgfqpoint{0.000000in}{0.000000in}}%
\pgfpathlineto{\pgfqpoint{0.048611in}{0.000000in}}%
\pgfusepath{stroke,fill}%
}%
\begin{pgfscope}%
\pgfsys@transformshift{4.585556in}{1.705567in}%
\pgfsys@useobject{currentmarker}{}%
\end{pgfscope}%
\end{pgfscope}%
\begin{pgfscope}%
\definecolor{textcolor}{rgb}{0.172549,0.627451,0.172549}%
\pgfsetstrokecolor{textcolor}%
\pgfsetfillcolor{textcolor}%
\pgftext[x=4.711944in, y=1.666987in, left, base]{\color{textcolor}{\sffamily\fontsize{8.000000}{9.600000}\selectfont\catcode`\^=\active\def^{\ifmmode\sp\else\^{}\fi}\catcode`\%=\active\def
\end{pgfscope}%
\begin{pgfscope}%
\pgfsetbuttcap%
\pgfsetroundjoin%
\definecolor{currentfill}{rgb}{0.000000,0.000000,0.000000}%
\pgfsetfillcolor{currentfill}%
\pgfsetlinewidth{0.803000pt}%
\definecolor{currentstroke}{rgb}{0.000000,0.000000,0.000000}%
\pgfsetstrokecolor{currentstroke}%
\pgfsetdash{}{0pt}%
\pgfsys@defobject{currentmarker}{\pgfqpoint{0.000000in}{0.000000in}}{\pgfqpoint{0.048611in}{0.000000in}}{%
\pgfpathmoveto{\pgfqpoint{0.000000in}{0.000000in}}%
\pgfpathlineto{\pgfqpoint{0.048611in}{0.000000in}}%
\pgfusepath{stroke,fill}%
}%
\begin{pgfscope}%
\pgfsys@transformshift{4.585556in}{1.965027in}%
\pgfsys@useobject{currentmarker}{}%
\end{pgfscope}%
\end{pgfscope}%
\begin{pgfscope}%
\definecolor{textcolor}{rgb}{0.172549,0.627451,0.172549}%
\pgfsetstrokecolor{textcolor}%
\pgfsetfillcolor{textcolor}%
\pgftext[x=4.711944in, y=1.926446in, left, base]{\color{textcolor}{\sffamily\fontsize{8.000000}{9.600000}\selectfont\catcode`\^=\active\def^{\ifmmode\sp\else\^{}\fi}\catcode`\%=\active\def
\end{pgfscope}%
\begin{pgfscope}%
\pgfsetbuttcap%
\pgfsetroundjoin%
\definecolor{currentfill}{rgb}{0.000000,0.000000,0.000000}%
\pgfsetfillcolor{currentfill}%
\pgfsetlinewidth{0.803000pt}%
\definecolor{currentstroke}{rgb}{0.000000,0.000000,0.000000}%
\pgfsetstrokecolor{currentstroke}%
\pgfsetdash{}{0pt}%
\pgfsys@defobject{currentmarker}{\pgfqpoint{0.000000in}{0.000000in}}{\pgfqpoint{0.048611in}{0.000000in}}{%
\pgfpathmoveto{\pgfqpoint{0.000000in}{0.000000in}}%
\pgfpathlineto{\pgfqpoint{0.048611in}{0.000000in}}%
\pgfusepath{stroke,fill}%
}%
\begin{pgfscope}%
\pgfsys@transformshift{4.585556in}{2.224486in}%
\pgfsys@useobject{currentmarker}{}%
\end{pgfscope}%
\end{pgfscope}%
\begin{pgfscope}%
\definecolor{textcolor}{rgb}{0.172549,0.627451,0.172549}%
\pgfsetstrokecolor{textcolor}%
\pgfsetfillcolor{textcolor}%
\pgftext[x=4.711944in, y=2.185906in, left, base]{\color{textcolor}{\sffamily\fontsize{8.000000}{9.600000}\selectfont\catcode`\^=\active\def^{\ifmmode\sp\else\^{}\fi}\catcode`\%=\active\def
\end{pgfscope}%
\begin{pgfscope}%
\definecolor{textcolor}{rgb}{0.000000,0.000000,0.000000}%
\pgfsetstrokecolor{textcolor}%
\pgfsetfillcolor{textcolor}%
\pgftext[x=4.885557in,y=1.510972in,,top,rotate=90.000000]{\color{textcolor}{\sffamily\fontsize{10.000000}{12.000000}\selectfont\catcode`\^=\active\def^{\ifmmode\sp\else\^{}\fi}\catcode`\%=\active\def
\end{pgfscope}%
\begin{pgfscope}%
\pgfpathrectangle{\pgfqpoint{0.614444in}{0.611944in}}{\pgfqpoint{3.971111in}{1.798056in}}%
\pgfusepath{clip}%
\pgfsetrectcap%
\pgfsetroundjoin%
\pgfsetlinewidth{1.405250pt}%
\definecolor{currentstroke}{rgb}{0.172549,0.627451,0.172549}%
\pgfsetstrokecolor{currentstroke}%
\pgfsetdash{}{0pt}%
\pgfpathmoveto{\pgfqpoint{0.898095in}{0.693674in}}%
\pgfpathlineto{\pgfqpoint{1.465397in}{0.771512in}}%
\pgfpathlineto{\pgfqpoint{2.032698in}{0.875296in}}%
\pgfpathlineto{\pgfqpoint{2.600000in}{1.082864in}}%
\pgfpathlineto{\pgfqpoint{3.167302in}{1.368269in}}%
\pgfpathlineto{\pgfqpoint{3.734603in}{1.783405in}}%
\pgfpathlineto{\pgfqpoint{4.301905in}{2.328270in}}%
\pgfusepath{stroke}%
\end{pgfscope}%
\begin{pgfscope}%
\pgfpathrectangle{\pgfqpoint{0.614444in}{0.611944in}}{\pgfqpoint{3.971111in}{1.798056in}}%
\pgfusepath{clip}%
\pgfsetbuttcap%
\pgfsetmiterjoin%
\definecolor{currentfill}{rgb}{0.172549,0.627451,0.172549}%
\pgfsetfillcolor{currentfill}%
\pgfsetlinewidth{0.000000pt}%
\definecolor{currentstroke}{rgb}{0.172549,0.627451,0.172549}%
\pgfsetstrokecolor{currentstroke}%
\pgfsetdash{}{0pt}%
\pgfsys@defobject{currentmarker}{\pgfqpoint{-0.032998in}{-0.054997in}}{\pgfqpoint{0.032998in}{0.054997in}}{%
\pgfpathmoveto{\pgfqpoint{0.000000in}{-0.054997in}}%
\pgfpathlineto{\pgfqpoint{0.032998in}{0.000000in}}%
\pgfpathlineto{\pgfqpoint{0.000000in}{0.054997in}}%
\pgfpathlineto{\pgfqpoint{-0.032998in}{0.000000in}}%
\pgfpathlineto{\pgfqpoint{0.000000in}{-0.054997in}}%
\pgfpathclose%
\pgfusepath{fill}%
}%
\begin{pgfscope}%
\pgfsys@transformshift{0.898095in}{0.693674in}%
\pgfsys@useobject{currentmarker}{}%
\end{pgfscope}%
\begin{pgfscope}%
\pgfsys@transformshift{1.465397in}{0.771512in}%
\pgfsys@useobject{currentmarker}{}%
\end{pgfscope}%
\begin{pgfscope}%
\pgfsys@transformshift{2.032698in}{0.875296in}%
\pgfsys@useobject{currentmarker}{}%
\end{pgfscope}%
\begin{pgfscope}%
\pgfsys@transformshift{2.600000in}{1.082864in}%
\pgfsys@useobject{currentmarker}{}%
\end{pgfscope}%
\begin{pgfscope}%
\pgfsys@transformshift{3.167302in}{1.368269in}%
\pgfsys@useobject{currentmarker}{}%
\end{pgfscope}%
\begin{pgfscope}%
\pgfsys@transformshift{3.734603in}{1.783405in}%
\pgfsys@useobject{currentmarker}{}%
\end{pgfscope}%
\begin{pgfscope}%
\pgfsys@transformshift{4.301905in}{2.328270in}%
\pgfsys@useobject{currentmarker}{}%
\end{pgfscope}%
\end{pgfscope}%
\begin{pgfscope}%
\pgfsetrectcap%
\pgfsetmiterjoin%
\pgfsetlinewidth{0.803000pt}%
\definecolor{currentstroke}{rgb}{0.000000,0.000000,0.000000}%
\pgfsetstrokecolor{currentstroke}%
\pgfsetdash{}{0pt}%
\pgfpathmoveto{\pgfqpoint{0.614444in}{0.611944in}}%
\pgfpathlineto{\pgfqpoint{0.614444in}{2.410000in}}%
\pgfusepath{stroke}%
\end{pgfscope}%
\begin{pgfscope}%
\pgfsetrectcap%
\pgfsetmiterjoin%
\pgfsetlinewidth{0.803000pt}%
\definecolor{currentstroke}{rgb}{0.000000,0.000000,0.000000}%
\pgfsetstrokecolor{currentstroke}%
\pgfsetdash{}{0pt}%
\pgfpathmoveto{\pgfqpoint{4.585556in}{0.611944in}}%
\pgfpathlineto{\pgfqpoint{4.585556in}{2.410000in}}%
\pgfusepath{stroke}%
\end{pgfscope}%
\begin{pgfscope}%
\pgfsetrectcap%
\pgfsetmiterjoin%
\pgfsetlinewidth{0.803000pt}%
\definecolor{currentstroke}{rgb}{0.000000,0.000000,0.000000}%
\pgfsetstrokecolor{currentstroke}%
\pgfsetdash{}{0pt}%
\pgfpathmoveto{\pgfqpoint{0.614444in}{0.611944in}}%
\pgfpathlineto{\pgfqpoint{4.585556in}{0.611944in}}%
\pgfusepath{stroke}%
\end{pgfscope}%
\begin{pgfscope}%
\pgfsetrectcap%
\pgfsetmiterjoin%
\pgfsetlinewidth{0.803000pt}%
\definecolor{currentstroke}{rgb}{0.000000,0.000000,0.000000}%
\pgfsetstrokecolor{currentstroke}%
\pgfsetdash{}{0pt}%
\pgfpathmoveto{\pgfqpoint{0.614444in}{2.410000in}}%
\pgfpathlineto{\pgfqpoint{4.585556in}{2.410000in}}%
\pgfusepath{stroke}%
\end{pgfscope}%
\end{pgfpicture}%
\makeatother%
\endgroup%

%% file: refs-bib-path.tex
\bibliography{refs.bib}